\documentclass[11pt,a4paper,twoside]{article}
\usepackage{titling}
\usepackage{titletoc}
\usepackage{url}

\usepackage[utf8]{inputenc}
\usepackage[T1]{fontenc}
\usepackage{amsfonts}
\usepackage{amssymb}
\usepackage{stmaryrd}
\usepackage[all]{xy}
\usepackage{makeidx}

\usepackage[english,french]{babel}\frenchsetup{GlobalLayoutFrench=false}
\usepackage{xspace}

\newcounter{bidon}

\newcommand{\rdb}{\refstepcounter{bidon}}

\newcommand \sibrouillon[1]{}

\newcommand \sibil[1] {#1}
\newcommand \sinotbil[1] {}

\usepackage[bookmarksopen=false,breaklinks=true,%
      backref=page,pagebackref=true,plainpages=false,%
      hyperindex=true,pdfstartview=FitH,colorlinks=true,%
      pdfpagelabels=true,linkcolor=blue,%
      citecolor=red,urlcolor=red,
      ]%
   {hyperref}

\begin{document}

\thispagestyle{empty}
~ 
\vspace{1cm}

This paper appeared in  \emph{Theoretical Computer Science}, {\bf 392}, no 1--3, p.\ 113--127, (2008)

\smallskip A companion paper appeared as: Barhoumi, Sami and Lombardi, Henri. An algorithm for the {T}raverso-{S}wan theorem on seminormal rings,
  \emph{Journal of Algebra}, {\bf 320}, no 4, p.\ 1531--1542, (2008)

\smallskip A constructive approach to the  Traverso-Swan theorem is also given in the book \emph{Commutative algebra: constructive methods. Finite projective modules}, Lombardi, Henri and Quitt{\'e}, Claude, Springer series Algebra and applications, {\bf 20}, (2015). A slightly extended and revised version is in \url{https://arxiv.org/abs/1605.04832}.

\bigskip \noindent In this file you find the English version starting on the page  numbered \pageref{beginenglish}.

\bigskip  \noindent  \centerline{{\Large \bf Seminormal Rings (following Thierry Coquand)}}

\medskip  \centerline{{\bf Abstract}}

\smallskip The Traverso-Swan theorem says that a reduced ring A is seminormal if and only if 
the natural morphism   $\mathrm{Pic}(\mathbf{A}) \rightarrow \mathrm{Pic}(\mathbf{A}[X])$ 
 is an isomorphism. We give here all the details needed to understand the elementary constructive proof for this result  given by Thierry Coquand 
in the paper: On seminormality. J. Algebra 305, no. 1, 577-584, (2006).

\bigskip \noindent  
Then the French version begins on the page numbered \pageref{beginfrench}.

\bigskip \noindent   \centerline{{\Large \bf Anneaux seminormaux
(d'apr\`es Thierry Coquand)}}

\medskip  \centerline{{\bf Résumé}}

\smallskip  Le théorème de Traverso-Swan affirme qu'un anneau 
réduit $\mathbf{A}$ est seminormal si, et seulement si, l'homomorphisme naturel 
$\mathrm{Pic} \,\mathbf{A}\rightarrow \mathrm{Pic} \,\mathbf{A}[X]$ est un isomorphisme. Nous exposons ici une démonstration constructive élémentaire de ce résultat qui a été donnée par Thierry Coquand 
dans: On seminormality. J. Algebra 305, n°1, 577-584, (2006).

\bigskip  La lectrice ou le lecteur sera sans doute surpris de l'alternance des sexes ainsi que de l'orthographe du mot 'corolaire', avec d'autres innovations auxquelles elle n'est pas habituée. En fait, nous avons essayé de suivre au plus près les préconisations de l'orthographe nouvelle recommandée, telle qu'elle est enseignée aujourd'hui dans les écoles en France.  

\bigskip\noindent   {\large \bf Authors}  

\smallskip \noindent Henri Lombardi, Université Marie et Louis Pasteur, F-25030 Besançon Cedex, France, \\
email: {\tt henri.lombardi@umlp.fr}

\smallskip \noindent Claude Quitté, Laboratoire de Math\'ematiques,
SP2MI, Boulevard 3, Teleport 2, BP 179, F-86960 FUTUROSCOPE Cedex,
FRANCE, 
 email: {\tt claude.quitte@orange.fr}

\normalsize
\newpage
\thispagestyle{empty}

~

\pagestyle{headings}
\setcounter{page}{0}
\renewcommand\thepage{E\arabic{page}}

\newcommand \sitout[1]{}

\begingroup

\selectlanguage{english}



\newtheorem{theorem}{Theorem}[section]
\newtheorem{thdef}[theorem]{Theorem and definition}
\newtheorem{lemma}[theorem]{Lemma}
\newtheorem{corollary}[theorem]{Corollary}
\newtheorem{conjecture}[theorem]{Conjecture}
\newtheorem{definition}[theorem]{Definition}
\newtheorem{definitions}[theorem]{Definitions}
\newtheorem{proposition}[theorem]{Proposition}
\newtheorem{propdef}[theorem]{Proposition and definition}
\newtheorem{remark}[theorem]{Remark}
\newtheorem{fact}[theorem]{Fact}
\newtheorem{example}[theorem]{Example}
\newtheorem{notation}[theorem]{Notation}
\newtheorem{notadefi}[theorem]{Notation and definition} 
\newtheorem{comment}[theorem]{Comment}
\newtheorem{convention}[theorem]{Convention}
\newtheorem{problem}[theorem]{Problem}
\newtheorem{question}[theorem]{Question}


\newcommand {\junk}[1]{}

\newcommand \rem  {\noindent \emph{Remark. }  }
\newcommand \rems {\noindent \emph{Remarks. }  }
\newcommand \comm {\noindent \emph{Commentk. }  }
\newcommand \exl  {\noindent \emph{Example. }}

\newenvironment{proof}{
\trivlist \item[\hskip \labelsep{\it Proof.}]\hskip 1pt }
{\hfill \mbox{$\Box$}
\endtrivlist}
\makeatletter

\newenvironment{prooF}{
\trivlist \item[\hskip \labelsep{\it Proof.}]\hskip 1pt}
{\hfill \mbox{\textsf{Not so bad, indeed?}}
\endtrivlist}
\makeatletter

\newenvironment{Proof}[1]{
\trivlist \item[\hskip \labelsep{\it #1}]\hskip 0pt\\}
{\hfill \mbox{$\Box$}
\endtrivlist}
\makeatletter

\newcommand \eop {\hbox{}\nobreak\hfill
\vrule width 1.4mm height 1.4mm depth 0mm \par \goodbreak 
\smallskip}

\newcommand\aec{\texttt{\`A \'ecrire.}}

\DeclareRobustCommand{\guig}{\mbox{{\usefont{U}{lasy}%
{\if b\expandafter\@car\f@series\@nil b\else m\fi}{n}%
\char40\kern-0.20em\char40}~}}
\DeclareRobustCommand{\guid}{\mbox{~\usefont{U}{lasy}%
{\if b\expandafter\@car\f@series\@nil b\else m\fi}{n}%
\char41\kern-0.20em\char41}}
\newcommand\gui[1]{``{#1}''}

\newcommand \aigu{\mathaccent19}    
\renewcommand \grave{\mathaccent18}    
\newcommand \noi {\noindent}
\newcommand \sms {\smallskip}
\newcommand \sni {\sms\noi}
\newcommand \ms {\medskip}
\newcommand \mni {\ms\noi}
\newcommand \bs {\bigskip}
\newcommand \bni {\bs\noi}
\newcommand \hs {\qquad}
\newcommand \alb {\allowbreak}
\newcommand \ce {\centerline}

\newcommand\sta{^\star}
\newcommand\ista{_\star}
\newcommand \bu {{$\bullet$}}
\newcommand \bl {^\bullet}
\renewcommand \cir {^\circ}
\newcommand\equidef{\buildrel{{\rm def}}\over{\quad\Longleftrightarrow\quad}} 


\newcommand\mapright[1]{\smash{\mathop{\longrightarrow}\limits^{#1}}} 
\newcommand\maprightto[1]{\smash{\mathop{\longmapsto}\limits^{#1}}} 
\newcommand\mapdown[1]{\downarrow\rlap{$\vcenter{\hbox{$\scriptstyle 
#1$}}$}}
\newcommand\eqdf[1]{\buildrel{#1}\over{\; =\;}}
\newcommand\equivdf[1]{\buildrel{#1}\over \longleftrightarrow}
\newcommand\vers[1]{\buildrel{#1}\over \longrightarrow }
\newcommand\impdef[1]{\buildrel{#1}\over \Longrightarrow} 
\newcommand\tra[1]{{\,^{\rm t}\!#1}}
\newcommand\gen[1]{\left\langle{#1}\right\rangle} 
\newcommand\so[1]{\left\{{#1}\right\}} 
\newcommand \sur[1]{\!\left/#1\right.}
\newcommand \aqo[2]{#1\sur{\gen{#2}}}

\newcommand \snic[1] {\sni\centerline{$#1$}\sms}
\newcommand \snif[3] 
{\vspace{#1}\noindent\centerline{$#3$}\vspace{#2}}

\newcommand \ov[1] {\overline{#1}}

\newcommand \wh[1] {\widehat{#1} }
\newcommand \wi[1] {\widetilde{#1} }

\newcommand \cmatrix[1]{\left[\matrix{#1}\right]}  
\newcommand \bloc[4]{\left[\matrix{#1 & #2 \cr #3 & #4}\right]}

\newcommand \CC{\mathbb {C}} 
\newcommand \NN{\mathbb {N}} 
\newcommand \ZZ{\mathbb {Z}} 

\newcommand \gx{{\underline{x}}}
\newcommand \gy{{\underline{y}}}
\newcommand \gz{{\underline{z}}}
\newcommand \gu{{\underline{u}}}
\newcommand \gt{{\underline{t}}}
\newcommand \gc{{\underline{c}}}

\newcommand \gk{{\bf k}}
\newcommand \gA{\mathbf{A}}
\newcommand \AX {\gA[X]}
\newcommand \gB{\mathbf{B}}
\newcommand \gC{\mathbf{C}}
\newcommand \gK{\mathbf{K}}
\newcommand \gL{\mathbf{L}}
\newcommand \gM{\mathbf{M}}
\newcommand \gT{\mathbf{T}}
\newcommand \gV{\mathbf{V}}
\newcommand \gZ{\mathbf{Z}}

\newcommand \GL{\mathsf{GL}}

\newcommand \cC {{\cal C}}
\newcommand \cD {{\cal D}}
\newcommand \cI {{\cal I}}
\newcommand \cJ {{\cal J}}
\newcommand \cF {{\cal F}}
\newcommand \cN {{\cal N}}
\newcommand \cP {{\cal P}}
\newcommand \cQ {{\cal Q}}
\newcommand \cM {{\cal M}}
\newcommand \cT {{\cal T}}
\newcommand \cL {{\cal L}}
\newcommand \cR {{\cal R}}
\newcommand \cS {{\cal S}}

\newcommand \rI {\mathrm{I}}
\newcommand \rD {\mathrm{D}}
\newcommand \rG {\mathrm{G}}
\newcommand \rV {\mathrm{V}}
\newcommand \rJ {\mathrm{J}}
\newcommand \rH {\mathrm{H}}
\newcommand \rK {\mathrm{K}}
\newcommand \rN {\mathrm{N}}
\newcommand \rP {\mathrm{P}}
\newcommand \rL{\mathrm{L}}
\newcommand \rPr{\mathrm{Pr}}

\newcommand\fa{\mathfrak{a}}
\newcommand\fb{\mathfrak{b}}
\newcommand\fc{\mathfrak{c}}
\newcommand\fA{\mathfrak{A}}
\newcommand\fB{\mathfrak{B}}
\newcommand\fD{\mathfrak{D}}
\newcommand\fI{\mathfrak{i}}
\newcommand\fII{\mathfrak{I}}
\newcommand\fj{\mathfrak{j}}
\newcommand\fJ{\mathfrak{J}}
\newcommand\fF{\mathfrak{F}}
\newcommand\ff{\mathfrak{f}}
\newcommand\ffg{\mathfrak{g}}
\newcommand\fG{\mathfrak{G}}
\newcommand\fh{\mathfrak{h}}
\newcommand\fl{\mathfrak{l}}
\newcommand\fm{\mathfrak{m}}
\newcommand\fM{\mathfrak{M}}
\newcommand\fp{\mathfrak{p}}
\newcommand\fP{\mathfrak{P}}
\newcommand\fq{\mathfrak{q}}
\newcommand\fV{\mathfrak{V}}

\newcommand \Zg {{\Z[G]}}

\newcommand \vu {\vee} 
\newcommand \vi {\wedge} 
\newcommand \Vu {\bigvee}
\newcommand \Vi {\bigwedge}
\newcommand \im {\rightarrow} 
\newcommand \da {\,\downarrow\!}
\newcommand \ua {\,\uparrow\!}
\newcommand \vd {\,\vdash\,}
\newcommand \vdu[1] {\,\vdash^{#1}\,}
\newcommand \vdb[1] {\,\vdash_{#1}\,}


\newcommand \uX {\underline{X}}
\newcommand \Ared {\gA\red}
\newcommand \AuX {\gA[\uX]}
\newcommand \Xm {X_1,\ldots,X_m}
\newcommand \AXn {\gA[\Xn]}
\newcommand \AXm {\gA[\Xm]}
\newcommand \xn {x_1,\ldots,x_n}
\newcommand \yn {y_1,\ldots,y_n}
\newcommand \cq {c_1,\ldots,c_q}

\newcommand \Pf {{{\cal P}_{\mathrm{f}}}}

\newcommand \Ann {\mathrm{Ann}}
\newcommand \Diag {\mathrm{Diag}}
\newcommand \Hom {\mathrm{Hom}}
\renewcommand \det {\mathrm{det}}
\renewcommand \dim {\mathrm{dim}}
\newcommand \Coker {\mathrm{Coker}}
\newcommand \Frac {\mathrm{Frac}}
\newcommand \Ker {\mathrm{Ker}}
\renewcommand \Im {\mathrm{Im}}
\newcommand \Id {\mathrm{Id}}
\newcommand \I {\mathrm{I}}
\newcommand \In {\I_n}
\newcommand \Mat {\mathrm{Mat}}
\newcommand \Rad {\mathrm{Rad}}
\newcommand \Tr {\mathrm{Tr}}
\newcommand \mod {\,\mathrm{mod}\,}
\newcommand \pgcd {\mathrm{gcd}}
\newcommand \red {_\mathrm{red}}

\newcommand \GKO {\mathsf{GK}_{0}}
\newcommand \Pic {\mathsf{Pic}}
\newcommand \Spec {\mathsf{Spec}}

\newcommand \num {{n$^{\mathrm{o}}$}}
 
\newcommand \recu {induction }
\newcommand \recuz {induction}
\newcommand \hdr {induction hypo\-thesis }
\newcommand \hdrz {induction hypo\-thesis}
\newcommand \cad {i.e., }
\newcommand \Cad {I.e., }
\newcommand \ssi {if and only if }
\newcommand \cnes {necessary and sufficient condition }
\newcommand \spdg {w.l.o.g. }
\newcommand \Propeq {The following are equivalent:}
\newcommand \propeq {the following are equivalent:}
\newcommand \disept {17$^{th}$ Hilbert's problem 
}


\newcommand \Amo {$\gA$-mo\-du\-le }
\newcommand \Amos {$\gA$-mo\-du\-les }
\newcommand \Amoz {$\gA$-mo\-du\-le}
\newcommand \Amosz {$\gA$-mo\-du\-les}

\newcommand \Bmo {$\gB$-mo\-du\-le }
\newcommand \Bmos {$\gB$-mo\-du\-les }
\newcommand \Bmoz {$\gB$-mo\-du\-le}
\newcommand \Bmosz {$\gB$-mo\-du\-les}

\newcommand \Ali {$\gA$-\lin map }
\newcommand \Alis {$\gA$-\lin maps }
\newcommand \Aliz {$\gA$-\lin map}
\newcommand \Alisz {$\gA$-\lin maps}

\newcommand \alg {algebra }
\newcommand \algs {algebras }
\newcommand \algz {algebra}
\newcommand \algsz {algebras}

\newcommand \algo{algorithm }
\newcommand \algos{algorithms }
\newcommand \algoz{algorithm}
\newcommand \algosz{algorithms}

\newcommand \ali {\lin map }
\newcommand \alis {\lin maps }
\newcommand \aliz {\lin map}
\newcommand \alisz {\lin maps}

\newcommand \auto {automorphism }
\newcommand \autos {automorphisms }
\newcommand \autosz {automorphisms}
\newcommand \autoz {automorphism}

\newcommand \cac {algebraically closed field }
\newcommand \cacz {algebraically closed field}
\newcommand \cacs {algebraically closed fields }
\newcommand \cacsz {algebraically closed fields}

\newcommand \carn{characterisation }  
\newcommand \carns{characterisations }

\newcommand \coe {coefficient }
\newcommand \coes {coefficients }
\newcommand \coez {coefficient}
\newcommand \coesz {coefficients}

\newcommand \coli {\lin combination }
\newcommand \colis {\lin combinations }
\newcommand \coliz {\lin combination}
\newcommand \colisz {\lin combinations}

\newcommand \com {comaximal }
\newcommand \comz {comaximal}

\newcommand \ddk {Krull dimension }
\newcommand \ddkz {Krull dimension}
\newcommand \ddi {with \ddk $\leq$~}

\newcommand \dfn{definition }  
\newcommand \dfns{definitions }  
\newcommand \dfnz{definition}  
\newcommand \dfnsz{definitions}  

\newcommand \egt{equality } 
\newcommand \egts{equalities } 
\newcommand \egtz{equality} 
\newcommand \egtsz{equalities}

\newcommand \elr{elementary }  
\newcommand \elrz{elementary}

\newcommand \elt{element }  
\newcommand \elts{elements }  
\newcommand \eltz{element}  
\newcommand \eltsz{elements}

\newcommand \egmt {also }

\newcommand \entrel {entailment relation }
\newcommand \entrelz {entailment relation}
\newcommand \entrels {entailment relations }
\newcommand \entrelsz  {entailment relations}

\newcommand \evc{vector space } 
\newcommand \evcs{vector spaces } 
\newcommand \evcz{vector space} 
\newcommand \evcsz{vector spaces} 

\newcommand \fdi{strongly discrete } 
\newcommand \fdiz{strongly discrete} 

\newcommand \gtr{generator }  
\newcommand \gtrs{generators }  
\newcommand \gtrz{generator}  
\newcommand \gtrsz{generators}

\newcommand \homo {homomorphism }
\newcommand \homoz {homomorphism}
\newcommand \homos {homomorphisms }
\newcommand \homosz {homomorphisms}

\newcommand \id {ideal }
\newcommand \ids {ideals }
\newcommand \idz {ideal}
\newcommand \idsz {ideals}

\newcommand \idd {ideal  d\'eter\-minan\-tiel }
\newcommand \idds {id\'eaux d\'eter\-minan\-tiels }
\newcommand \iddz {ideal  d\'eter\-minan\-tiel}
\newcommand \iddsz {id\'eaux d\'eter\-minan\-tiels}

\newcommand \idema {maximal ideal }
\newcommand \idemas {maximal ideals }
\newcommand \idemaz {maximal ideal}
\newcommand \idemasz {maximal ideals}

\newcommand \idep {prime ideal }
\newcommand \ideps {prime ideals }
\newcommand \idepz {prime ideal}
\newcommand \idepsz {prime ideals}

\newcommand \idemi {minimal prime }
\newcommand \idemis {minimal primes }
\newcommand \idemiz {minimal prime}
\newcommand \idemisz {minimal primes}

\newcommand \idf {Fitting ideal }
\newcommand \idfs {Fitting ideals }
\newcommand \idfz {Fitting ideal}
\newcommand \idfsz {Fitting ideals}

\newcommand \idm {idempotent }
\newcommand \idms {idempotents }
\newcommand \idmz {idempotent}
\newcommand \idmsz {idempotents}

\newcommand \idme \idm
\newcommand \idmes \idm
\newcommand \idmez \idmz
\newcommand \idmesz \idmz

\newcommand \itf {\tf ideal }
\newcommand \itfs {\tf ideals }
\newcommand \itfz {\tf ideal}
\newcommand \itfsz {\tf ideals}

\newcommand \iso {isomorphism }
\newcommand \isos {isomorphisms }
\newcommand \isosz {isomorphisms}
\newcommand \isoz {isomorphism}

\newcommand \lin {linear }
\newcommand \linz {linear}

\newcommand \mo {monoid }
\newcommand \mos {monoids }
\newcommand \moz {monoid}
\newcommand \mosz {monoids}
\newcommand \moco {\com \mos}
\newcommand \mocoz {\com \mosz}

\newcommand \mpr {\pro module }
\newcommand \mprs {\pro modules }
\newcommand \mprz {\pro module}
\newcommand \mprsz {\pro modules}

\newcommand \mptf {\ptf module }
\newcommand \mptfs {\ptf modules }
\newcommand \mptfz {\ptf module }
\newcommand \mptfsz {\ptf module}

\newcommand \mrc {projective module of constant rank }
\newcommand \mrcs {projective modules of constant rank }
\newcommand \mrcz {projective module of constant rank}
\newcommand \mrcsz {projective modules of constant rank}

\newcommand \ndz {non zero divisor }
\newcommand \ndzs {non zero divisors }

\newcommand \noe {Noetherian }
\newcommand \noco {\noe\coh }
\newcommand \noez {Noetherian}
\newcommand \nocoz {\noe\cohz}

\newcommand \nst {Nullstellensatz }
\newcommand \nstz {Nullstellensatz}
\newcommand \nsts {Nullstellens\"atze }
\newcommand \nstsz {Nullstellens\"atze}

\newcommand \odz {Zariski open set }

\newcommand \oqc {\qc open }
\newcommand \oqcs {\qc opens }
\newcommand \oqcz {\qc open}
\newcommand \oqcsz {\qc opens}

\newcommand \pa {saturated pair }
\newcommand \pas {saturated pairs }
\newcommand \paz {saturated pair}
\newcommand \pasz {saturated pairs}

\newcommand \pb{problem }  
\newcommand \pbs{problems }  
\newcommand \pbz{problem}  
\newcommand \pbsz{problems} 

\newcommand \pf {finitely presented }
\newcommand \pfz {finitely presented}

\newcommand \pn {presentation }
\newcommand \pns {presentations }
\newcommand \pnz {presentation}
\newcommand \pnsz {presentations}

\newcommand \pol {polynomial }
\newcommand \pols {polynomials }
\newcommand \polz {polynomial}
\newcommand \polsz {polynomials}

\newcommand \polcar {characteristic \pol }
\newcommand \polcarz {characteristic \polz}

\newcommand \prc {rank constant \pro }

\newcommand \prn {projection }
\newcommand \prns {projections }
\newcommand \prnz {projection}
\newcommand \prnsz {projections}

\newcommand \proi {potential prime }
\newcommand \prois {potential primes }
\newcommand \proiz {potential prime}
\newcommand \proisz {potential primes}

\newcommand \proc {potential chain }
\newcommand \procs {potential chains }
\newcommand \Procs {potential chains }
\newcommand \procz {potential chain}
\newcommand \procsz {potential chains}

\newcommand \proel {elementary \proc}
\newcommand \proels {elementary \procs}
\newcommand \proelz {elementary \procz}
\newcommand \proelsz {elementary \proczs}
\newcommand \proelo {\proel of length }
\newcommand \proelos {\proels of length }

\newcommand \prolo {\proc of length }
\newcommand \prolos {\procs of length }

\newcommand \pro {projective }
\newcommand \proz {projective}

\newcommand \ptf {\tf \pro}
\newcommand \ptfs {\ptf}
\newcommand \ptfz {\tf \proz}
\newcommand \ptfsz {\ptfz}
\newcommand \ptfe {\tf \pro}

\newcommand \qc {compact }
\newcommand \qcs {compact }
\newcommand \qcz {compact}
\newcommand \qcsz {compact}

\newcommand \qi {quasi integral }
\newcommand \qis {quasi integral }
\newcommand \qisz {quasi integral}
\newcommand \qiz {quasi integral}

\newcommand \rdl {linear dependence relation }
\newcommand \rdls {linear dependence relations }
\newcommand \rdlsz {linear dependence relation}
\newcommand \rdlz {linear dependence relations}

\newcommand \rdi {integral dependence relation }
\newcommand \rdis {integral dependence relations }
\newcommand \rdiz {integral dependence relation}
\newcommand \rdisz {integral dependence relations}

\newcommand \sad {dynamic algebraic structure }
\newcommand \sads {dynamic algebraic structures }
\newcommand \sadz {dynamic algebraic structure}
\newcommand \sadsz {dynamic algebraic structures}

\newcommand \sdz {without \dvz}
\newcommand \sdzz {without \dvzz}

\newcommand \sli {\lin \sys }
\newcommand \slis {\lin \syss }
\newcommand \sliz {\lin \sysz}
\newcommand \slisz {\lin \syssz}

\newcommand \sys {system }
\newcommand \syss {systems }
\newcommand \sysz {system}
\newcommand \syssz {systems}

\newcommand \Tho {Theorem }
\newcommand \tho {theorem }
\newcommand \thos {theorems }
\newcommand \thoz {theorem}
\newcommand \thosz {theorems}

\newcommand \tf {finitely generated }
\newcommand \tfz {finitely generated}

\newcommand \sfio {basic system of orthogonal \idms }
\newcommand \sfios {basic systems of orthogonal \idms }
\newcommand \sfioz {basic system of orthogonal \idmsz}
\newcommand \sfiosz {basic systems of orthogonal \idmsz}

\newcommand \vfn {verification }
\newcommand \vfns {verifications }
\newcommand \vfnz {verification}
\newcommand \vfnsz {verifications}

\newcommand \zed {zero-dimensional }
\newcommand \zedz {zero-dimensional}
\newcommand \zede {zero-dimensional }
\newcommand \zedez {zero-dimensional}
\newcommand \zeds {zero-dimensional }
\newcommand \zedsz {zero-dimensional}
\newcommand \zedes {zero-dimensional }
\newcommand \zedesz {zero-dimensional}


\newcommand \cof {constructive }
\newcommand \cofs {constructive}
\newcommand \cofz {constructive }
\newcommand \cofsz {constructive}

\newcommand \cov {constructive }
\newcommand \covz {constructive}
\newcommand \covsz {constructives}
\newcommand \covs {constructives }

\newcommand \coma {\cov \maths}
\newcommand \comaz {\cov \mathsz}
\newcommand \clama {classical \maths}
\newcommand \clamaz {classical \mathsz}

\newcommand \maths {mathematics }
\newcommand \mathsz {mathematics}

\newcommand \LLPO{{\bf LLPO}}

\newcommand \prco {\cov proof }
\newcommand \prcos {\cov proofs }
\newcommand \prcoz {\cov proof}
\newcommand \prcosz {\cov proofs}

\renewcommand \cot {constructively }
\newcommand \cotz {constructively}

\newcommand \pte {law of excluded middle }
\newcommand \ptez {law of excluded middle}

\newcommand \tcg {compactness theorem }
\newcommand \tcgz {compactness theorem}
\newcommand \acgz {G\"odel's completness axiom}
\newcommand \Tcgi {The \tcg implies the following result. }

\def\proofname{\emph{Proof}}

\title{ Seminormal Rings
\\
(following Thierry Coquand)}
\author{
Henri Lombardi (\thanks{~Équipe de Mathématiques, UMR CNRS 6623,
UFR des Sciences and Techniques, Université Marie et Louis Pasteur,
25030 BESANCON cedex, FRANCE,
email: {\tt henri.lombardi@umlp.fr}.}~), Claude Quitt\'e (\thanks {~Laboratoire de Math\'ematiques,
SP2MI, Boulevard 3, Teleport 2, BP 179, 86960 FUTUROSCOPE Cedex,
FRANCE, email: {\tt claude.quitte@orange.fr}}~) }
\date{November 2007}

\maketitle

\rdb
\label{beginenglish}

\startcontents[english]

\begin{abstract}
The Traverso-Swan \tho says that a reduced ring $\gA$ is seminormal \ssi 
the natural \homo  $\Pic \,\gA\to\Pic 
\,\AX $ is an \iso (\cite{Tra,Swan}). We give here all the details
needed to understand the \elr  \prco 
for this result  given by Thierry Coquand 
in~\cite{coq}.

This example is typical of a new  \cov method.
The final proof is simpler than the initial classical one.
More important: the classical argument by absurdum using \gui{an abstract ideal object} is deciphered with a general technique
based on the following idea: purely ideal objects constructed using LEM and Choice may be replaced by concrete objects that are ``finite approximations'' of these ideal objects. 
\end{abstract}



\small

\newpage

\setcounter{tocdepth}{4}

\printcontents[english]{}{1}{}
\normalsize

\newpage

\section {Introduction}
\markboth{Seminormalrings}{Introduction}

%
%
%
%
%
%
%
%
%
%

The Traverso-Swan \tho says that a reduced ring $\gA$ is seminormal \ssi 
the natural \homo  $\Pic \,\gA\to\Pic 
\,\AX $ is an \iso (\cite{Tra,Swan}). We give here all the details
needed to understand the \elr  \prco 
for this result  given by Thierry Coquand 
in~\cite{coq}.

First, we have to give a classical proof (using LEM and Choice)
as elementary as possible. After this first simplification we have
to remove remaining non constructive arguments. 
Here it is a proof by absurdum based on the introduction of an abstract ideal object, which is a minimal prime.

The deciphering of this non constructive argument is based on
the so called \gui{dynamical method}.

\smallskip This example is paradigmatic of a new general \cov method
inspired by the following semantic: purely ideal objects constructed using LEM and Choice may be replaced by 
\emph{concrete objects that are finite approximations of these ideal objects}.

An important step, where this method was introduced in Computer Algebra
from an efficiency point of view, was the computer algebra system D5 
\cite{D5}: here we see that it is possible to compute inside the algebraic closure of an arbitrary computable field, contrarily to the well known fact
that such an algebraic closure cannot exist constructively as a static object. So D5 told us that, from a constructive point of view, the algebraic closure of an arbitrary computable field does exist, not as a static object, but as a dynamical one.

In the paper \cite{CLR} the dynamical method is explained on the example
of abstract proofs, via model theory, of results similar to the Hilbert \nstz. 
Here ideal abstract objects are the models of a coherent first order theory.
These models have to exist in classical mathematics: this is the compactness
theorem in (classical) model theory. 
When the classical proof is deciphered in a constructive one, each one
of these models
is replaced by  \gui{a finite amount of  information concerning it}.

In the papers \cite{cl,CLR2}, chains of prime ideals that are used in \clama
in order to define the Krull dimension are replaced by finite sequences
of elements of the ring. In this way we obtain an \elr \dfn
of the Krull dimension, without using any prime ideal.
The Krull dimension of usual rings matches the elementary definition in a constructive way. 
So  theorems in commutative algebra that have in their hypothesis a 
bound on the Krull dimension can now be reread in a constructive way,
and for several important ones a constructive proof, much more precise
than the classical one, has been found.  
E.g., Serre's \gui{splitting-off}, 
\gui{stable range} and \gui{cancellation}  \thos of Bass, and  Forster-Swan \thoz.
Moreover the \cov versions (\cite{Coq3,clq}) are an improvement of the most sophisticated classical versions of these \thos given by R. Heitmann in his remarkable \gui{non\noez}  1984 paper~\cite{Hei84}.

Finally let us mention that in \cite{Y1}, I. Yengui has shown how to reread in a dynamical way classical proofs that use maximal ideals.

\smallskip In the example given in the present paper, we get a proof which is simpler and more elegant than the classical ones. 
But the most important fact is that the classical argument \gui{by absurdum
and using a purely ideal object} is deciphered by following the general method
we have sketched. 
The localisation at a generic \idemi $\fp$ is replaced by a tree computation
where we try to make invertible all elements that appear in the
computational proof. The tree comes from the fact that in the classical reasoning
one uses an argument saying \gui{any \elt $x$ of the ring is either inside or outside the generic \idemi $\fp$ we consider}. Since the prime is minimal, a priori $x$ have to be outside of 
$\fp$. We have to use the branch ``$x$ inside $\fp$''  only in the case where the computation shows that $0$ becomes invertible if $x$ is outside $\fp$.

\smallskip We shall explain first in Section \ref{sec2} what happens with an integral ring.
We give the proof of the general case in the Annex.

\newpage
\section{Preliminaries} \label{sec-prelim}

$\gA$, $\gB$, $\gC$ are commutative rings. 
Used without more precision an \gui{\homoz} is always a ring \homoz.

\subsubsection*{Seminormal rings} 
\addcontentsline{toc}{subsection}{Seminormal rings} 
\markboth{Seminormalrings}{Preliminaries}

An integral ring $\gA$ is said to be \emph{seminormal} if whenever 
$b^2=c^3\neq 0$ the \elt $a=b/c$ of the fraction field is in $\gA$. 
Remark that 
 $a^3=b$ and $a^2=c$.

An arbitrary ring $\gA$  is said to be \emph{seminormal} if whenever $b^2=c^3$, there exists $a\in\gA$ such that  $a^3=b$ and $a^2=c$.
This implies $\gA$  is reduced: if $b^2=0$  then $b^2=0^3$, 
so we get an $a\in\gA$ with $a^3=b$ and $a^2=0$, thus $b=0$.

In  a ring if $x^2=y^2$ and  $x^3=y^3$ then $(x-y)^3=0$.
So:

\begin{fact} 
\label{factRed1} 
In  a reduced ring  $x^2=y^2$ and  $x^3=y^3$ imply $x=y$.
\end{fact}

Consequently the \elt $a$ here upon is always unique. Moreover  
$\Ann\,b=\Ann\,c=\Ann\,a$.

\subsubsection*{The category of \ptf \Amos} 
\addcontentsline{toc}{subsection}{The category of \ptf \Amos}

A \mptf is a module $M$ which is isomorphic to  a direct summand of a 
 finite rank free module: $M\oplus M'\simeq\gA^m$. Equivalenty, it is  a module isomorphic to  the image of an 
\idm matrix.

An \Ali $\psi:M\to N$ between  \mptfs with  $M\oplus M'\simeq\gA^m$ and 
 $N\oplus N'\simeq\gA^n$ can be given by the linear map 
$\wi{\psi}:\gA^m \to \gA^n$ defined by $\wi{\psi}(x\oplus 
x')=\psi(x)$.

In other words the category of \mptfs over $\gA$ is 
equivalent to  the category whose objects are \idm matrices with \coes in $\gA$, a morphism
from $P$ to $Q$ being a matrix $H$  such that  
$QH=H=HP$. In particular the identity of $P$ is represented by~$P$.

\begin{fact} 
\label{factcatmptf} 
If \mptfs $M$ and $N$ are represented by \idm matrices
$P=(p_{i,j})_{i,j\in I}\in \gA^{I\times I}$ and 
$Q=(q_{k,\ell})_{k,\ell\in J}\in \gA^{J\times J}$, then: 
\begin{enumerate}
\item 
The direct sum $M\oplus N$ is represented by 
 $\Diag(P,Q)=\bloc{P}{0}{0}{Q}$.
\item The tensor product
 $M\otimes N$ is represented by the  Kronecker product
 $$P
 \otimes Q=(r_{(i,k),(j,\ell)})_{(i,k),(j,\ell)\in I\times J}, \;
 \mathrm{where} \;
r_{(i,k),(j,\ell)}=p_{i,j}q_{k,\ell}.
$$
\item \label{item3factcatmptf} $M$ and $N$ are isomorphic  \ssi  matrices $\Diag(P,0_n)$ and  
$\Diag(0_m,Q)$ are similar.
\end{enumerate}
\end{fact}
%
\begin{proof}
\emph{3.} Remark that the projection on $M$ in $M\oplus M'\oplus \gA^n$ is 
represented by the matrix   $\Diag(P,0_n)$  and  
the projection on $N$ in $\gA^m\oplus N\oplus N'$ is 
represented by the matrix    $\Diag(0_m,Q)$. Writing 
$\gA^m\oplus\gA^n$ as
$M\oplus M'\oplus N \oplus N'$ we see that the two projections are 
conjugate by the \auto exchanging $M$ and~$N$. 
\end{proof}
%

\subsubsection*{Rank of a \mptf} 
\addcontentsline{toc}{subsection}{Rank of a \mptf}
If $\varphi :M\to M$ is an endomorphism of the \ptf \Amo  $M$ image of
the \idm matrix $P\in\gA^{n\times n}$ and if $H\in\gA^{n\times n}$ represents $\varphi$  (with $H=PH=HP$),
we let $N=\Ker\,P$. So $M\oplus N =\gA^n$ and we can define
the \emph{determinant} of $\varphi$ by 
\[\det(\varphi)=\det(\varphi \oplus \Id_N)=\det(H+(\In-P)).\]

Let $\mu_X$ be the multiplication by $X$ inside the $\AX $-module $M[X]$. 
This module, extended of~$M$ from~$\gA$, is  represented 
by the matrix $P$ viewed in $\gA[X]^{n\times n}$. 
Then $\det(\mu_X)=\mathrm{R}_M(X)=r(X)$ is a \pol satisfying $r(XY)=r(X)r(Y)$ and $r(1)=1$. 
In other words its \coes are a \sfioz. The module is said \emph{of rank $k$} if  $r(X)=X^k$.

A direct computation shows the following fact.

\begin{fact} 
\label{factprc1} 
A matrix $P=(p_{i,j})$ is an \idm matrix whose image is a \mrc 1 \ssi the following 
properties are satisfied
\begin{itemize}
\item  $\Vi^2\,P=0$, \cad all $2\times 2$ minors are null,
\item  $\Tr\,P=\sum_{i}p_{ii}=1$.
\end{itemize}
\end{fact}

\subsubsection*{When the image of a projection matrix is free} 
\addcontentsline{toc}{subsection}{When the image of a projection matrix is free}

If $P\in\gA^{n\times n}$ is a \idm matrix whose image  is 
free
of rank $r$, its kernel is not always free, so the matrix 
is not always  similar 
to the standard matrix of projection \[\I_{n,r}=\Diag(\I_{r},0_{n-r})= 
\bloc{\I_{r}}{0} {0}{0_{n-r}}.\]

Let us give a simple \carn for the fact that the image of an \idm matrix is free.

\begin{proposition} 
\label{propImProjLib} 
Let $P\in\gA^{n\times n}$. The matrix 
$P$ is \idm and its image is free of rang $r$  \ssi
there exist two matrices $X\in\gA^{n\times r}$ and  $Y\in\gA^{r\times 
n}$ such that
$YX=\I_r$ and $P=XY$. Moreover,
\begin{enumerate}
\item  $\Im\,P=\Im\,X\simeq \Im\,Y$.
\item  For any matrices $X',Y'$ with same  formats as $X$  and $Y$ 
and such that $P=X'Y'$, there exists
a unique  matrix   $U\in\GL_r(\gA)$ such that $X'=XU$ and $Y=UY'$. In fact
$U=YX'$, $U^{-1}=Y'X$, $Y'X'=\I_r$ and the columns of $X'$ form a 
basis of~$\Im\,P$.
\end{enumerate}
Another possible \carn is that the matrix $\Diag(P,0_r)$ 
is similar  to the standard \prn matrix $\I_{n+r,r}$.
\end{proposition}
\begin{proof}
Assume that $\Im\,P$ is free of rank $r$.
We take for the columns of $X$ a basis of $\Im\,P$. So, there exists 
a unique 
matrix $Y$ such that  $P=XY$. Since $PX=X$ (because $P^2=P$) one has 
$XYX=X$.
Since  $X$ is injective and $(\I_r-YX)X=0$ 
one has $\I_r=YX$.

\noindent Let us assume $YX=\I_r$ and $P=XY$. Thus 
\[P^2=XYXY=X\I_rY=XY=P\;\hbox{  and }\; PX=XYX=X .\] Hence $\Im\,P=\Im\,X$. Moreover the columns of $X$ are 
independent because $XZ=0$ implies $Z=YXZ=0$.  

\smallskip \noindent \emph{1.} The sequence $\gA^n\vers{\In-P}\gA^n\vers{Y}\gA^r$ is 
exact: indeed
$Y(\In-P)=0$ and if $YZ=0$ then $PZ=0$ thus $Z=(\In-P)Z$. So 
$\Im\,Y\simeq \gA^n\sur{\Ker\,Y}=\gA^n\sur{\Im(\In-P)\simeq \Im\,P}$.

\smallskip \noindent \emph{2.} If   $X'$ and $Y'$ have same formats as $X$ and $Y$,  and if $P=X'Y'$, let  $U=YX'$ and \hbox{$V=Y'X$}.
Thus $UV=YX'Y'X=YPX=YX=\I_r$; $X'V=X'Y'X=PX=X$, so $X'=XU$;
$UY'=YX'Y'=YP=Y$, so $Y'=VY$. Finally $Y'X'=VYXU=VU=\I_r$.

\sni Concerning the last \carn it is a simple application 
of Item \emph{\ref{item3factcatmptf}} in Fact~\ref{factcatmptf}.
\end{proof}

For \mrcs 1 we get the following.

\begin{lemma} 
\label{lempropImProjLib} 
 An \idm matrix  $P$ of rank $1$ has its image free \ssi there exist
  a column vector $x$ and a row vector $y$ such that  $yx=1$ 
and $xy=P$. Moreover  $x$ and $y$  are unique up to multiplication by a unit as soon as $xy=P$. 
\end{lemma}

     
\subsubsection*{The Grothendieck  semiring 
\texorpdfstring{$\GKO\,\gA$ and the Picard group $\Pic\,\gA$}{GKO(A) 
and  Pic(A)}} 
\addcontentsline{toc}{subsection}{$\GKO\,\gA$ and $\Pic\,\gA$}

$\GKO\,\gA$ is the set of \iso classes of \mptfs over $\gA$.
It is a semiring for laws $\oplus$ and 
$\otimes$. 

Since $\gA$ is assumed to be commutative, the subsemiring of  
$\GKO\,\gA$ generated by $1$ (the \iso class of $\gA$) is 
isomorphic to  $\NN$, except when $\gA$ is the trivial ring. 

Any \elt of $\GKO\,\gA$ can be represented by an 
\idm matrix with \coes in~$\gA$.

$\Pic\,\gA$ is the subset of $\GKO\,\gA$ whose \elts are \iso 
classes of \mrcs 1.
It is a group for $\otimes$. The \gui{inverse} of $M$ 
is its dual. If $M\simeq \Im\,P$ then $M\sta\simeq 
\Im\,\tra{P}$. In particular
if $P$ is an \idm matrix of rank 1, $P\otimes \tra{P}$ is an
\idm matrix whose image is a free module of rank 1.

This can be verified directly by applying Lemma~\ref{lempropImProjLib}. 

\subsubsection*{\texorpdfstring{$\Pic\,\gA$}{Pic(A)} and 
classes of invertible ideals} 
\addcontentsline{toc}{subsection}{$\Pic\,\gA$ and 
classes of invertible ideals}

An \id $\fa$ of $\gA$  is  \emph{invertible} if there exists an \id 
$\fb$ such that  $\fa\fb=a\gA$ where $a$ is a regular \eltz.
In this case there exist $\xn$ and $\yn$ in $\gA$ such that  
$\fa=\gen{\xn}$, $\fb=\gen{\yn}$ and $\sum_ix_iy_i=a$. Moreover 
for all $i,j$ there exists a unique  $m_{i,j}$ such that  
$y_ix_j=am_{i,j}$. 
One deduces that the matrix $(m_{i,j})$ is an \idm matrix  of rank 1, 
and its image is isomorphic to  $\fa$ as \Amoz.

Two invertible \ids  $\fa,\fb$ are isomorphic  as \Amos \ssi 
there exist  regular \elts $a,b$ such that  $a\fa=b\fb$. 
This allows to see the class group of $\gA$ (i.e., the group of  classes of invertible \idsz) as a subgroup of $\Pic\,\gA$. In most cases the  two 
groups are identical.

For example if $\gA$ is integral,
any matrix $(a_{i,j})$ which is \idm of rank $1$ has a regular \elt on
its diagonal and the \coes of the corresponding row generate an invertible \id 
isomorphic to  the image of the matrix.

\subsubsection*{Change of ring} 
\addcontentsline{toc}{subsection}{Change of ring}

Let $\rho$ be an \homo $\gA\to\gB$. The change of ring from 
$\gA$ to $\gB$ transforms a
\mptf $M$  over $\gA$ in a \mptf $\rho\ista(M)\simeq M\otimes_\gA\gB$ over $\gB$. 
Any \Bmo isomorphic to  such a module $\rho\ista(M)$ is said 
\gui{extended} from $\gA$.
For \idm matrices this amounts to consider the  matrix
after transformation by the \homoz~$\rho$.

This gives an \homo $\GKO\,\rho: \GKO\,\gA\to\GKO\,\gB$.
Whence the natural following \pbz: \gui{Is each \mptf over $\gB$ 
extended form a \mptf over $\gA$?}.
In other words: \gui{Is $\GKO\,\rho$ onto?}.

For example if $\gZ$ is the subring of $\gA$  generated by 
$1_\gA$,
we know that  $\gZ$-\mrcs  are free, and the question 
\gui{Are  \mrcs extended from $\gZ$?} is 
equivalent to
\gui{Are  \mrcs   free?}.

When $\gB=\AXm=\AuX$, one has the evaluation \homo in 0,  
$\gB\vers{\theta}\gA$, with $\theta\circ\rho=\Id_\gA$. 
This implies that the \ptf \Bmo $M=M(\uX)$ is extended from $\gA$
\ssi it is 
isomorphic to  $M(0)=\theta\ista(M)$. 

Concerning \prn matrices,
an \idm matrix  $P\in\gB^{n\times n}$ represents a module which is 
extended from $\gA$ \ssi its image is isomorphic to  the image of 
$P(0)$. 

If all  \ptf \Bmos are extended from $\gA$ then $P$ is similar to $P(0)$, 
but it may be easier to show only the isomorphim of the images. 

Concerning $\Pic$ one has  two group \homos 
$\Pic\,\gA\vers{\Pic\,\rho}\Pic\,\AuX\vers{\Pic\,\theta}\Pic\,\gA$
whose composition is the identity. The first one is injective, the 
second one surjective, and they are \isos \ssi the first one is surjective, 
\ssi the second one is injective.

The last property means that if a matrix $P(\uX)$  is \idm 
of rank 1 over $\AuX$ and if $\Im(P(0))$ is free, 
then  $\Im(P(\uX))$ is free.

In fact if  $\Im(P(0))$ is free, then the  bloc diagonal matrix 
$\Diag(P(0),0_1)$ is similar to a standard \prn matrix 
 $\I_{n+1,1}$. As $\Im(\Diag(P(\uX),0_1))$ is isomorphic to  
$\Im\,P(\uX)$, we get the following result.

\begin{lemma} 
\label{lemPicPic1} \Propeq
\begin{enumerate}
\item The natural \homo  $\Pic\,\gA\to\Pic\,\AuX$ is an \isoz.
\item If a matrix $P(\uX)\in\AuX^{n\times n}=(m_{i,j}(\uX))_{i,j\in \so{1,\ldots 
,n}}$  is \idm 
of rank 1 and if $P(0)=\I_{n,1}$, then  there exist 
$f_1,\ldots, f_n, g_1,\ldots, g_n\in\AuX$ such that  $m_{i,j}=f_ig_j$ 
for all $i,j$.
\end{enumerate}
\end{lemma}

\subsubsection*{Reducing problems to reduced rings: 
\texorpdfstring{$\GKO\,\Ared = \GKO\,\gA$}{GKO(Ared)=GK0(A)}} 
\addcontentsline{toc}{subsection}{Reducing problems to reduced rings}

We note $\Ared$ for  $\gA\sur{\sqrt{0}}$.
\begin{proposition} 
\label{propComparRed} 
The natural map $\GKO(\gA)\to\GKO(\Ared)$ is bijective. 
\begin{enumerate}
\item Injective:  this means that  if two \mptfs $E,F$  over 
$\gA$ are isomorphic  over $\Ared$, they are also isomorphic over $\gA$.
\item If
two \idm matrices $P,Q\in\gA^{n\times n}$  are conjugate  over 
$\Ared$, they are \egmt conjugate over $\gA$.
\item Surjective:  any \mptf over $\Ared$ comes from a \mptf 
over~$\gA.$ 
\end{enumerate}
\end{proposition}
\begin{proof}
\emph{2.} Let us note $\ov{x}$ the object $x$  viewed modulo $\sqrt{0}$. Let 
$C\in\gA^{n\times n}$ be a matrix such that  $\ov{C}\, \ov{P}\, {\ov{C}}^{-1}=\ov{Q}$. Since $\det\,C$ is invertible modulo $\sqrt{0}$, 
$\det\,C$ is invertible in $\gA$ and $C $ belongs to $\GL_n(\gA)$. Thus 
$\ov{Q}=\ov{C\,P\,C^{-1}}$. Replacing  $P$ 
by $C\,P\,C^{-1}$ we may assume $\ov{Q}=\ov{P}$ and $\ov C=\In$. 
Then the 
matrix $A=QP+(\In-Q)(\In-P)$ gives $AP=QP=QA$  and $\ov{A}=\In$: thus $A$ is invertible, $APA^{-1}=Q$ and $\ov A=\ov C$.

\smallskip \noindent \emph{1.} Two  residually isomorphic \mptfs  $E\simeq\Im\,P$ and 
$F\simeq\Im\,Q$ are images of residually conjugate matrices: 
$\Diag(P,0_m)$ and  $\Diag(0_n,Q)$ with $\Diag(\ov{P},0_m)$ similar  
to $\Diag(0_n,\ov{Q})$ (see Fact \ref{factcatmptf}). Thus we can apply Item \emph{1.}

\smallskip \noindent \emph{3.} Any  \mptf over $\Ared$ can be seen as the residual module of a
\mptf over $\gA$: apply Newton method. More precisely let $\fa$ be the \id generated by the \coes of $P^2-P$. If $\fa$ is contained in the nilradical of $\gA$, there exists~$k$ such that  $\fa^{2^k}=0$. On the other hand  if $Q=3P^2-2P^3$, then $Q\equiv P \mod \fa$ and $Q^2-Q$ is a multiple of 
$(P^2-P)^2$, thus $Q^2 - Q$ has its \coes in $\fa^2$. Iterating 
$k$ times the operation $P\leftarrow 3P^2-2P^3$ we get the result.
\end{proof}

\begin{corollary} 
\label{corpropComparRed} 
The canonical \homo  $\Pic\,\gA\to \Pic\,\AuX$ is an \iso  \ssi 
the canonical \homo  $\Pic\,\Ared\to \Pic\,\Ared[\uX]$ is an \isoz.
\end{corollary}

\begin{convention} 
\label{convPicPic} 
In the sequel we abbreviate the sentence \gui{the canonical \homo  
$\Pic\,\gA\to \Pic\,\AuX$ is an \isoz} and we write simply 
\gui{$\Pic\,\gA=\Pic\,\AuX$}.
\end{convention}

\subsubsection*{Invertible \elts of 
\texorpdfstring{$\gA[\protect\underline{X}]$}{A[X]}} 
\addcontentsline{toc}{subsection}{Invertible \elts of 
$\AX $}

\begin{lemma} 
\label{lemUnitRedX} 
If the ring $\gA$ is reduced, the group \homo  
$\gA^{\times}\to(\AuX)^{\times}$ is an \isoz.
In other words if $f(\uX)\in\AuX$ is invertible, then 
$f=f(0)\in\gA^{\times}$.
\end{lemma}

It is sufficient to consider $\AX $.
A direct computation shows that if $f(X)g(X)=1$ with $\deg(f)\leq m$, $m\geq 1$, then the \coe of degree $m$  of $f$  is nilpotent.
\subsubsection*{Kronecker's \tho} 
\addcontentsline{toc}{subsection}{Kronecker's \tho}

\begin{theorem} 
\label{thKro} 
Let $f,g\in\AuX$ and $h=fg$. Let $a$ be a \coe of $f$  and $b$ a \coe 
of~$g$, then $ab$ is integral over the subring of $\gA$ generated  
by the \coes of $h$. 
\end{theorem}

Using \gui{the Kronecker trick} (i.e., replace each variable 
$X_k$ with $T^{m^k}$ for an $m\gg 0$) reduces the problem to univariate \polsz.
For univariate \pols \prcos  are given in the literature
(cf. \cite{Ed,Hu}, and for a survey \cite{CDLQ}).

\section{Traverso-Swan \thoz, with integral rings} 
\label{sec2}
\markboth{Seminormalrings}{Traverso-Swan \thoz, with integral rings}
\subsubsection*{The condition is necessary: Schanuel example} 
\addcontentsline{toc}{subsection}{Schanuel example}

We show that if $\gA$ is reduced and $\Pic\,\gA=\Pic\,\AX $ then 
$\gA$ is seminormal. 
We use the \carn given in Lemma \ref{lempropImProjLib}.

Let $b,c$ be \elts in a  reduced ring $\gA$ with $b^2=c^3$. Let $\gB=\gA[a]=\gA+a\gA$ be 
a reduced ring containing $\gA$ with $a^3=b, \,a^2=c$.
Let $f_1=1+aX$, $f_2=cX^2=g_2$ and $g_1=(1-aX)(1+cX^2)$.
We have $f_1g_1+f_2g_2=1$, thus  the matrix $M(X)=(f_ig_j)_{1\leq i,j\leq2}$ is \idme of rank $1$.
Its \coes are in $\gA$ and  
$M(0)=\I_{2,1}$.
Thus its image is free over $\gB[X]$. If it is free over $\AX $ 
then there exist  $f'_i$'s and $g'_j$'s in $\AX $ with $f'_ig'_j=f_ig_j$.
By unicity $f'_i=uf_i$ with $u$ invertible in $\AX $. Since $\gA$
is reduced $u$ is invertible in $\gA$. Since $uf_1\in\AX$ we get $a\in\gA$.

\sni NB: we can take $\gB=\left(\aqo{\gA[T]}{T^2-c, T^3-b} 
\right)\red$, with $a=$ class of $T$. If some $a$ does exist in~$\gA$, we get $\gB\simeq\gA$.
\subsubsection*{Case of a gcd domain} 
\addcontentsline{toc}{subsection}{Case of a gcd domain}

Let us recall that a gcd domainn is an integral ring where
two arbitrary \elts have a gcd,
\cad a lower bound for the divisibility relation.
Also if $\gA$ is a gcd domain, then $\AuX$ is a gcd domain. 
\begin{lemma} 
\label{lemPicGcd} 
If $\gA$ is a gcd domain, $\Pic\,\gA=\so{1}$.
\end{lemma}
\rem Consequently  $\Pic\,\gA\to \Pic\,\AuX$ is an \isoz.
This works if  $\gA$  is a discrete field.

\begin{proof}
We use the \carn given in Lemma \ref{lempropImProjLib}. 
Let $P=(m_{i,j})$ be an \idm matrix of rank 1. 
Since $\sum_i m_{i,i}=1$ we may assume that $m_{1,1}$ is 
regular.
Let $f$ be the gcd  of the first row. We have $m_{1,j}=fg_j$
with the gcd of $g_j$'s equal to 1. Since $f$ is regular and 
$m_{1,1}m_{i,j}=m_{1,j}m_{i,1}$ we have $g_1m_{i,j}=m_{i,1}g_j$.
So $g_1$ divides all the $m_{i,1}g_j$ and also their gcd 
$m_{i,1}$.
Let us write $m_{i,1}=g_1f_i$. Since $g_1f_1=m_{1,1}=fg_1$ we get 
$f_1=f$. Finally the \egt $m_{1,1}m_{i,j}=m_{1,j}m_{i,1}$ gives 
$f_1g_1m_{i,j}=f_1g_jg_1f_i$ and $m_{i,j}=f_ig_j.$
\end{proof}

\subsubsection*{Case of an integral normal ring} 
\addcontentsline{toc}{subsection}{Case of an integral normal ring}
\begin{lemma} 
\label{lemIntegclos} 
If $\gA$ is integral and integrally closed, then 
$\Pic\,\gA=\Pic\,\AuX$.
\end{lemma}
\begin{proof}
We use the \carn given in Lemma \ref{lemPicPic1}. Let 
$P(\uX)=(m_{i,j}(\uX))_{i,j=1,\ldots ,n}$ be an \idm matrix of rank 
$1$  with
$P(0)=\I_{n,1}$. Let $\gK$ be the fraction field  of $\gA$.
On $\gK[\uX]$ the module  $\Im\,P(\uX)$ is free. Thus there exist  
$f=(f_1(\uX),\ldots ,f_n(\uX))$ and $g=(g_1(\uX),\ldots ,g_n(\uX))$ 
in
$\gK[\uX]^n$ such that  $m_{i,j}=f_ig_j$ for all $i,j$.
Moreover  since $f_1(0)g_1(0)=1$ and since we can modify $f$ and 
$g$ multiplying them by units, we can assume that 
$f_1(0)=g_1(0)=1$. Thus since $f_1g_j=m_{1,j}$ and using Kronecker's \thoz, the \coes of $g_j$'s are integral over the ring  generated  
by the \coes of $m_{1,j}$'s. In the same way the \coes of $f_i$'s are 
integral over the ring  generated by the \coes of $m_{i,1}$'s.
As $\gA$ is integrally closed the $f_i$'s and  
$g_j$'s are in $\gA[\uX]$. 
\end{proof}

\subsubsection*{Case of an integral seminormal ring} 
\addcontentsline{toc}{subsection}{Case of an integral seminormal ring}

Traverso \cite{Tra} has proved the \tho for \noe reduced ring
 (with some restrictions). For   proofs in the case of integral
 rings without \noe hypothesis see
 \cite{BC,Querre,GH}. 

\begin{theorem} 
\label{propIntSemin} 
If $\gA$ is integral and seminormal, then $\Pic\,\gA=\Pic\,\AuX$.
\end{theorem}
\begin{proof}
We start the proof as in Lemma \ref{lemIntegclos}.
There exist $f_1(\uX),\ldots ,f_n(\uX),g_1(\uX),\ldots ,g_n(\uX)$ in
$\gK[\uX]^n$ such that  $m_{i,j}=f_ig_j$ for all $i,j$. Moreover  
$f_1(0)=g_1(0)=1$. Let us call $\gB$ the subring of $\gK$ generated  
by $\gA$ and by the \coes of $f_i$'s and  $g_j$'s. Kronecker's \tho says that
 $\gB$ is a finite extension of 
$\gA$ (i.e., $\gB$ is a \tf \Amoz). 
Our aim is now to show that $\gA=\gB$.
Let us call  $\fa$ the \emph{conductor  of $\gA$ in $\gB$}, \cad
 $\so{x\in\gB\,|\,x\gB\subseteq\gA}$. It is an 
\id of $\gA$ and of $\gB$. Our aim is now to show that 
$\fa=\gen{1}$, \cad that $\gC=\gA\sur{\fa}$ is trivial. We need some lemmas.

\begin{lemma} 
\label{lemIntSemin1} 
If $\gA\subseteq\gB$, $\gA$ seminormal and $\gB$ reduced, then the  
conductor $\fa$ of $\gA$ in $\gB$ is a  radical \id of $\gB$.
\end{lemma}
\begin{Proof}{Proof of Lemma \ref{lemIntSemin1}.}
We have to show that if $u\in\gB$ and $u^2\in\fa$ then $u\in\fa$. 
Let  $c\in\gB$, we have to show that  $uc\in\gA$. We have 
$u^2c^2\in\gA$, and
 $u^3c^3=u^2(uc^3)\in\gA$ since $u^2\in\fa$. 
Since $(u^3c^3)^2=(u^2c^2)^3$ there exists $a\in\gA$ such that  
 $a^2=(uc)^2$ and $a^3=(uc)^3$. Since $\gB$ is reduced this implies 
$a=uc$, and thus $uc\in\gA$. 
\end{Proof}

\rem The \emph{seminormal closure} of a ring 
$\gA$ in a reduced overring  $\gB$ is obtained by starting
with $\gA$ and adding \elts $x$ of $\gB$  such that  $x^2$ and $x^3$ are in the 
previoulsly constructed ring. Fact \ref{factRed1} implies
that  $x$ is uniquely determined by  $x^2$ and $x^3$.  
So the previous proof can be seen as a proof of the following lemma.

\begin{lemma} 
\label{lemIntSemin1bis} 
Let $\gA\subseteq\gB$ be reduced rings, $\gA_{1}$ 
the seminorml closure of  $\gA$ in $\gB$, and  
$\fa$ the conductor  of $\gA_{1}$ in $\gB$. 
Then $\fa$ is a radical \id of
$\gB$.
\end{lemma}

\begin{lemma} 
\label{lemIntSemin2} 
Let $\gA\subseteq\gB=\gA[\cq]$ be reduced rings 
with $\gB$  finite over $\gA$. Let   
$\fa$ be the conductor  of $\gA$ in $\gB$. Assume that $\fa$ is a 
 radical \idz. Then  $\fa$ is 
equal to $\so{x\in\gA\,|\,xc_1,\ldots ,xc_q\in\gA}$.
\end{lemma}
\begin{Proof}{Proof of Lemma \ref{lemIntSemin2}.}
Indeed if $xc_i\in\gA$ then  $x^\ell c_i^\ell\in\gA$ for all 
$\ell$, and thus for $N$ big enough
$x^N y\in\gA$ for all $y\in\gB$, thus $x$ is in the radical of 
$\fa$ (if $d$ bounds the degrees of integral dependence equations
of the $c_i$'s over $\gA$, one can take $N=(d-1)q$).
\end{Proof}
\emph{End of the proof of \Tho \ref{propIntSemin}},
stated within \clamaz.\\
Let us assume by contradiction that $\fa\neq \gen{1}$. 
One has $\gC=\gA\sur{\fa}\subseteq \gB\sur{\fa}=\gC'$.
Let  $\fp$ be~a \idemi of 
$\gC$, $\fP$ the corresponding ideal of $\gA$,
$S=\gC\setminus\fp$  the complementary part. 
Since~$\fp$  
is~a \idemi  and since~$\gC$ is reduced 
$S^{-1}\gC=\gL$ is a field contained in the reduced ring  $S^{-1}\gC'=\gL'$.\\
If  $x$ is an object defined over $\gA$ let us call
 $\ov{x}$ what it becomes after the change of ring  $\gA\to\gL'$.
The module  $\ov{M}$ is defined by the matrix $\ov{P}$ whose \coes 
are in $\gL[\uX]$. Since $\gL$ is a field, $\Im\,\ov{P}$ is 
free over~$\gL[\uX]$. 
This implies, by unicity (Lemma~\ref{lempropImProjLib}) and  since $f_1(0)=g_1(0)=1$, that the \pols $\ov{f_i}$ and $\ov{g_j}$ are in $\gL[\uX]$ (if $u(X)\in\gL[\uX]$ is invertible and $u(0)=1$, then $u=1$).
This means that there exists $s\in \gA\setminus \fP$  such that the \pols 
$sf_i$ and $sg_j$ have their \coes in $\gA$.  
Thus Lemma \ref{lemIntSemin2} implies that $s\in\fa$, a contradiction.
\end{proof}

The proof we have given for \Tho  \ref{propIntSemin} 
is a simplification of existing ones in the literature. 
Nevertheless it is not fully \cov and this gives only the integral case.

\subsubsection*{Constructive proof (case seminormal and integral)} 
\addcontentsline{toc}{subsection}{Constructive proof}

Remark first that the proof by contradiction shows that 
 the ring  $\gA\sur{\fa}$  is trivial in the following way: if the ring  were
 not trivial \&ct\ldots, it should be trivial. 
In fact the argument proves directly that the ring is trivial
after a slight modification. For this kind of things
see Richman's paper~\cite{Ric} about the nontrivial use of the trivial ring.

A most difficult task is to eliminate the use of
the \idemiz, which is a \emph{purely ideal object} appearing in
the classical proof. 
A lemma is needed for doing this job.   

The intuitive meaning of the lemma is the following: \\
\emph{Let $\gC$ be a reduced ring and $P$  a \pro module 
of rank 1 over $\gC[\uX]$; if $\gC$ is not trivial, some nontrivial localication $S^{-1}\gC$ of $\gC$ have to exist where  $P$ 
becomes free.}

In \clama the answer is easy: use the localisation in a \idemiz. 
This argument appeared in the proof for the ring  $\gC=\gA/\fa$.

The lemma in this intuitive form 
\gui{is not true} from a \cofz point of view (we lack primes).
But fortunately it is  the contraposed form which is needed:\\
\emph{Let $\gC$ be a reduced ring and $P$  a \pro  module
of rank 1 over $\gC[\uX]$; if each localisation  $S^{-1}\gC$ of $\gC$ 
for which  $P$ becomes free is trivial, then $\gC$ 
is itself trivial.}\\
And this form \gui{is true} from a \cofz point of view, \cad 
we get an \algoz!  

In fact we need the following version
where localisations consist only in inverting one \eltz.
Here is THE crucial lemma.  

\begin{lemma} 
\label{lemThierry} \emph{(elimination of a \idemiz)}\\
Let $\gC$  be a reduced ring and $P=(m_{i,j})\in\gC[\uX]^{n \times 
n}$ an \idm matrix of rank 1 
such that  $P(0)=\I_{n,1}$. Let us assume the following implication: 

\smallskip\noindent  
\centerline{$\forall a\in\gC$, if $\Im\,P$ 
is free over $\gC[1/a][\uX], $ then $a=0$.}

\smallskip \noindent  Then $\gC$ is trivial.
\end{lemma}
\begin{Proof}{Proof that Lemma \ref{lemThierry} implies  
\Tho \ref{propIntSemin}. }
We can rewrite the end of the proof of \Tho \ref{propIntSemin}, 
merely replacing the  localisation at the ``purely idealistic'' \idemi ideal $\fp$ by the localisation in one 
\elt $a$. \\
We have two reduced rings $\gC=\gA\sur{\fa}\subseteq \gB\sur{\fa}=\gC'$.
We want to show that $\gC$ is 
trivial. It is sufficient to show that  $\gC$ satisfies, with the matrix 
$P$ mod $\fa$, the hypotheses of THE lemma.\\
So let $a$ be an \elt of $\gA$ such that   $\Im\,P$ 
is free over $\gC[1/a][\uX]$. 
Let $\gC[1/a]=\gL\subseteq  \gC'[1/a]\allowbreak=\gL'$, which is a reduced ring. 
If  $x$ is an object defined over $\gA$  let us call
 $\ov{x}$ what it becomes after the change of ring  $\gA\to\gL'$.
The module  $\ov{M}$ is free over $\gL[\uX]$. 
This implies, by unicity (Lemma~\ref{lempropImProjLib}) and since $f_1(0)=g_1(0)=1$, that the \pols $\ov{f_i}$ and $\ov{g_j}$ are in $\gL[\uX]$.\\
This means that there exists $N\in \NN$  such that  the $a^Nf_i$ and 
$a^Ng_j$ have their \coes in $\gA$. 
Thus   Lemmas~\ref{lemIntSemin1} and~\ref{lemIntSemin2}  imply $a\in\fa$, 
\cad $a=0$ in $\gC$.
\end{Proof}

\begin{Proof}{Proof of Lemma \ref{lemThierry}.}
A classical proof: 
let us assume that $\gC$ is non trivial and let $\fp$ be a \idemiz;
 since $\gC$ is reduced, $\gC_\fp$ is a field; thus  $\Im\,P$ 
becomes free over $\gC_\fp[\uX]$; this implies there exists an 
$a\notin\fp$ such that    $\Im\,P$ 
becomes free over $\gC[1/a][\uX]$; 
thus $a=0$, a contradiction.\\
We have a lemma eliminating a \idemiz. But the proof of
the elimination lemma is a proof by contradiction using a \idemiz! 
\emph{This looks like a bad joke.}\\
No, because this abstract proof can be reread dynamically
and becomes \covz. Here is what happens.\\
Imagine  the ring  $\gC$ is a discrete field. 
Then the $f_i$'s and $g_j$'s 
are calculated with an \algo corresponding to the case of a discrete field.\\ 
This \algo uses disjunction \gui{$a$ is zero or 
invertible},
for \elts $a$ computed by the \algo from the
\coes of $m_{i,j}$'s.  But $\gC$ is only a reduced ring,
without equality or inversibility test. So the \algo 
for discrete fields has to be replaced  
by a tree where we open two branches each time a question 
\gui{Is $a$ zero or invertible?} is asked by 
the \algoz.\\
We get a tree, huge, but finite. 
Assume that the branch \gui{$a$ invertible}  
is put on the left and let us see what happens 
at the leaf of the leftmost branch. 
Some elements $a_1$, \ldots , $a_n$ have been inverted and the module 
$P$  became free over
$\gC[1/(a_1\cdots a_n)][\uX]$. 

\noindent \emph{Conclusion: in the ring  $\gC,$ one has $a_1\cdots 
a_n=0$.}

\noindent Let us go up one step. \\
In  the ring  $\gC[1/(a_1\cdots a_{n-1})]$, we have $a_n=0$. 
So there was no need to open left branch. 
What happens in the branch  $a_n=0$?
We see what is the computation in the leftmost branch after this node.
We have inverted  $a_1$, \ldots , $a_{n-1}$, 
and after we invert $b_1,\ldots 
,b_k$ (if $k=0$ let $b_k=a_{n-1}$).\\
The module $P$  became free on
$\gC[1/(a_1\cdots a_{n-1} b_1\cdots  b_k)][\uX]$.

\noindent \emph{Conclusion: in the ring  $\gC,$ one has $a_1\cdots a_{n-
1} b_1\cdots  b_k=0$.}  \\
Let us go up one step. Since $b_k=0$ there was no need to open 
the left branch. 
What happens in the branch  $b_k=0$? \ldots 

\noindent \emph{And so on.} At the end of the tale
we are at the root of the tree and the module  $P$ is free on the ring
$\gC[\uX]=\gC[1/1][\uX]$. So $1=_\gC0$ by Lemma~\ref{lemThierry}.
\end{Proof}

If we use Lemma \ref{lemIntSemin1bis} instead of Lemma \ref{lemIntSemin1} we get the following more precise result.

\begin{theorem} 
\label{propIntSeminBis} 
If $\gA$ is integral and seminormal and $M$ a
\pro module of rank $1$ over $\AuX$, there exist $c_{1},\ldots,c_{m}$ in the fraction field of $\gA$ such that:
\begin{enumerate}
\item $c_{i}^2$ and $c_{i}^3$ are in $\gA[(c_{j})_{j<i}]$ for $i=1,\ldots,m$,
\item $M$ is free over $\gA[(c_{j})_{j\leq m}][X]$.
\end{enumerate}
\end{theorem}

This gives a strongly explicit form of the Traverso-Swan theorem for integral rings.

\section*{Annex: \zeds reduced rings} 
\addcontentsline{toc}{section}{Annex: \zeds reduced rings}
\label{secAnnex}
\markboth{Seminormalrings}{Annex: \zeds reduced rings}

\setcounter{section}{1}
\setcounter{theorem}{0}

\def\thesection{\Alph{section}}

In  this part, we give some important fact in the theory of
\zeds reduced rings. These rings are good substitute of fields.

As a consequence we get the general form of the Traverso-Swan \thoz.

Moreover  we get a new proof (without computation tree)
of Lemma \ref{lemThierry}  (in fact it is essentially the same proof,
the tree is only hidden behind \idmsz).

\medskip 
\rem The idea of replacing the fraction field of $\gA$  by a 
\zed reduced ring  containing $\gA$ is not in \cite{Swan}: Swan 
uses arguments much more sophisticated in order to reduce the general
case to the \noe case.
The proof of the general case in \cite{coq} is thus a striking improvement
of Swan's proof. Moreover  the
\tho is new since it gives an \algo instead of a purely abstract statement.

\subsection*{A. Basic facts}
\addcontentsline{toc}{subsection}{A. Basic facts}

A ring is \emph{\zedz} when we have
\begin{equation} \label{eqZed}
\forall x\in \gA~\exists a\in\gA~\exists d\in
\NN\quad \quad x^{d}=ax^{d+1}.
\end{equation}

If the ring is reduced $d=1$  is sufficient because $x^d(1-xa)=0$ implies $x(1-
xa)=0.$   

In  a  commutative ring $\gC$, two \elts $a$ and $b$ are  
\emph{quasi inverse}
if one has
              $$a^2b=a,\quad \quad  b^2a=b.$$ 
We say also that $b$ is \emph{the} quasi inverse of $a$. Indeed it is unique: if  $a^2b=a=a^2c$,
 $b^2a=b$  and  $c^2a=c$, then since $ab=a^2b^2$, $ac=a^2c^2$ and   
$a^2(c-b)=a-a=0$, we get 
$$c-b=a(c^2-b^2)=a(c-b)(c+b)=a^2(c-b)(c^2+b^2)=0.$$

On the other hand  if $x^2y=x$, one sees that $xy^2$ is quasi inverse of 
$x$. So:

\begin{fact} 
\label{factZDRQI} 
A ring is \zed reduced \ssi each \elt has a quasi inverse. 
\end{fact}

\smallskip Such rings  are also called  \emph{absolutely 
flat} or \emph{von Neuman regular} (this is mainly used in the non commutative case, with the 
equations $aba=a$ and $bab=b$).

So, \zed reduced rings can be defined as equational structures,
adding a unary law $a\mapsto a\bl$ satisfying (\ref{eqAxqiv})
\begin{equation} \label{eqAxqiv}
a^2\,a\bl=a, \quad\quad a\,(a\bl)^2=a\bl.
\end{equation}

This implies, with $e_a=aa\bl,$

$$
\left|
\begin{array}{lll} 
 e_a^2=e_a,& e_aa=a,& e_aa\bl=a\bl,\\ 
(a\bl)\bl=a,\quad\quad& (ab)\bl=a\bl\, b\bl,& 0\bl = 0, \\ 
1\bl=1,&  
x\,\,\mathrm{regular}\,\Leftrightarrow \,x\,x\bl=1,\quad
&x\,\,\mathrm{\idm}\,\Leftrightarrow \,x=x\bl.
\end{array}
\right.
$$

As an easy consequence:

\begin{fact} 
\label{factZRD} 
A ring is \zed reduced \ssi any \itf is generated by an
\idmz.
\end{fact}

The notion of \zed reduced ring is \emph{the good equational generalisation} of
the notion of field. A field  is nothing but a \zed reduced ring 
which is \emph{connected} 
(\cad with $0$ and $1$ as unique \idmsz).

\begin{lemma} 
\label{lemzedred0} 
Let $\gA\subseteq\gC$ with $\gC$ \zed reduced and $a\in\gC$. We use the notation 
$e_a=aa\bl$.
\begin{enumerate}
\item $e_a$ is the unique  \idm of $\gC$ such that 
$\gen{a}=\gen{e_a}$. Moreover  $\Ann_\gC(a)=\Ann_\gC(e_a)=\gen{1-e_a}$
\item $\gC=e_a\gC\oplus(1-e_a)\gC$ with  
$e_a\gC\simeq\gC[1/e_a]\simeq\aqo{\gC}{1-e_a}$ 
and $(1-e_a)\gC\simeq\aqo{\gC}{e_a}$ \\
(NB: the \id $e_a\gC$ is not a subring, but it is  a ring 
with $e_a$ as $1$).
\item \label{item3lemzedred0} In  $e_a\gC$, $a$ is invertible and in $\aqo{\gC}{e_a}$, $a$ 
is null.
\item If $a\in\gA$, then $e_a\gA[a\bl]\simeq \gA[1/a]$.
\item \label{itemlemzedred0} More generally, with $a,b,c\in\gA$
one has  $(e_ae_be_c)\gA[a\bl,b\bl,c\bl]\simeq \gA[1/(abc)]$.
\item \label{item2lemzedred0} If moreover  $abc=0$, then  
$(e_ae_b)\gA[a\bl,b\bl,c\bl]\simeq \gA[1/(ab)]$.
\end{enumerate}
\end{lemma}
\begin{proof}
The three first items are easy and well known. Let us see 
\ref{itemlemzedred0}.
In  the ring  $\gB=(e_ae_be_c)\gA[a\bl,b\bl,c\bl]$, $abc$  is 
invertible, with inverse
$a\bl b\bl c\bl$. Thus the \homo   
$$
\psi\,:\,\gA\vers{j}\gA[a\bl,b\bl,c\bl]\vers{x\mapsto e_ae_be_cx} 
\gB$$ factorises with a unique  $\theta$ in the following way 
$$
\gA\vers{\pi}\gA[1/(abc)]\vers{\theta}\gB
.$$ 
Since $\gA\subseteq\gC$, $j$ is injective 
and we can identify $x\in\gA$ and $j(x)$. 
The \homo $\theta$ is surjective because
$\theta(1/abc)= a\bl b\bl c\bl=u$ and in $\gB$, $a\bl=bcu,\, 
b\bl=acu, \,c\bl=abu$. On the other hand  $\Ker\,\pi=\Ann_\gA(abc)\subseteq 
\Ker\,\psi $ and if
$x\in\Ker\,\psi$,  then $ e_ae_be_cx=e_{abc}x=0$, thus $abcx=0$. \\
Let us see \ref{item2lemzedred0}. Since $abc=0$, 
$0=e_{abc}=e_ae_be_c$ and in $(e_ae_b)\gA[a\bl,b\bl,c\bl]=\gB_1$ one has
$c\bl=e_ae_bc\bl=e_ae_b(e_cc\bl)=0$ thus 
$\gB_1=(e_ae_b)\gA[a\bl,b\bl]$ and we conclude with \ref{itemlemzedred0}.
\end{proof}

The two last items generalise with an arbitrary
finite number of \elts of $\gA$.

\smallskip 
A possible interpretation of Lemma \ref{lemzedred0} is 
that it works as a formalisation of what happens when we do
dynamic computations in a reduced ring 
 \gui{as if} it were a subring of a field.
Item \emph{\ref{item3lemzedred0}} says that this dynamical computation is possible (at least if
we can find $\gC$). Last items show that  this dynamical computation
 can mimic efficiently the localisation at a \idemiz.

\setcounter{section}{2}
\setcounter{theorem}{0}
\subsection*{B. Reduced rings as subrings of a \zed reduced ring} 
\addcontentsline{toc}{subsection}{B. A reduced ring as subring of a \zed reduced ring}

Since the notion of \zed reduced ring is purely equational, 
universal algebra says that any commutative ring generates
a \zed reduced ring (this gives the adjoint functor to the forgetful functor). We have to see that if the ring 
$\gA$ is reduced, the \homo from~$\gA$ to the  \zed reduced ring
it generates  is injective.

\begin{lemma} 
\label{lem1Zedred1} 
If $\gA\subseteq\gC$ with $\gC$ \zed reduced, and if  
$x\bl$ denotes the quasi inverse of $x$, then
the ring $\gA[(a\bl)_{a\in\gA}]$ is \zed  
(thus it is the least  \zed subring of $\gC$ containing~$\gA$).\\
Variant: if  $\gA\subseteq\gB$ are reduced rings, and if each $a\in\gA$ has a quasi inverse $a\bl$ in $\gB$, then the ring  $\gA[(a\bl)_{a\in\gA}]$ 
is \zedz.
\end{lemma}
\begin{proof}
We have to show that each \elt of $\gA[(a\bl)_{a\in\gA}]$ has a 
quasi inverse. Since $(ab)\bl=a\bl b\bl$ each \elt of 
$\gA[(a\bl)_{a\in\gA}]$ can be written $\sum a_i b_i\bl$ with $a_i, b_i\in\gA$.
On the other hand  $a_i b_i\bl =a_i b_i\bl r_i$  with $r_i=a_i a_i \bl$ 
idempotent. Moreover if we have \idms $r_1,\ldots ,r_k$ 
they generate a Boolean algebra containing a \sfio  $e_1,\ldots 
,e_n$ such that  $r_i=\sum_{e_jr_i=e_j}e_j$ ($i\in\so{1,\ldots,k}$). 
Finally if $e_1,\ldots,e_n$ is a \sfio  in $\gC$, if $a_1,\ldots 
,a_n,b_1,\ldots ,b_n\in\gA$, if  
 $c=\sum_{i=1}^na_ib_i\bl e_i$ and $c'=\sum_{i=1}^na_i\bl b_ie_i$, 
then $c^2c'=c$ and $c'^2c=c'$, thus $c'=c\bl$.
\end{proof}

\begin{lemma} 
\label{lem1Zedred2} 
Let $\gA$ be a reduced ring and $a\in\gA$. 
Let $\gB=\aqo{\gA[T]}{aT^2-T,a^2T-a}$ and $\gC=\gB\red$.
Let $a\bl$ be the image of $T$ in $\gC$. Then 
\begin{enumerate}
\item $\gC\simeq (\aqo{\gA}{a})\red\times \gA[1/a]$ and the natural \homo  
$\gA\to\gC$ is injective (one identifies $\gA$ to a subring of 
$\gC$).
\item $a\bl$ is quasi inverse of $a$ in $\gC.$
\item For any \homo $\gA\vers{\varphi}\gA'$ such that  $\varphi(a)$ 
has a quasi inverse in $\gB$, there exists a unique  \homo 
$\gC\vers{\theta} \gA'$ such that  the \homo $\gA\to\gC\vers{\theta} \gA'$ 
is equal to~$\varphi$.
\end{enumerate}
\end{lemma}

The proof is left to the 
reader. The following corollary is a consequence of the strong unicity 
property given in Lemma~\ref{lem1Zedred2}.

\begin{corollary} 
\label{corlem1Zedred2} 
Let $a_1,\ldots, a_n \in \gA$. Then the ring 
we obtain by repeating the construction of Lemma \ref{lem1Zedred2}
 for each $a_i$
does not depend, up to unique \isoz, of the ordering of $a_i$'s.
\end{corollary}

Example: let us denote $\gA_{\so{a}}$ the ring  constructed in Lemma \ref{lem1Zedred2}; let $a,b,c\in\gA$; then there exists a unique  $\gA$-\homo from 
$((\gA_{\so{a}})_{\so{b}})_{\so{c}}$ to 
$((\gA_{\so{c}})_{\so{b}})_{\so{a}}$ and it is  an \isoz.

Lemma \ref{lem1Zedred2} and Corollary \ref{corlem1Zedred2} give
the following \thoz.

\begin{theorem} 
\label{thAnnexe1} 
Let $\gA$ be a reduced ring. We denote by  $\wh{\gA}$ the ring we
obtain as filtered colimit by iterating the  construction of Lemma 
\ref{lem1Zedred2} (Corollary \ref{corlem1Zedred2} says that this works). \\
Then  $\wh{\gA}$  is a \zed reduced ring and the natural \homo  
$\gA\to \wh{\gA}$ is injective. Moreover 
this ring is \emph{the \zed reduced ring generated by $\gA$} with
the precise following meaning: for any \zed reduced ring $\gA'$, any \homo 
$\gA\vers{\varphi }\gA'$ factorises in a unique way  via the natural \homo
$\gA\to\wh{\gA}$.
\end{theorem}

In a shorter form:
\begin{theorem} 
\label{thRedZed} 
Any reduced ring $\gA$ is contained in a \zed reduced ring
$\gC=\gA[(a\bl)_{a\in\gA}]$.
\end{theorem}

\setcounter{section}{3}
\setcounter{theorem}{0}
\subsection*{C. Zero-dimensional reduced rings and fields} 
\addcontentsline{toc}{subsection}{C. Zero-dimensional reduced rings and fields}

We  said that the notion of \zed reduced ring is 
\emph{the good equational generalisation} of the notion of 
field. In particular any equational consequence 
of field theory is an equational consequence of the theory of \zed 
reduced rings.

In an informal way we can give the following local-global 
\elr  principle.

\medskip \noindent 
{\bf Local-global \elr machinery:     
from discrete fields to  \zed reduced rings.
}\label{mlge}
{\it Most \algos that work  with discrete fields  
can be modified in order to work with  
\zed reduced rings, decomposing the ring
in the product of two components each time the \algo (written for discrete fields) uses the test
\gui{Is this \elt zero or invertible?}. In the first component 
the \elt is zero, in the second one it is invertible.}

\medskip 
We have written \gui{most} rather than \gui{all} because the result of the
\algo given for discrete fields has to be written in a form where
there is no reference to the connectedness of a discrete field.

\medskip Applying the previous local-global machinery allows to get
 \Tho \ref{thZedLib} from Lemma \ref{lemPicGcd}, as soon as we have seen
 that this lemma gives an  \algo for discrete fields.

\begin{theorem} 
\label{thZedLib} 
Let $\gC$ be  a \zed reduced ring. Then any \mrcz~$1$
over $\gC[\uX]$ is free.
\end{theorem}

For the sceptical reader, we give some details in Annex E. 

\setcounter{section}{4}
\setcounter{theorem}{0}
\subsection*{D. Traverso-Swan's \thoz: general case} 
\addcontentsline{toc}{subsection}{D. Traverso-Swan's \thoz: general case}
\begin{Proof}{New \prco of Lemma~\ref{lemThierry}}
Theorems \ref{thRedZed} and \ref{thZedLib} imply there exists a
\zed reduced ring $\gC=\gA[(a\bl)_{a\in\gA}]\supseteq\gA$ with  
$\Im\,P$ 
free over $\gC[\uX]$. This property remains true for 
a ring $\gB\subseteq\gC$  generated by a finite number of quasi 
inverses $a_1\bl,\ldots ,a_r\bl$ of \elts of $\gA$. We write 
$e_i=a_ia_i\bl$ ($e_i$ is an \idm
such that  $e_ia_i=a_i$ and $e_ia_i\bl=a_i\bl$) and $e'_i=1-
e_i$. We give the argument for $r=3$ but it is clear that the argument 
is general. We decompose the ring  $\gB$  in a product of $2^r$ 
rings. Equivalently we write the ring as a direct sum of $2^r$ 
ideals.
\begin{equation}  \label{eqlemThierry}
\gB=e_1e_2e_3\gB\oplus e_1e_2e'_3\gB\oplus e_1e'_2e_3\gB\oplus 
e_1'e_2e_3\gB\oplus e_1e'_2e'_3\gB\oplus e'_1e_2e'_3\gB\oplus 
e'_1e'_2e_3\gB\oplus e'_1e'_2e'_3\gB.  
\end{equation}
Lemma \ref{lemzedred0} Item \emph{\ref{itemlemzedred0}} shows that
$$e_1e_2e_3\gB\simeq e_1e_2e_3\gA[a_1\bl, a_2\bl, a_3\bl)] 
\simeq\gA[1/(a_1a_2a_3)]$$
Since the module   $\Im\,P$  is free over $\gB[\uX]$, it is free
over each of the $2^r$ components. In particular it is free
over $e_1e_2e_3\gB[\uX]\simeq\gA[1/(a_1a_2a_3)][\uX]$. From the hypothesis
in Lemma~\ref{lemThierry} we get $a_1a_2a_3=0$, thus 
$e_1e_2e_3=0$,  $e_1e_2e_3'=e_1e_2$, etc\ldots, 
and the decomposition (\ref{eqlemThierry}) becomes
$$
\gB= e_1e_2\gB\oplus e_1e_3\gB\oplus e_2e_3\gB\oplus 
e_1e'_2e'_3\gB\oplus e'_1e_2e'_3\gB\oplus e'_1e'_2e_3\gB\oplus 
e'_1e'_2e'_3\gB.  
$$
Lemma \ref{lemzedred0} Item \emph{\ref{item2lemzedred0}} shows that  
$e_1e_2\gB\simeq\gA[1/(a_1a_2)]$. Since $P$ is free over this component we get $a_1a_2=0$, thus 
$e_1e_2=0$, $e_1e'_2=e_1$, $e'_1e_2=e_2.$ Similarly $a_1a_3=0=e_1e_3$, 
$a_2a_3=0=e_2e_3$ and finally $e_1e'_2e'_3=e_1$, $e'_1e_2e'_3=e_2$, 
$e'_1e'_2e_3=e_3$. 
We get a new decomposition
$$
\gB=  e_1\gB\oplus e_2\gB\oplus e_3\gB\oplus e'_1e'_2e'_3\gB.  
$$
At the end each $a_i$ is null and $\gB=\gA=\gA[1/1]$. So $1=0$ in $\gA$.
\end{Proof}

\begin{theorem} 
\label{thTSC}\emph{(Traverso-Swan-Coquand)}\\ 
If $\gA$ is a seminormal ring, then $\Pic\,\gA=\Pic\,\AuX$.\\
More precisely if a matrix $P(\uX)\in\AuX^{n\times n}=(m_{i,j}(\uX))_{i,j\in \so{1,\ldots 
,n}}$  is \idm 
of rank 1 and if $P(0)=\I_{n,1}$, then we can construct \pols 
$f_1,\ldots, f_n, g_1,\ldots, g_n\in\AuX$ such that  $m_{i,j}=f_ig_j$ 
for all $i,j$. 
\end{theorem}
\begin{proof}
This proof is only a slight variation of the one given for the integral case.\\
We use the \carn given in Lemma \ref{lemPicPic1}. Let 
$P(\uX)=(m_{i,j}(\uX))_{i,j=1,\ldots ,n}$ be an \idm matrix of rank 
$1$  with
$P(0)=\I_{n,1}$. Let $\gK$ be a \zed reduced ring containing $\gA$.
On $\gK[\uX]$ the module  $\Im\,P(\uX)$ is free. Thus there exist  
$f=(f_1(\uX),\ldots ,f_n(\uX))$ and $g=(g_1(\uX),\ldots ,g_n(\uX))$ 
in
$\gK[\uX]^n$ such that  $m_{i,j}=f_ig_j$ for all $i,j$.
Moreover  since $f_1(0)g_1(0)=1$ and since we can modify $f$ and 
$g$ multiplying them by units, we can assume that 
$f_1(0)=g_1(0)=1$. Since $f_1g_j=m_{1,j}$ and using  
Kronecker \thoz, the \coes des $g_j$ are integral over the ring  generated  
by the \coes of $m_{1,j}$'s. In the same way the \coes of $f_i$'s are integral over the ring  generated by the \coes of $m_{i,1}$'s.\\
Let $\gB$ be the subring of $\gK$ generated by $\gA$ and by the
 \coes of $f_i$'s and $g_j$'s. Thus $\gB$ is a finite extension 
of  $\gA$ (i.e., $\gB$ is a  \tf \Amoz). We have to show $\gA=\gB$.
Let us call $\fa$ the conductor  of $\gA$ in $\gB$. Our aim is now to show 
$\fa=\gen{1}$, \cad  $\gA\sur{\fa}$ is trivial. \\
Following Lemma \ref{lemIntSemin1} $\fa$ is a  radical ideal of 
$\gB$. Lemma \ref{lemIntSemin2} applies with $\gA\subseteq\gB$. 
We have $\gA\sur{\fa}=\gC\subseteq \gB\sur{\fa}=\gC'$, which is reduced, 
and $f_ig_j=m_{i,j}$ in $\gB\sur{\fa}$. 
To show that  $\gC$ is 
trivial, it is sufficient to show that  $\gC$ satisfies, with the matrix 
$P$ mod $\fa$, the hypotheses of Lemma \ref{lemThierry}.\\
So let us consider an $a\in\gA$ such that   $\Im\,P$ 
is free over $\gC[1/a][\uX]$ and
let $\gC[1/a]=\gL\subseteq  \gC'[1/a]=\gL'$.
If  $x$ is an object defined over $\gA$ let us call
 $\ov{x}$ what it becomes after the change of ring  $\gA\to\gL'$.
The module  $\ov{M}$ is free over $\gL[\uX]$. 
This implies, by unicity (Lemma \ref{lempropImProjLib}) and 
since $f_1(0)=g_1(0)=1$, that the \pols $\ov{f_i}$ and $\ov{g_j}$ are in 
$\gL[\uX]$ (if $u(X)\in\gL[\uX]$ is invertible and $u(0)=1$, then $u=1$).\\
This means that there exists $N\in \NN$  such that  the \pols $a^Nf_i$ and 
$a^Ng_j$ have their \coes in $\gA$. Thus Lemma \ref{lemIntSemin2}
  implies that $a\in\fa$, i.e., $a=0$ in
$\gC$.
\end{proof}

If we use Lemma \ref{lemIntSemin1bis} instead of Lemma \ref{lemIntSemin1} we get the following more precise result.

\begin{theorem} 
\label{thTSCBis} 
If $\gA$ is a ring contained in
a \zed reduced  ring $\gB$ and $M$ a projective module
of rank $1$ over $\AuX$, there exist $c_{1},\ldots,c_{m}$ in  $\gB$ such that:
\begin{enumerate}
\item $c_{i}^2$ and $c_{i}^3$ are in $\gA[(c_{j})_{j<i}]$ for $i=1,\ldots,m$,
\item $M$ is free over $\gA[(c_{j})_{j\leq m}][X]$.
\end{enumerate}
\end{theorem}

\setcounter{section}{5}
\setcounter{theorem}{0}
\subsection*{E. Gcd domains} 
\addcontentsline{toc}{subsection}{E. Gcd domains}

In this section we give a detailed proof of \Tho \ref{thZedLib}
without using the local-global \elr machinery page~\pageref{mlge}.

\begin{definition} 
\label{defqi} A ring $\gA$ is called \emph{a pp-ring} if the
annihilator of each \elt  
is (a principal ideal generated by an) \idmz.
For $a\in\gA$, we denote  $e_a$  the \idm such that  $\Ann(a)=\gen{1-
e_a}$. So $a$  is regular in $\gA[1/e_a]$ and null in
 $\gA[1/(1-e_a)]$.
\end{definition}

An integral ring is exactly a connected pp-ring.

\begin{lemma} \label{lemQI}
Let $x_1,\dots,x_n$ be \elts of a commutative ring. If 
one has
$\Ann(x_i) = \gen{r_i}$ where $r_i$'s are  idempotent  ($1 \leq i
\leq n$), let $s_i=1-r_i$,  $t_1=s_1$, $t_2=r_1s_2$, $t_3=r_1r_2s_3 ,\dots$,
$t_{n+1}=r_1r_2\cdots r_n$.
Then  $t_1,\dots,t_{n+1}$ is a \sfio
and the \elt $x=x_1+t_2x_2+\cdots +t_nx_n$ satisfies
$\Ann(x_1,\dots,x_n) = \Ann(x) = \gen{t_{n+1}}$.
\end{lemma}

\begin{corollary} 
\label{corlemQI2} 
Let $\gA$ be a pp-ring and $P=(m_{ij})_{1\leq i,j\leq n}$  a square matrix  such that  
$\Tr(P)$ is regular. Then  there exists a matrix $J\in\gA^{n\times n}$  such that  $J^2=\In$ and $JPJ=JPJ^{-1}$ has a regular \coe in 
position $(1,1)$.
\end{corollary}
\begin{proof}
We apply Lemma \ref{lemQI} with the \elts $x_i=m_{i,i}$. We have $t_{n+1}=0$ because $t_{n+1}\Tr(P)=0$.
Thus $(t_1,\ldots ,t_n)$ is a \sfioz.
Let $J_k$ be the  permutation matrix exchanging vectors $1$ and $k$
in the canonical basis. Let 
$J=t_1\In+t_2J_2+\cdots+ t_nJ_n$. We have $J^2=\In$ and  the \coe in position 
$(1,1)$ of  $JPJ$ is equal to $x=t_1x_1+t_2x_2+\cdots 
+t_nx_n=x_1+t_2x_2+\cdots +t_nx_n$, thus it is regular.
\end{proof}

A \zed reduced ring is a pp-ring and if $\gA$ is a pp-ring,
then the  total fraction  ring of $\gA$, denoted by $\Frac(\gA),$ is a \zed reduced ring: for all $a$, 
$\wi{a}=(1-e_a)+a$ is regular and $a/\wi{a}=a\bl$ is a quasi inverse 
of $a$ in $\Frac(\gA)$. Moreover, for all $a\in\gA$, 
$\gA[1/a]$ is a pp-ring  and $\Frac(\gA[1/a])$ can be indentified
with $e_a\Frac(\gA)\simeq \Frac(\gA)[1/a]$. 

Finally, if $\gA$ is a pp-ring
then $\AX $ is a pp-ring and the annihilator of a \pol $f$ is
generated by the \idm equal to the product of annihilators of the \coesz.

\smallskip In  a pp-ring  if $a$ divides $b$ and $b$ divides $a$,
one has $e_a=e_b$ and $ua=b$  with an  invertible \elt $u$.
This allows to develop a theory of gcd pp-rings analogous to
the theory of gcd domains.

\begin{definition} 
\label{defMongcd}  A commutative regular monoid is called 
a \emph{gcd monoid} if any two \elts do have a greatest common divisor. 
If $g$ is a gcd for $a$ and $b$ we write  $g=\gcd(a,b)$
(in fact a gcd is defined up to a unit).
\end{definition}

\begin{lemma} 
\label{lemQigcd} Let  $\gA$ be a pp-ring. \Propeq 
\begin{enumerate}
\item The monoid of regular \elts is a gcd monoid.
\item For any \idm $e$  regular \elts of $\gA[1/e]$ 
give a gcd monoid.
\item Two arbitrary \elts have a gcd.
\end{enumerate}
In this case we say that $\gA$ is  a \emph{gcd pp-ring}. 
\end{lemma}
%
\begin{proof}
For example,  to show that 1. implies 2., one introduces, for $a\in e\gA$ 
with $a$ regular in $\gA[1/e]$, the \elt $\wi{a}=(1-e_a)+a$ 
which is regular in $\gA$. If $g$ is the gcd of  $\wi{a}$ and  $\wi{c}$ 
in $\gA$, the same \elt $g$, viewed in $\gA[1/e]$, is the gcd of $a$ 
and $c$.
\end{proof}

A  gcd pp-ring which is connected is a gcd domain.
A \zed reduced ring is a gcd pp-ring.

Let  $\gA$ be a gcd pp-ring and  $f(X)=\sum_{k=0}^nf_kX^k\in\gA[X]$, we 
denote by $\rG(f)$ the gcd (defined up to a unit) of the \coes of 
$f$. If $\rG(f)=1$ one says that $f$ is primitive{\footnote{~Warning. This conflicts  another   traditional terminology: $f$ is primitive 
when the ideal of \coes of $f$ contains~$1$.}}.

\medskip 

\medskip We have to see that arguments in the proof of Lemma
\ref{lemPicGcd} work also for gcd pp-rings. 
In particular, \emph{if $\gA$ is a gcd pp-ring, so is $\AX$}. 
So for any \zed reduced ring  $\gA$, the ring  $\AuX$ 
is a gcd pp-ring and thus any \mrcz~1 over $\AuX$ is free.

Let us see the first argument in the proof:
\emph{Let $P=(m_{i,j})$ be an \idm matrix of rank~1. 
Since $\sum_i m_{i,i}=1$ we can assume that $m_{1,1}$ is 
regular.}
Corollary \ref{corlemQI2} gives the answer.

\smallskip 
For the end of the proof we look at the
\gui{bible} \cite{MRR}, where all proofs are algorithmic 
(and often very simple).

\begin{lemma} 
\label{lemGCD1} \emph{(cf. \Tho 1.1 page 108 in \cite{MRR})}\\
Let $a,b,c$ be \elts of a gcd pp-ring. Then 
\begin{enumerate}
\item $\gcd(\gcd(a,b),c)=\gcd(a,\gcd(b,c))$.
\item $c\cdot \gcd(a,b)=\gcd(ca,cb)$.
\item If $x=\gcd(a,b)$, then $\gcd(a,bc)=\gcd(a,xc)$.
\item If $a|bc$ and $\gcd(a,b)=e_b$ then $a|e_bc$. 
\end{enumerate}
\end{lemma}

%
\begin{proof}
If one of the 3 \elts $a,b,c$ is null, all is clear. 
In the general case let $r_i$  be an \elt of a \sfio generated by $e_a$, $e_b$ and $e_c$. Each \elt $a,b,c$ 
is null or regular in $\gA[1/r_i]$. The proof given in \cite{MRR} for
gcd monoids  works for the component in which  $a,b,c$ are regular. 
\end{proof}

 A consequence of Item \emph{2} in Lemma \ref{lemGCD1} is
that in a gcd pp-ring, a primitive \pol  is a regular \elt 
of~$\AX $. 

\begin{lemma} 
\label{lemPrimFact} \emph{(Lemma 4.2 page 123 in \cite{MRR})}
Let $\gA$ be a gcd pp-ring, $\gK=\Frac(\gA)$ and  
$f\in\gK[X]$. We can find  a primitive \pol $g\in\AX $ and
$c\in\gK$ such that  $f=cg$. If we have  another   decomposition $f=c'g'$
then there exists $u\in\gA^\times$ such that~$c=uc'$.
\end{lemma}
%
\begin{proof}
If $f=0$ we take $g=1$ and $c=0$. 
If $\rG(f)$ is regular, the proof in \cite{MRR} works,
replacing \hbox{``$a \neq 0$''} by ``$a$ is regular''.
Thus we decompose  the ring  in two components by using the \idmz~$e_{\rG(f)}$. 
\end{proof}
%

\begin{lemma} 
\label{lemGauss} \emph{(Gauss Lemma, Lemma 4.3 page 123 in 
\cite{MRR})}\\
Let $\gA$ be a gcd pp-ring and $f,g\in\AX $. Then 
$\rG(f)\rG(g)=\rG(fg)$.
\end{lemma}
%
\begin{proof}
Let $(r_i)$ be the  \sfio generated by $e_c$'s for all  
\coesz~$c$  of $f$ and $g$. In  each ring  $\gA[1/r_i]$ 
\pols $f$ and $g$ have a well defined degree{\footnote{~Precisely we know an integer $q\geq 0$ such that  the 
\coe of degree $q$ is leading  and regular. Note there is no need to assume we know if the ring is trivial or not.}}. Let us see that the elegant 
proof by induction on $n+m=\deg(f)+\deg(g)$ given in \cite{MRR} 
works.

\noindent We reason by induction on $m+n$.
By distributivity (Item \emph{2} in Lemma \ref{lemGCD1}) and using Lemma~\ref{lemPrimFact}, we are reduced to the case where $\rG(f)=\rG(g)=1$. 
Let~$c=\rG(fg)$ and~$d=\gcd(f_n,c)$. 
Then~$d$ divides $(f-f_nX^n)\,g$. If $f=f_nX^n$ the result is clear. 
In the other case, by \hdrz~$d$ divides $\rG(f-f_nX^n)\,\rG(g)=\rG(f-f_nX^n)$, thus $d$ divides $f$ and $d=1$. 
So $\gcd(f_n,c)=1$. 
Similarly  $\gcd(g_m,c)=1$ and since $c$ divides $f_ng_m$, $c=1$. 
\end{proof}

Finally proofs in \cite{MRR} for the two following results do work
in our new context.
\begin{corollary} 
\label{corlemGauss} \emph{(Corollary 4.4 page 123 in \cite{MRR})}\\
Let $\gA$ be a gcd pp-ring,  $f,g\in\AX $ and 
$\gK=\Frac(\gA)$. Then  $f$ divides $g$ in $\AX $ \ssi
 $f$ divides $g$ in $\gK[X]$ and $\rG(f)$ divides $\rG(g)$.
\end{corollary}

\begin{theorem} 
\label{thGauss} \emph{(\Tho 4.6 page 124 in \cite{MRR})}\\
If $\gA$  is a gcd pp-ring, then so is $\AX$.
\end{theorem}

\medskip In fact all these verifications are quasi automatic.
Proofs in \cite{MRR}, which are also algorithms, are 
based on the disjunction \gui{$x=0$ 
or $x$ regular} in a gcd integral ring. 
In the case of gcd pp-rings, it is sufficient 
to realise the disjunction by decomposing 
the ring  in two components by using the \idm $e_x$.

\newpage

\rdb


\normalsize
\endgroup
\stopcontents[english]

\clearpage
\newpage
\thispagestyle{empty}


\clearpage
\newpage

\renewcommand\thepage{F\arabic{page}}\renewcommand\theHsection{F\arabic{section}}

\begingroup
\clearpage
\setcounter{page}{1} 
\setcounter{section}{0}
\setcounter{subsection}{0}
\setcounter{equation}{0}

\selectlanguage{french}
\def\frenchproofname{\textsl{Démonstration}}

\FrenchFootnotes



\newtheorem{ftheorem}{Théorème}[section]   
\newtheorem{fthdef}[ftheorem]{Théorème et définition}
\newtheorem{fproposition}[ftheorem]{Proposition}
\newtheorem{flemma}[ftheorem]{Lemme}
\newtheorem{fcorollary}[ftheorem]{Corolaire} 
\newtheorem{fpropdef}[ftheorem]{Proposition et définition}
\newtheorem{fremark}[ftheorem]{Remarque}
\newtheorem{fcomment}[ftheorem]{Commentaire}
\newtheorem{fexample}[ftheorem]{Exemple}
\newtheorem{ffact}[ftheorem]{Fait}
\newtheorem{fdefinition}[ftheorem]{Définition}
\newtheorem{fdefinitions}[ftheorem]{Définitions}
\newtheorem{fconvention}[ftheorem]{Convention}
\newtheorem{fnotation}[ftheorem]{Notation} 
\newtheorem{fnotadefi}[ftheorem]{Notation et définition} 
\newtheorem{fproblem}[ftheorem]{Problème}
\newtheorem{fquestion}[ftheorem]{Question}

\newtheorem{ftheorem1}{Théorème}[section]   
\newtheorem{fproposition1}[ftheorem1]{Proposition}
\newtheorem{flemma1}[ftheorem1]{Lemme}
\newtheorem{fcorollary1}[ftheorem1]{Corollaire} 
\newtheorem{fpropdef1}[ftheorem1]{Proposition et définition}
\newtheorem{fremark1}[ftheorem1]{Remarque}
\newtheorem{ffact1}[ftheorem1]{Fait}
\newtheorem{fdefinition1}[ftheorem1]{Définition}

\newcommand {\junk}[1]{}

\newcommand {\rem}{\noindent \emph{Remarque. }  }
\newcommand {\rems}{\noindent \emph{Remarques. }  }
\newcommand {\comm}{\noindent \emph{Commentaire. }  }
\newcommand\exl{\mni{\it Exemple. }}

\newenvironment{proof}{
\trivlist \item[\hskip \labelsep{\it Démonstration.}]\hskip 0pt\\}
{\hfill \mbox{$\Box$}
\endtrivlist}
\makeatletter

\newenvironment{prooF}{
\trivlist \item[\hskip \labelsep{\it Démonstration.}]\hskip 0pt\\}
{\hfill \mbox{\textsf{Pas mal n'est-ce pas~?}}
\endtrivlist}
\makeatletter

\newenvironment{Proof}[1]{
\trivlist \item[\hskip \labelsep{\it #1}]\hskip 0pt\\}
{\hfill \mbox{$\Box$}
\endtrivlist}
\makeatletter

\newcommand \eop {\hbox{}\nobreak\hfill
\vrule width 1.4mm height 1.4mm depth 0mm \par \goodbreak 
\smallskip}

\newcommand\aec{\texttt{\`A écrire.}}

\DeclareRobustCommand{\guig}{\mbox{{\usefont{U}{lasy}%
{\if b\expandafter\@car\f@series\@nil b\else m\fi}{n}%
\char40\kern-0.20em\char40}~}}
\DeclareRobustCommand{\guid}{\mbox{~\usefont{U}{lasy}%
{\if b\expandafter\@car\f@series\@nil b\else m\fi}{n}%
\char41\kern-0.20em\char41}}
\newcommand\gui[1]{\guig{#1}\guid}

\newcommand \aigu{\mathaccent19}    
\renewcommand \grave{\mathaccent18}    
\newcommand \noi {\noindent}
\newcommand \sms {\smallskip}
\newcommand \sni {\sms\noi}
\newcommand \ms {\medskip}
\newcommand \mni {\ms\noi}
\newcommand \bs {\bigskip}
\newcommand \bni {\bs\noi}
\newcommand \hs {\qquad}
\newcommand \alb {\allowbreak}
\newcommand \ce {\centerline}

\newcommand\sta{^\star}
\newcommand\ista{_\star}
\newcommand \bu {{$\bullet$}}
\newcommand \bl {^\bullet}
\renewcommand \cir {^\circ}
\newcommand\equidef{\buildrel{{\rm def}}\over{\quad\Longleftrightarrow\quad}} 
\newcommand\eqdefi{\buildrel{\rm def}\over {\;=\;}}


\newcommand\mapright[1]{\smash{\mathop{\longrightarrow}\limits^{#1}}} 
\newcommand\maprightto[1]{\smash{\mathop{\longmapsto}\limits^{#1}}} 
\newcommand\mapdown[1]{\downarrow\rlap{$\vcenter{\hbox{$\scriptstyle 
#1$}}$}}
\newcommand\eqdf[1]{\buildrel{#1}\over =}
\newcommand\equivdf[1]{\buildrel{#1}\over \longleftrightarrow}
\newcommand\vers[1]{\buildrel{#1}\over \longrightarrow }
\newcommand\impdef[1]{\buildrel{#1}\over \Longrightarrow} 
\newcommand\tra[1]{{\,^{\rm t}\!#1}}
\newcommand\gen[1]{\left\langle{#1}\right\rangle} 
\newcommand\so[1]{\left\{{#1}\right\}} 
\newcommand \sur[1]{\!\left/#1\right.}
\newcommand \aqo[2]{#1\sur{\gen{#2}}}

\newcommand \snic[1] {\sni\centerline{$#1$}\sms}
\newcommand \snif[3] 
{\vspace{#1}\noindent\centerline{$#3$}\vspace{#2}}

\newcommand \ov[1] {\overline{#1}}

\newcommand \wh[1] {\widehat{#1} }
\newcommand \wi[1] {\widetilde{#1} }

\newcommand \cmatrix[1]{\left[\matrix{#1}\right]}  
\newcommand \bloc[4]{\left[\matrix{#1 & #2 \cr #3 & #4}\right]}

\newcommand \CC{\mathbb {C}} 
\newcommand \NN{\mathbb {N}} 
\newcommand \ZZ{\mathbb {Z}} 

\newcommand \gx{{\underline{x}}}
\newcommand \gy{{\underline{y}}}
\newcommand \gz{{\underline{z}}}
\newcommand \gu{{\underline{u}}}
\newcommand \gt{{\underline{t}}}
\newcommand \gc{{\underline{c}}}

\newcommand \gk{{\bf k}}
\newcommand \gA{\mathbf{A}}
\newcommand \gB{\mathbf{B}}
\newcommand \gC{\mathbf{C}}
\newcommand \gK{\mathbf{K}}
\newcommand \gL{\mathbf{L}}
\newcommand \gM{\mathbf{M}}
\newcommand \gT{\mathbf{T}}
\newcommand \gV{\mathbf{V}}
\newcommand \gZ{\mathbf{Z}}

\newcommand \GL{\mathsf{GL}}

\newcommand \cC {{\cal C}}
\newcommand \cD {{\cal D}}
\newcommand \cI {{\cal I}}
\newcommand \cJ {{\cal J}}
\newcommand \cF {{\cal F}}
\newcommand \cN {{\cal N}}
\newcommand \cP {{\cal P}}
\newcommand \cQ {{\cal Q}}
\newcommand \cM {{\cal M}}
\newcommand \cT {{\cal T}}
\newcommand \cL {{\cal L}}
\newcommand \cR {{\cal R}}
\newcommand \cS {{\cal S}}

\newcommand \rI {\mathrm{I}}
\newcommand \rD {\mathrm{D}}
\newcommand \rG {\mathrm{G}}
\newcommand \rV {\mathrm{V}}
\newcommand \rJ {\mathrm{J}}
\newcommand \rH {\mathrm{H}}
\newcommand \rK {\mathrm{K}}
\newcommand \rN {\mathrm{N}}
\newcommand \rP {\mathrm{P}}
\newcommand \rL{\mathrm{L}}
\newcommand \rPr{\mathrm{Pr}}

\newcommand\fa{\mathfrak{a}}
\newcommand\fb{\mathfrak{b}}
\newcommand\fc{\mathfrak{c}}
\newcommand\fA{\mathfrak{A}}
\newcommand\fB{\mathfrak{B}}
\newcommand\fD{\mathfrak{D}}
\newcommand\fI{\mathfrak{i}}
\newcommand\fII{\mathfrak{I}}
\newcommand\fj{\mathfrak{j}}
\newcommand\fJ{\mathfrak{J}}
\newcommand\fF{\mathfrak{F}}
\newcommand\ff{\mathfrak{f}}
\newcommand\ffg{\mathfrak{g}}
\newcommand\fG{\mathfrak{G}}
\newcommand\fh{\mathfrak{h}}
\newcommand\fl{\mathfrak{l}}
\newcommand\fm{\mathfrak{m}}
\newcommand\fM{\mathfrak{M}}
\newcommand\fp{\mathfrak{p}}
\newcommand\fP{\mathfrak{P}}
\newcommand\fq{\mathfrak{q}}
\newcommand\fV{\mathfrak{V}}

\newcommand \Zg {{\Z[G]}}

\newcommand \vu {\vee} 
\newcommand \vi {\wedge} 
\newcommand \Vu {\bigvee}
\newcommand \Vi {\bigwedge}
\newcommand \im {\rightarrow} 
\newcommand \da {\,\downarrow\!}
\newcommand \ua {\,\uparrow\!}
\newcommand \vd {\,\vdash\,}
\newcommand \vdu[1] {\,\vdash^{#1}\,}
\newcommand \vdb[1] {\,\vdash_{#1}\,}


\newcommand \uX {\underline{X}}
\newcommand \Ared {\gA\red}
\newcommand \AuX {\gA[\uX]}
\newcommand \Xm {X_1,\ldots,X_m}
\newcommand \AXn {\gA[\Xn]}
\newcommand \AXm {\gA[\Xm]}
\newcommand \xn {x_1,\ldots,x_n}
\newcommand \yn {y_1,\ldots,y_n}
\newcommand \cq {c_1,\ldots,c_q}

\newcommand \Pf {{{\cal P}_{\mathrm{f}}}}

\newcommand \Ann {\mathrm{Ann}}
\newcommand \Diag {\mathrm{Diag}}
\newcommand \Hom {\mathrm{Hom}}
\renewcommand \det {\mathrm{det}}
\renewcommand \dim {\mathrm{dim}}
\newcommand \Coker {\mathrm{Coker}}
\newcommand \Frac {\mathrm{Frac}}
\newcommand \Ker {\mathrm{Ker}}
\renewcommand \Im {\mathrm{Im}}
\newcommand \Id {\mathrm{Id}}
\newcommand \I {\mathrm{I}}
\newcommand \In {\I_n}
\newcommand \Mat {\mathrm{Mat}}
\newcommand \Rad {\mathrm{Rad}}
\newcommand \Tr {\mathrm{Tr}}
\newcommand \mod {\,\mathrm{mod}\,}
\newcommand \pgcd {\mathrm{pgcd}}
\newcommand \red {_\mathrm{red}}

\newcommand \GKO {\mathsf{GK}_{0}}
\newcommand \Pic {\mathsf{Pic}}
\newcommand \Spec {\mathsf{Spec}}

\newcommand \num {{n$^{\mathrm{o}}$}}
 
\newcommand \recu {récur\-rence }
\newcommand \recuz {récur\-rence}
\newcommand \hdr {hypo\-thèse de \recu }
\newcommand \hdrz {hypo\-thèse de \recuz}
\newcommand \cad {c'est-\`a-dire }
\newcommand \Cad {C'est-\`a-dire }
\newcommand \cade {c'est-\`a-dire encore }
\newcommand \ssi {si, et seulement si, }
\newcommand \cnes {con\-di\-tion néces\-saire et 
suffi\-sante }
\newcommand \spdg {sans per\-te de géné\-ra\-lité }
\newcommand \Propeq {Les pro\-prié\-tés suivantes sont 
équiva\-lentes:}
\newcommand \propeq {les pro\-prié\-tés sui\-van\-tes sont 
équiva\-lentes:}
\newcommand \disept {17$^{\mathrm{\grave{e}me}}$ problème de Hilbert 
}


\newcommand \Amo {$\gA$-mo\-du\-le }
\newcommand \Amos {$\gA$-mo\-du\-les }
\newcommand \Amoz {$\gA$-mo\-du\-le}
\newcommand \Amosz {$\gA$-mo\-du\-les}

\newcommand \Bmo {$\gB$-mo\-du\-le }
\newcommand \Bmos {$\gB$-mo\-du\-les }
\newcommand \Bmoz {$\gB$-mo\-du\-le}
\newcommand \Bmosz {$\gB$-mo\-du\-les}

\newcommand \Ali {ap\-pli\-ca\-tion $\gA$-\lin }
\newcommand \Alis {ap\-pli\-ca\-tions $\gA$-\lins }
\newcommand \Aliz {ap\-pli\-ca\-tion $\gA$-\linz}
\newcommand \Alisz {ap\-pli\-ca\-tions $\gA$-\linsz}

\newcommand \alg {al\-gè\-bre }
\newcommand \algs {al\-gè\-bres }
\newcommand \algz {al\-gè\-bre}
\newcommand \algsz {al\-gè\-bres}

\newcommand \algo{al\-go\-rith\-me }
\newcommand \algos{al\-go\-rith\-mes }
\newcommand \algoz{al\-go\-rith\-me}
\newcommand \algosz{al\-go\-rith\-mes}

\newcommand \ali {appli\-ca\-tion \lin }
\newcommand \alis {appli\-ca\-tions \lins }
\newcommand \aliz {appli\-ca\-tion \linz}
\newcommand \alisz {appli\-ca\-tions \linsz}

\newcommand \auto {auto\-mor\-phisme }
\newcommand \autos {auto\-mor\-phismes }
\newcommand \autosz {auto\-mor\-phismes}
\newcommand \autoz {auto\-mor\-phisme}

\newcommand \cac {corps algé\-bri\-que\-ment clos }
\newcommand \cacz {corps algé\-bri\-que\-ment clos}

\newcommand \carn{carac\-té\-ri\-sation }  
\newcommand \carns{carac\-té\-ri\-sations }

\newcommand \coe {co\-ef\-fi\-cient }
\newcommand \coes {co\-ef\-fi\-cients }
\newcommand \coez {co\-ef\-fi\-cient}
\newcommand \coesz {co\-ef\-fi\-cients}

\newcommand \coli {com\-bi\-nai\-son li\-né\-ai\-re }
\newcommand \colis {com\-bi\-nai\-sons li\-né\-ai\-res }
\newcommand \coliz {com\-bi\-nai\-son li\-né\-ai\-re}
\newcommand \colisz {com\-bi\-nai\-sons li\-né\-ai\-res}

\newcommand \com {co\-ma\-xi\-maux }
\newcommand \comz {co\-ma\-xi\-maux}

\newcommand \ddk {di\-men\-sion de Krull }
\newcommand \ddkz {di\-men\-sion de Krull}
\newcommand \ddi {de di\-men\-sion infé\-rieure ou égale \`a~}

\newcommand \dfn{défi\-nition }  
\newcommand \dfns{défi\-nitions }  
\newcommand \dfnz{défi\-nition}  
\newcommand \dfnsz{défi\-nitions}

\newcommand \dvz {di\-viseur de zé\-ro }
\newcommand \dvzs {di\-viseurs de zé\-ro }
\newcommand \dvzz {di\-viseur de zé\-ro}
\newcommand \dvzsz {di\-viseurs de zé\-ro}

\newcommand \egmt {éga\-lement }

\newcommand \egt {éga\-lité }
\newcommand \egts {éga\-lités }
\newcommand \egtz {éga\-lité}
\newcommand \egtsz {éga\-lités}

\newcommand \elr{élé\-men\-tai\-re }  
\newcommand \elrs{élé\-men\-tai\-res }  
\newcommand \elrz{élé\-men\-tai\-re}  
\newcommand \elrsz{élé\-men\-tai\-res}  

\newcommand \elt{élé\-ment }  
\newcommand \elts{élé\-ments }  
\newcommand \eltz{élé\-ment}  
\newcommand \eltsz{élé\-ments}  

\def \endo {endo\-mor\-phisme }
\def \endoz {endo\-mor\-phisme}
\def \endos {endo\-mor\-phismes }

\newcommand \entrel {rela\-tion impli\-ca\-tive }
\newcommand \entrelz {rela\-tion impli\-ca\-tive}
\newcommand \entrels {rela\-tions impli\-ca\-tives }
\newcommand \entrelsz {rela\-tions impli\-ca\-tives}

\newcommand\evc{es\-pa\-ce vec\-to\-riel } 
\newcommand\evcs{es\-pa\-ces vec\-to\-riels } 
\newcommand\evcz{es\-pa\-ce vec\-to\-riel} 
\newcommand\evcsz{es\-pa\-ces vec\-to\-riels} 

\newcommand \gtr{gé\-né\-ra\-teur }  
\newcommand \gtrs{gé\-né\-ra\-teurs }  
\newcommand \gtrz{gé\-né\-ra\-teur}  
\newcommand \gtrsz{gé\-né\-ra\-teurs}

\newcommand \homo {ho\-mo\-mor\-phisme }
\newcommand \homoz {ho\-mo\-mor\-phisme}
\newcommand \homos {ho\-mo\-mor\-phismes }
\newcommand \homosz {ho\-mo\-mor\-phismes}

\newcommand \id {idéal }
\newcommand \ids {idéaux }
\newcommand \idz {idéal}
\newcommand \idsz {idéaux}

\newcommand \idd {idéal déter\-minan\-tiel }
\newcommand \idds {idéaux déter\-minan\-tiels }
\newcommand \iddz {idéal déter\-minan\-tiel}
\newcommand \iddsz {idéaux déter\-minan\-tiels}

\newcommand \idema {\id ma\-xi\-mal }
\newcommand \idemaz {\id ma\-xi\-mal}
\newcommand \idemas {\ids ma\-xi\-maux }
\newcommand \idemasz {\ids ma\-xi\-maux}

\newcommand \idep {\id pre\-mier }
\newcommand \idepz {\id pre\-mier}
\newcommand \ideps {\ids pre\-miers }
\newcommand \idepsz {\ids pre\-miers}

\newcommand \idemi {\idep mi\-ni\-mal }
\newcommand \idemis {ideps mi\-ni\-maux }
\newcommand \idemiz {\idep mi\-ni\-mal}
\newcommand \idemisz {ideps mi\-ni\-maux}

\newcommand \idf {idéal de Fitting }
\newcommand \idfs {idéaux de Fitting }
\newcommand \idfz {idéal de Fitting}
\newcommand \idfsz {idéaux de Fitting}

\newcommand \idm {idem\-po\-tent }
\newcommand \idms {idem\-po\-tents }
\newcommand \idmz {idem\-po\-tent}
\newcommand \idmsz {idem\-po\-tents}

\newcommand \idme {idem\-po\-tente }
\newcommand \idmes {idem\-po\-tentes }
\newcommand \idmez {idem\-po\-tente}
\newcommand \idmesz {idem\-po\-tentes}

\newcommand \iso {iso\-mor\-phisme }
\newcommand \isos {iso\-mor\-phismes }
\newcommand \isosz {iso\-mor\-phismes}
\newcommand \isoz {iso\-mor\-phisme}

\newcommand \itf {\id \tf}
\newcommand \itfs {\ids \tf}
\newcommand \itfz {\id \tfz}
\newcommand \itfsz {\ids \tfz}

\newcommand \lin {li\-né\-ai\-re }
\newcommand \lins {li\-né\-ai\-res }
\newcommand \linz {li\-né\-ai\-re}
\newcommand \linsz {li\-né\-ai\-res}

\newcommand \mo {mo\-no\"{\i}de }
\newcommand \moco {\mos\com}
\newcommand \mocoz {\mos\comz}
\newcommand \mos {mo\-no\"{\i}des }
\newcommand \mosz {mo\-no\"{\i}des}
\newcommand \moz {mo\-no\"{\i}de}

\newcommand \mptf {mo\-dule \ptf}
\newcommand \mptfs {mo\-dules \ptfs}
\newcommand \mptfz {mo\-dule \ptfz}
\newcommand \mptfsz {mo\-dules \ptfsz}

\newcommand \mrc {mo\-dule \prc }
\newcommand \mrcz {mo\-dule \prcz}
\newcommand \mrcs {mo\-dules \prcs }
\newcommand \mrcsz {mo\-dules \prcsz}

\newcommand \ndz {non diviseur de zéro }
\newcommand \ndzs {non diviseurs de zéro }

\newcommand \nst {Null\-stellen\-satz }
\newcommand \nstz {Null\-stellen\-satz}
\newcommand \nsts {Null\-stellen\-s\"atze }
\newcommand \nstsz {Null\-stellen\-s\"atze}

\newcommand \odz {ouvert de Zariski }

\newcommand \oqc {ouvert \qc}
\newcommand \oqcs {ouverts \qcs}
\newcommand \oqcz {ouvert \qcz}
\newcommand \oqcsz {ouverts \qcsz}

\newcommand \pa {couple saturé }
\newcommand \pas {couples saturés }
\newcommand \paz {couple saturé}
\newcommand \pasz {couples saturés}

\newcommand\pb{pro\-blè\-me }  
\newcommand\pbs{pro\-blè\-mes }  
\newcommand\pbz{pro\-blè\-me}  
\newcommand\pbsz{pro\-blè\-mes}

\newcommand \pf {de présen\-tation finie }
\newcommand \pfz {de présen\-tation finie}

\newcommand \pn {présen\-ta\-tion }
\newcommand \pns {présen\-ta\-tions }
\newcommand \pnz {présen\-ta\-tion}
\newcommand \pnsz {présen\-ta\-tions}

\newcommand \pol {poly\-nôme }
\newcommand \pols {poly\-nômes }
\newcommand \polz {poly\-nôme}
\newcommand \polsz {poly\-nômes}

\newcommand \polcar 
{\pol ca\-rac\-té\-ris\-ti\-que }
\newcommand \polcarz 
{\pol ca\-rac\-té\-ris\-ti\-que}

\newcommand \prc {\pro de rang constant }
\newcommand \prcs {\pros de rang constant }
\newcommand \prcz {\pro de rang constant}
\newcommand \prcsz {\pros de rang constant}

\newcommand \prn {pro\-jec\-tion }
\newcommand \prns {pro\-jec\-tions }
\newcommand \prnz {pro\-jec\-tion}
\newcommand \prnsz {pro\-jec\-tions}

\newcommand \pro {pro\-jec\-tif }
\newcommand \pros {pro\-jec\-tifs }
\newcommand \proz {pro\-jec\-tif}
\newcommand \prosz {pro\-jec\-tifs}

\newcommand \ptf {\pro \tf }
\newcommand \ptfz {\pro \tfz}
\newcommand \ptfs {\pros \tf }
\newcommand \ptfsz {\pros \tfz}

\newcommand \qc {quasi-compact }
\newcommand \qcs {quasi-compacts }
\newcommand \qcz {quasi-compact}
\newcommand \qcsz {quasi-compacts}

\newcommand \qi {qua\-si in\-tè\-gre }
\newcommand \qis {qua\-si in\-tè\-gres }
\newcommand \qisz {qua\-si in\-tè\-gres}
\newcommand \qiz {qua\-si in\-tè\-gre}

\newcommand \rdl {rela\-tion de dépen\-dance li\-néai\-re }
\newcommand \rdls {rela\-tions de dépen\-dance li\-néai\-re }
\newcommand \rdlsz 
{rela\-tions de dépen\-dance li\-néai\-re}
\newcommand \rdlz {rela\-tion de dépen\-dance li\-néai\-re}

\newcommand \rdi {rela\-tion de dépen\-dance in\-tégra\-le }
\newcommand \rdis {rela\-tions de dépen\-dance in\-tégra\-le }
\newcommand \rdiz {rela\-tion de dépen\-dance in\-tégra\-le}
\newcommand \rdisz 
{rela\-tions de dépen\-dance in\-tégra\-le}

\newcommand \sad {struc\-ture algé\-bri\-que dyna\-mi\-que }
\newcommand \sads {struc\-tures algé\-bri\-ques dyna\-mi\-ques }
\newcommand \sadz {struc\-ture algé\-bri\-que dyna\-mi\-que}
\newcommand \sadsz 
{struc\-tures algé\-bri\-ques dyna\-mi\-ques}

\newcommand \sfio {sys\-tème fondamental d'\idms ortho\-gonaux }
\newcommand \sfios {sys\-tèmes fondamentaux d'\idms ortho\-gonaux }
\newcommand \sfioz {sys\-tème fondamental d'\idms ortho\-gonaux}
\newcommand \sfiosz {sys\-tèmes fondamentaux d'\idms ortho\-gonaux}

\newcommand \sli {\sys \lin }
\newcommand \slis {\syss \lins }
\newcommand \slisz {\syss \linsz}
\newcommand \sliz {\sys \linz}

\newcommand \sys {sys\-tè\-me }
\newcommand \syss {sys\-tè\-mes }
\newcommand \sysz {sys\-tè\-me}
\newcommand \syssz {sys\-tè\-mes}

\newcommand \tf {de type fini }
\newcommand \tfz {de type fini} 

\newcommand \trdi {treil\-lis dis\-tri\-bu\-tif }
\newcommand \trdis {treil\-lis dis\-tri\-bu\-tifs }
\newcommand \trdiz {treil\-lis dis\-tri\-bu\-tif}
\newcommand \trdisz {treil\-lis dis\-tri\-bu\-tifs}

\newcommand \Tho {Théo\-rè\-me }
\newcommand \tho {théo\-rè\-me }
\newcommand \thos {théo\-rè\-mes }
\newcommand \thoz {théo\-rè\-me}
\newcommand \thosz {théo\-rè\-mes}

\newcommand \vfn {véri\-fi\-cation }
\newcommand \vfns {véri\-fi\-cations }
\newcommand \vfnz {véri\-fi\-cation}
\newcommand \vfnsz {véri\-fi\-cations}

\newcommand \zed {z\'{e}\-ro-di\-men\-sion\-nel }
\newcommand \zedz {z\'{e}\-ro-di\-men\-sion\-nel}
\newcommand \zede {z\'{e}\-ro-di\-men\-sion\-nel\-le }
\newcommand \zedez {z\'{e}\-ro-di\-men\-sion\-nel\-le}
\newcommand \zeds {z\'{e}\-ro-di\-men\-sion\-nels }
\newcommand \zedsz {z\'{e}\-ro-di\-men\-sion\-nels}
\newcommand \zedes {z\'{e}\-ro-di\-men\-sion\-nel\-les }
\newcommand \zedesz {z\'{e}\-ro-di\-men\-sion\-nel\-les}


\newcommand \cof {cons\-truc\-tif }
\newcommand \cofs {cons\-truc\-tifs }
\newcommand \cofz {cons\-truc\-tif}
\newcommand \cofsz {cons\-truc\-tifs}

\newcommand \cov {cons\-truc\-tive }
\newcommand \covz {cons\-truc\-tive}
\newcommand \covsz {cons\-truc\-tives}
\newcommand \covs {cons\-truc\-tives }

\newcommand \coma {\maths\covs}
\newcommand \comaz {\maths\covsz}
\newcommand \clama {\maths classiques }
\newcommand \clamaz {\maths classiques}

\renewcommand \cot {cons\-truc\-ti\-vement }
\newcommand \cotz {cons\-truc\-ti\-vement}

\newcommand \LLPO{{\bf LLPO}}

\newcommand \maths {mathé\-ma\-tiques }
\newcommand \mathsz {mathé\-ma\-tiques}
\renewcommand \math {mathé\-ma\-tique }
\newcommand \mathz {mathé\-ma\-tique}

\newcommand \prco {démonstration \cov}
\newcommand \prcos {démonstrations \covs}
\newcommand \prcoz {démonstration \covz}
\newcommand \prcosz {démonstration \covsz}

\newcommand \pte {principe du tiers exclu }

\newcommand \sdz {sans \dvz}
\newcommand \sdzz {sans \dvzz}

\newcommand \tcg {théorème de complé\-tude de G\"odel 
} 
\newcommand \tcgz {théorè\-me de com\-plé\-tude de G\"odel} 
\newcommand \acgz {axiome de com\-plé\-tude de G\"odel} 
\newcommand \Tcgi {Le \tcg impli\-que le ré\-sultat 
sui\-vant. }

\hyphenation{al-go-rith-mi-que
al-go-rith-mi-que-ment
cons-tant cons-tan-te cons-tants cons-tan-tes
cons-truc-tif cons-truc-ti-ve cons-truc-ti-ve-ment
exac-te exac-te-ment
ex-pli-ci-te ex-pli-ci-tes ex-pli-ci-te-ment ex-pli-ci-ter
ex-pli-ci-tons
ex-ten-sion ex-ten-sions 
intui-tif intui-tive intui-tion 
ma-xi-mal ma-xi-maux
res-pec-ti-ve-ment
uni-mo-du-lai-re
}

\thickmuskip = 7mu plus 2mu

\date{novembre 2007}
\title{ Anneaux seminormaux
\\
(d'après Thierry Coquand)}

\sibil{\author{Henri Lombardi  $^*$ , Claude Quitté $^\dag$}}
\sinotbil{\author{
Henri Lombardi (\thanks{~Equipe de Mathématiques, UMR CNRS 6623,
UFR des Sciences et Techniques, Université Marie et Louis Pasteur,
25030 BESANCON cedex, FRANCE,
email: {\tt henri.lombardi@umlp.fr}.}~), Claude Quitté (\thanks {~Laboratoire de Mathématiques, SP2MI, Boulevard 3, Teleport 2, BP 179, 86960 FUTUROSCOPE Cedex,
FRANCE, email: {\tt claude.quitte@orange.fr}}~) }
}

\pagestyle{headings}

\maketitle

\rdb
\label{beginfrench}

\begin{abstract}
 Le théorème de Traverso-Swan affirme qu'un anneau 
réduit $\gA$ est seminormal \ssi l'\homo naturel $\Pic \,\gA\to\Pic 
\,\gA[X]$ est un \iso (\cite{fTra,fSwan}). Nous exposons ici la \prco 
\elr de ce résultat qui a été donnée par Thierry Coquand 
dans~\cite{fcoq}.

Cet exemple est paradigmatique de la méthode \covz.
On obtient au bout du compte une démonstration plus simple que 
la démonstration classique initiale. Mais le plus important est que 
l'argument classique \gui{par l'absurde et au moyen d'un objet 
idéal} peut être décrypté selon une technique générale
qui s'inspire de la philosophie suivante: l'utilisation des objets
purement idéaux construits avec l'axiome du choix et le principe du 
tiers exclu peut être remplacée par celle d'objets concrets qui 
sont des approximations finies de ces objets idéaux. 
\end{abstract}

\medskip \noindent {\bf Mots clés}. Anneaux seminormaux, théorème de Traverso, \alg \covz, idéal premier minimal, méthode dynamique.
 

\smallskip \noindent {\bf MSC}. 03F65, 13F45, 13B40, 14Qxx

\newpage
\startcontents[french]

\setcounter{tocdepth}{4}
\markboth{Table des matières}{Table des matières}

\printcontents[french]{}{1}{}
\normalsize

\newpage

\section {Introduction}

\begin{flushright}
{\small
 Quant  à moi je proposerais de s'en tenir 
 aux règles 
suivantes:


\item 1. Ne jamais envisager que des objets susceptibles 
d'être définis 

en un nombre fini de mots;

\item 2. Ne jamais perdre de vue que toute proposition
 sur l'infini doit
  
être la traduction, l'énoncé abrégé 
de propositions sur le 
fini;

\item 3. \'Eviter les classifications et les définitions 
non 
prédicatives.

 
\medskip \rm Henri Poincaré,  

in  {\it La logique de l'infini } 
(Revue de Métaphysique et de Morale 1909). 

Réédité dans  {\it 
Dernières pensées}, Flammarion.
}
\end{flushright}

Le théorème de Traverso-Swan affirme qu'un anneau 
réduit $\gA$ est seminormal \ssi l'\homo naturel $\Pic \,\gA\to\Pic 
\,\gA[X]$ est un \iso (\cite{fTra,fSwan}). 

Nous exposons ici la \prco 
\elr de ce résultat qui a été donnée par Thierry Coquand 
dans~\cite{fcoq}.

La méthode utilisée consiste à mettre tout d'abord en place une 
démonstration classique la plus \elr possible. 
Après cette simplification, il reste des arguments 
hautement non \cofsz: démonstration par l'absurde basée sur la 
considération d'un \idep minimal.

Le décryptage se fait alors avec la \gui{méthode dynamique} qui 
permet de gérer à la fois le tiers exclu à l'{\oe}uvre dans le 
raisonnement par l'absurde et l'objet idéal que constitue l'\idep 
minimal générique présent dans la démonstration classique.

\smallskip Cet exemple est paradigmatique d'une méthode \cov mise au
point récemment, selon une technique générale
qui s'inspire de la philosophie suivante: l'utilisation des objets
purement idéaux construits avec l'axiome du choix et le principe du 
tiers exclu peut être remplacée par celle 
\emph{d'objets concrets qui 
sont des approximations finies de ces objets idéaux}.

L'histoire commence avec le système de calcul formel D5 
\cite{fD5} dans lequel est mis en évidence que l'on peut calculer dans la clôture algébrique d'un corps, même si l'on ne sait pas la construire comme un objet mathématique usuel. Ainsi était donnée une signification constructive claire
à l'objet idéal \gui{clôture algébrique}.

Dans l'article \cite{fCLR} est expliqué comment 
on peut interpréter les démonstrations abstraites des résultats de type Nullstellensatz obtenues via la théorie des modèles. Ici les objets idéaux sont les modèles d'une théorie formelle cohérente (ces modèles existent en vertu du tiers exclu et d'une version affaiblie de l'axiome du choix). 
Dans la démonstration devenue constructive, chacun de ces objets idéaux est remplacé par \gui{une information finie concernant l'objet idéal}.

Dans \cite{fcl,fCLR2}, les chaines d'\ideps qui interviennent dans la
définition abstraite de la dimension de Krull d'un anneau $\gA$ sont remplacées par des
suites finies d'\elts de l'anneau. 
Ainsi est obtenue une \dfn \cov \elr de la dimension de Krull,
dans laquelle les \ideps ont été totalement éliminés. 
Pour les anneaux usuellement utilisés en \maths la \dfn constructive de la dimension de Krull devient un outil algorithmique, même quand
ne sont pas disponibles les facilités apportées par les bases de Gr\"obner. En particulier certains grands \thos d'\alg commutative
qui utilisent la dimension de Krull ont été complètement décryptés \cot
dans \cite{fCoq3,fclq}. C'est le cas pour le \gui{splitting-off} de Serre, les \thos \gui{stable range} et \gui{de simplification} de Bass, et le \tho de Forster-Swan.
En outre la version \cov qui a été mise au point égale ou améliore les meilleures versions classiques de ces \thosz, obtenues par R. Heitmann dans son remarquable article \gui{non noethérien} de 1984~\cite{fHei84}.

Signalons enfin que dans \cite{fY1}, I. Yengui a montré comment éliminer l'utilisation des \idemas dans les démonstrations classiques pour les rendre \covs et a ainsi apporté un raffinement essentiel à la méthode dynamique.

\smallskip Dans l'exemple qui est traité ici, on obtient au bout du compte une démonstration élégante plus simple que 
la démonstration classique initiale. Mais le plus important est que 
l'argument classique \gui{par l'absurde et au moyen d'un objet 
idéal} peut être décrypté selon la méthode générale expliquée
ci-dessus. 
Le fait de considérer la localisation en un \idep minimal $\fp$ générique 
est remplacé par
un calcul arborescent où l'on essaie de rendre inversibles le maximum 
d'\elts qui se présentent comme obstacles à la démonstration.
L'arborescence provient du fait que dans la démonstration classique, on 
utilise un argument du type \gui{tout \elt $x$ est dans $\fp$ ou hors 
de $\fp$}. Comme l'\idep est minimal, a priori $x$ doit être hors de 
$\fp$, et ce n'est que lorsque le calcul montre que l'on a inversé $0$ 
qu'on revient en arrière pour ouvrir une autre branche du calcul.

\smallskip Dans la section \ref{fsec2} nous expliquons la transformation de démonstration mise en œuvre dans le cas intègre. 
Nous donnons en annexe
une démonstration détaillée du cas d'un anneau seminormal arbitraire.

\section{Préliminaires} \label{fsecprelim}
\markboth{Anneaux seminormaux}{1. Préliminaires}

Dans cet article $\gA$, $\gB$, $\gC$ désignent des anneaux commutatifs. 

Si l'on ne précise pas un \homo est un \homo d'anneaux.

\subsubsection*{Anneaux seminormaux} 
\addcontentsline{toc}{subsection}{Anneaux seminormaux}

Un anneau intègre $\gA$ est dit \emph{seminormal} si lorsque 
$b^2=c^3\neq 0$ alors
l'\elt $a=b/c$ du corps des fractions est en fait dans $\gA$. Notons 
que $a^3=b$ et $a^2=c$.

Un anneau quelconque $\gA$ est dit \emph{seminormal} si chaque fois 
que $b^2=c^3$, il existe $a\in\gA$ tel que $a^3=b$ et $a^2=c$.

Ceci implique que $\gA$  est réduit: si $b^2=0$  alors $b^2=0^3$, 
d'où un  $a\in\gA$ avec $a^3=b$ et $a^2=0$, donc $b=0$.

Dans un anneau si $x^2=y^2$ et  $x^3=y^3$ alors $(x-y)^3=0$.
Ainsi:

\begin{ffact} 
\label{ffactRed1} 
Dans un anneau réduit  $x^2=y^2$ et  $x^3=y^3$ impliquent $x=y$.
\end{ffact}

En conséquence le $a$ ci-dessus est toujours unique. En outre 
$\Ann\,b=\Ann\,c=\Ann\,a$.

\subsubsection*{Catégorie des \Amos \ptfs} 
\addcontentsline{toc}{subsection}{Catégorie des \Amos \ptfsz}

Un \mptf est un module $M$ isomorphe à un facteur direct dans un 
module libre de rang fini: $M\oplus M'\simeq\gA^m$. De manière 
équivalente, c'est un module isomorphe à l'image d'une matrice de 
\prnz.

Une \Ali $\psi:M\to N$ entre \mptfs avec  $M\oplus M'\simeq\gA^m$ et 
 $N\oplus N'\simeq\gA^n$ peut être représentée par 
$\wi{\psi}:\gA^m \to \gA^n$ définie par $\wi{\psi}(x\oplus 
x')=\psi(x)$.

En d'autres termes la catégorie des \mptfs sur $\gA$ est 
équivalente à la catégorie dont les objets sont les matrices 
carrées \idmes à \coes dans $\gA$, un morphisme
de $P$ vers $Q$ étant une matrice $H$ de format convenable telle que 
$QH=H=HP$. En particulier l'identité de $P$ est représentée par 
$P$.

\begin{ffact} 
\label{ffactcatmptf} 
Si les \mptfs $M$ et $N$ sont représentés par les matrices \idmes  
$P=(p_{i,j})_{i,j\in I}\in \gA^{I\times I}$ et 
$Q=(q_{k,\ell})_{k,\ell\in J}\in \gA^{J\times J}$, alors:
\begin{enumerate}
\item 
La somme directe $M\oplus N$ est représentée par 
 $\Diag(P,Q)=\bloc{P}{0}{0}{Q}$.
\item Le produit tensoriel
 $M\otimes N$ est représenté par le produit de Kronecker
 $P\otimes Q=(r_{(i,k),(j,\ell)})_{(i,k),(j,\ell)\in I\times J}$, où 
$r_{(i,k),(j,\ell)}=p_{i,j}q_{k,\ell}$.
\item \label{item3ffactcatmptf} $M$ et $N$ sont isomorphes \ssi les matrices $\Diag(P,0_n)$ et  
$\Diag(0_m,Q)$ sont semblables
\end{enumerate}
\end{ffact}
Le dernier point se vérifie en remarquant que
la projection sur $M$ dans $M\oplus M'\oplus \gA^n$ est 
représentée par la matrice   $\Diag(P,0_n)$  tandis que la  
projection sur $N$ dans $\gA^m\oplus N\oplus N'$ est représentée 
par la matrice   $\Diag(0_m,Q)$, et en décomposant 
$\gA^m\oplus\gA^n$ sous la forme
$M\oplus M'\oplus N \oplus N'$ on voit que les deux projections sont 
conjuguées par l'\auto qui échange~$M$ et~$N$.

\subsubsection*{Rang d'un \mptf} 
\addcontentsline{toc}{subsection}{Rang d'un \mptf}
Si $\varphi :M\to M$ est un \endo du \Amo \ptf $M$ image de la
matrice \idme $P\in\gA^{n\times n}$ et si $H\in\gA^{n\times n}$ représente $\varphi$  (avec $H=PH=HP$),
notons $N=\Ker\,P$ de sorte que $M\oplus N =\gA^n$.
Alors on peut définir le \emph{déterminant} de $\varphi$ par 
\[\det(\varphi)=\det(\varphi \oplus \Id_N)=\det(H+(\In-P)).\]

Soit $\mu_X$ la multiplication par $X$  dans le $\gA[X]$-module $M[X]$. 
Ce module, étendu de $M$  depuis~$\gA$, est  représenté 
par la matrice $P$ vue comme \elt de $\gA[X]^{n\times n}$. Alors $\det(\mu_X)=\mathrm{R}_M(X)=r(X)$ est
un \pol qui vérifie $r(XY)=r(X)r(Y)$ et $r(1)=1$. En d'autres termes 
ses \coes forment un \sfioz. Le module est dit de rang $k$ si 
$r(X)=X^k$.

Un calcul direct montre le fait suivant.

\begin{ffact} 
\label{ffactprc1} 
Une matrice $P=(p_{i,j})$ a pour image un \mrc 1 \ssi les deux 
propriétés suivantes sont vérifiées
\begin{itemize}
\item  $\Vi^2\,P=0$, \cad tous les mineurs d'ordre 2 sont nuls,
\item  $\Tr\,P=\sum_{i}p_{ii}=1$.
\end{itemize}
\end{ffact}

\subsubsection*{Quand l'image d'une matrice de projection est libre} 
\addcontentsline{toc}{subsection}{Quand l'image d'une matrice de 
projection est libre}

Si $P\in\gA^{n\times n}$ est une matrice de \prn dont l'image est 
libre
de rang $r$, son noyau n'est
pas automatiquement libre, et la matrice n'est donc pas à tout coup 
semblable
à la matrice de projection standard \[\I_{n,r}=\Diag(\I_{r},0_{n-r})= 
\bloc{\I_{r}}{0} {0}{0_{n-r}}.\]

Donnons une caractérisation simple pour
le fait que l'image d'une matrice \idme est libre.

\begin{fproposition} 
\label{fpropImProjLib} 
Soit $P\in\gA^{n\times n}$. La matrice 
$P$ est \idme et d'image libre de rang $r$  \ssi
il existe deux matrices $X\in\gA^{n\times r}$ et  $Y\in\gA^{r\times 
n}$ telles que
$YX=\I_r$ et $P=XY$. En outre,
\begin{enumerate}
\item  $\Im\,P=\Im\,X\simeq \Im\,Y$.
\item  Pour toutes matrices $X',Y'$ de mêmes formats que $X$  et $Y$ 
et
telles que $P=X'Y'$, il existe
une unique matrice   $U\in\GL_r(\gA)$ telle $X'=XU$ et $Y=UY'$. En 
fait
$U=YX'$, $U^{-1}=Y'X$, $Y'X'=\I_r$ et les colonnes de $X'$ forment une 
base de~$\Im\,P$.
\end{enumerate}
Une autre \carn possible est la suivante: la matrice $\Diag(P,0_r)$ 
est
semblable à la matrice de \prn standard $\I_{n+r,r}$.
\end{fproposition}
\begin{proof}
Supposons que $\Im\,P$ est libre de rang $r$.
Pour colonnes de $X$ on prend une base de $\Im\,P$. Alors, il existe 
une unique
matrice $Y$ telle que $P=XY$. Puisque $PX=X$ (car $P^2=P$) on~a 
$XYX=X$.
Puisque les colonnes de $X$ sont indépendantes et que $(\I_r-YX)X=0$ 
on a $\I_r=YX$.

\noindent Supposons $YX=\I_r$ et $P=XY$. Alors 
\[P^2=XYXY=X\I_rY=XY=P\;\hbox{ et }\;PX=XYX=X.\] 
Donc $\Im\,P=\Im\,X$. En outre les colonnes de $X$ sont 
indépendantes car $XZ=0$ implique \hbox{$Z=YXZ=0$}.  

\smallskip \noindent \emph{1.} La suite $\gA^n\vers{\In-P}\gA^n\vers{Y}\gA^r$ est 
exacte: en effet
$Y(\In-P)=0$ et si $YZ=0$ alors $PZ=0$ donc $Z=(\In-P)Z$. Ainsi 
$\Im\,Y\simeq \gA^n\sur{\Ker\,Y}=\gA^n\sur{\Im(\In-P)\simeq \Im\,P}$.

\smallskip \noindent \emph{2.} Si maintenant  $X'$ et $Y'$ sont de mêmes formats que $X$ et $Y$ 
et si $P=X'Y'$, on pose $U=YX'$ et \hbox{$V=Y'X$}.
Alors $UV=YX'Y'X=YPX=YX=\I_r$; $X'V=X'Y'X=PX=X$. Donc \hbox{$X'=XU$};
$UY'=YX'Y'=YP=Y$, et $Y'=VY$. Enfin $Y'X'=VYXU=VU=\I_r$.

\smallskip \noindent  Concernant la dernière \carn il s'agit d'une simple application 
du point \emph{\ref{item3ffactcatmptf}} dans le fait~\ref{ffactcatmptf}.
\end{proof}

Nous résumons la situation pour les modules \prcs 1.

\begin{flemma} 
\label{flempropImProjLib} 
 Une matrice de \prn de rang $1$, $P$, a son image libre \ssi il 
existe un vecteur colonne $x$ et un vecteur ligne $y$ tels que $yx=1$ 
et $xy=P$. En outre $x$ et $y$  sont uniques, au produit par une 
unité près, sous la seule condition que $xy=P$. 
\end{flemma}

     
\subsubsection*{Le semi anneau de Grothendieck 
\texorpdfstring{$\GKO\,\gA$ et le groupe de Picard $\Pic\,\gA$}{GKO(A) 
et le groupe de Picard Pic(A)}} 
\addcontentsline{toc}{subsection}{GK0(A) et Pic(A)}

$\GKO\,\gA$ est l'ensemble des classes d'\iso de \mptfs sur $\gA$.
C'est un semi anneau pour les lois héritées de $\oplus$ et 
$\otimes$. 

Puisque $\gA$ est supposé commutatif, le sous semi anneau de  
$\GKO\,\gA$ engendré par $1$ (la classe du \mptf $\gA$) est 
isomorphe à $\NN$, sauf dans le cas où $\gA$ est l'anneau trivial. 

Tout \elt de $\GKO\,\gA$ peut être représenté par une matrice 
\idme à \coes dans~$\gA$.

$\Pic\,\gA$ est le sous ensemble de $\GKO\,\gA$ formé par les
classes d'\iso des \mrcs 1.
Il s'agit d'un groupe pour la multiplication. L'\gui{inverse} de $M$ 
est le dual de $M$. Si $M\simeq \Im\,P$, alors $M\sta\simeq 
\Im\,\tra{P}$. En particulier,
si $P$ est une matrice de \prn de rang 1, $P\otimes \tra{P}$ est une
matrice de \prn dont l'image est un module libre de rang 1.

On peut d'ailleurs vérifier directement cette propriété en 
utilisant la \carn donnée au lemme~\ref{flempropImProjLib}. 

\subsubsection*{Rapport entre \texorpdfstring{$\Pic\,\gA$}{Pic(A)} et 
les classes d'idéaux inversibles} 
\addcontentsline{toc}{subsection}{Rapport entre Pic(A) et les classes 
d'idéaux inversibles}

Un \id $\fa$ de $\gA$  est dit \emph{inversible} s'il existe un \id 
$\fb$ tel que $\fa\fb=a\gA$ où $a$ est un \elt régulier.
Dans ce cas il existe $\xn$ et $\yn$ dans $\gA$ tels que 
$\fa=\gen{\xn}$, \hbox{$\fb=\gen{\yn}$} et $\sum_ix_iy_i=a$. En outre
pour tous $i,j$ il existe un unique $m_{i,j}$ tel que 
\hbox{$y_ix_j=am_{i,j}$} 
On en déduit que la matrice $(m_{i,j})$ est une matrice \idme de rang 1, 
et que son image est isomorphe à $\fa$ en tant que \Amoz.

Deux \ids inversibles $\fa,\fb$ sont isomorphes en tant que \Amos \ssi 
il existe~$a$ et~$b\in\gA$ réguliers tels que $a\fa=b\fb$. Ceci permet de 
définir le groupe des classes d'idéaux inversibles comme 
sous-groupe de $\Pic\,\gA$. En fait la plupart du temps les deux 
groupes coïncident.

Par exemple si $\gA$ 
est intègre%
,
toute matrice $(a_{i,j})$ \idme de rang $1$ a un \elt régulier sur 
la diagonale et les \coes de la ligne correspondante engendrent un \id 
inversible isomorphe à l'image de la matrice.

\subsubsection*{Changement d'anneau de base} 
\addcontentsline{toc}{subsection}{Changement d'anneau de base}

Si l'on a un \homo $\gA\vers{\rho}\gB$, l'extension des scalaires de 
$\gA$ à $\gB$ transforme un
\mptf $M$  sur $\gA$ en un \mptf $\rho\ista(M)$ sur $\gB$. 
Tout \Bmo isomorphe à un tel module $\rho\ista(M)$ est dit 
\gui{étendu} depuis $\gA$.

Du point de vue
matrices de \prnz, cela correspond à considérer la  matrice
transformée par l'\homo $\rho$.

Cela donne un \homo $\GKO\,\rho: \GKO\,\gA\to\GKO\,\gB$.
D'où le \pb qui se pose naturellement: \gui{tout \mptf sur $\gB$ 
provient-il d'un \mptf sur $\gA$?}.
Ou encore: \gui{$\GKO\,\rho$ est il surjectif?}.

Par exemple si $\gZ$ est le sous anneau de $\gA$  engendré par 
$1_\gA$,
on sait que tous les \mrcs sur $\gZ$ sont libres, et la question 
\gui{les \Amos \prcs sont-ils tous étendus depuis $\gZ$?} est 
équivalente à
\gui{tous les \Amos \prcs sont-ils libres?}.

Dans le cas $\gB=\AXm=\AuX$, on a de plus l'\homo d'évaluation en 0,  
$\gB\vers{\theta}\gA$, avec $\theta\circ\rho=\Id_\gA$. 
On en déduit que le \Bmo \ptf $M=M(\uX)$ est étendu \ssi il est 
isomorphe à $M(0)=\theta\ista(M)$. 

En ce qui concerne les matrices de \prnz,
une matrice \idme  $P\in\gB^{n\times n}$ représente un module 
étendu depuis $\gA$ \ssi son image est isomorphe à l'image de 
$P(0)$. 

Si tous les \Bmos \ptfs sont étendus depuis $\gA$ alors $P$ doit 
être semblable à $P(0)$, mais ceci peut s'avérer plus difficile 
à démontrer directement que l'\iso des images. 

Concernant les $\Pic$ on a les deux \homos de groupe
$\Pic\,\gA\vers{\Pic\,\rho}\Pic\,\AuX\vers{\Pic\,\theta}\Pic\,\gA$
qui se composent selon l'identité. Le premier est injectif, le 
second surjectif, et ce sont des \isos \ssi le premier est surjectif, 
\ssi
le second est injectif.

Cette dernière propriété signifie: toute matrice $P(\uX)$ \idme 
de rang 1 sur $\AuX$, vérifiant \gui{$\Im(P(0))$ est libre}, 
vérifie 
elle-même  \gui{$\Im(P(\uX))$ est libre}.

En fait si  $\Im(P(0))$ est libre, alors la matrice diagonale par 
blocs  $\Diag(P(0),0_1)$ est semblable à une matrice de \prn 
standard $\I_{n,1}$. Comme $\Im(\Diag(P(\uX),0_1))$ est isomorphe à 
$\Im\,P(\uX)$, on obtient le résultat qui suit.

\begin{flemma} 
\label{flemPicPic1} \Propeq
\begin{enumerate}
\item L'\homo naturel $\Pic\,\gA\to\Pic\,\AuX$ est un \isoz,
\item Pour toute
matrice  $P(\uX)\in\AuX^{n\times n}=(m_{i,j}(\uX))_{i,j\in 1,\ldots 
,n}$  \idme de rang 1 vérifiant $P(0)=\I_{n,1}$, il existe 
$f_1,\ldots, f_n, g_1,\ldots, g_n\in\AuX$ tels que $m_{i,j}=f_ig_j$ 
pour tous $i,j$.
\end{enumerate}
\end{flemma}

\subsubsection*{Seuls importent les anneaux réduits: 
\texorpdfstring{$\GKO\,\Ared = \GKO\,\gA$}{GKO(Ared)=GK0(A)}} 
\addcontentsline{toc}{subsection}{Seuls importent les anneaux 
réduits}

Nous notons $\Ared$ l'anneau réduit associé à $\gA$, \cad 
$\gA\sur{\sqrt{0}}$.
\begin{fproposition} 
\label{fpropComparRed} 
L'application naturelle $\GKO(\gA)\to\GKO(\Ared)$ est bijective. 
\begin{enumerate}
\item Injectivité:  cela signifie que si deux \mptfs $E,F$  sur 
$\gA$ sont isomorphes sur $\Ared$, ils le sont \egmt sur $\gA$.
\item De manière plus précise si
deux matrices \idmes $P,Q$ de même format sont conjuguées sur 
$\Ared$, elles le sont \egmt sur $\gA$, via un \iso qui relève 
l'\iso
de conjugaison résiduel.
\item Surjectivité:  tout \mptf sur $\Ared$ provient d'un \mptf 
sur~$\gA.$ 
\end{enumerate}
\end{fproposition}
\begin{proof}
\emph{2.}  On note $\ov{x}$ l'objet $x$  vu modulo $\sqrt{0}$. Soit 
$C\in\gA^{n\times n}$ une matrice telle que $\ov{C}\, \ov{P}\, {\ov{C}}^{-1}=\ov{Q}$. Puisque $\det(C)$ est inversible dans $\Ared$, 
il est inversible dans $\gA$ et $C\in\GL_n(\gA)$. On a donc 
$\ov{Q}=\ov{C\,P\,C^{-1}}$. Quitte à remplacer $P$ 
par $C\,P\,C^{-1}$ on peut supposer $\ov{Q}=\ov{P}$. Alors $PQ$ code 
une \Ali
de $\Im\,P$ vers $\Im\,Q$ qui donne résiduellement l'identité. De 
même  $(\In-P)(\In-Q)$ code une \Ali de  $\Ker\,P$ vers $\Ker\,Q$ 
qui donne résiduellement l'identité. On considère alors la 
matrice $A=PQ+(\In-P)(\In-Q)$ qui réalise $AQ=PQ=PA$ (\vfn 
immédiate) et $\ov{A}=\In$: ainsi $A$ est inversible et relève 
l'\iso de conjugaison résiduel.

\sni \emph{1.} Pour deux \mptfs résiduellement isomorphes $E\simeq\Im\,P$ et 
$F\simeq\Im\,Q$ on  réalise $E$ et $F$ comme images de matrices 
\idmes de même format et résiduellement conjuguées: 
$\Diag(P,0_m)$ et  $\Diag(0_n,Q)$ avec $\Diag(\ov{P},0_m)$ semblable 
à $\Diag(0_n,\ov{Q})$ (voir le fait \ref{ffactcatmptf}). Puis on applique le point \emph{1.}

\sni \emph{3.} On a la possibilité de relever tout \mptf gr\^ace à la 
méthode de Newton. Plus précisément soit $\fa$ l'\id engendré 
par les \coes de $P^2-P$. Si $\fa$ est contenu dans le nilradical de~$\gA$, il existe $k$ tel que $\fa^{2^k}=0$. Par ailleurs si $Q=3P^2-
2P^3$, alors $Q\equiv P \mod \fa$ et $Q^2-Q$ est multiple de 
$(P^2-P)^2$ donc a ses \coes dans $\fa^2$. Il suffit donc d'itérer 
$k$ fois l'affectation $P\leftarrow 3P^2-2P^3$ pour obtenir le 
résultat souhaité.
\end{proof}

\begin{fcorollary} 
\label{fcorpropComparRed} 
L'\homo canonique $\Pic\,\gA\to \Pic\,\AuX$ est un \iso \ssi l'\homo 
canonique  $\Pic\,\Ared\to \Pic\,\Ared[\uX]$ est un \isoz.
\end{fcorollary}

\begin{fconvention} 
\label{fconvPicPic} 
Dans la suite nous abrégeons la phrase \gui{l'\homo canonique 
$\Pic\,\gA\to \Pic\,\AuX$ est un \isoz} en disant (par abus) 
\gui{$\Pic\,\gA=\Pic\,\AuX$}.
\end{fconvention}

\subsubsection*{\'Eléments inversibles de 
\texorpdfstring{$\gA[\protect\underline{X}]$}{A[X]}} 
\addcontentsline{toc}{subsection}{\'Eléments inversibles de 
$\gA[X]$}

\begin{flemma} 
\label{flemUnitRedX} 
Si $\gA$ est réduit l'\homo de groupes 
$\gA^{\times}\to(\AuX)^{\times}$ est un \isoz.
Autrement dit si $f(\uX)\in\AuX$ est inversible, alors 
$f=f(0)\in\gA^{\times}$.
\end{flemma}

Il suffit de faire la démonstration en une variable, et elle résulte d'un 
calcul direct: si $f(X)g(X)=1$ avec $\deg(f)\leq m$, $m\geq 1$, on 
montre que le \coe de degré $m$  dans $f$  est nilpotent.
\subsubsection*{Le \tho de Kronecker} 
\addcontentsline{toc}{subsection}{Le \tho de Kronecker}

\begin{ftheorem} 
\label{fthKro} 
Soient $f,g\in\AuX$ et $h=fg$. Soit $a$ un \coe de $f$  et $b$ un \coe 
de $g$, alors $ab$ est entier sur le sous anneau de $\gA$ engendré 
par les \coes de $h$. 
\end{ftheorem}

En utilisant \gui{l'astuce de Kronecker} (remplacer chaque variable 
$X_k$ par $T^{m^k}$ pour un $m$ suffisamment grand) il suffit de le 
montrer pour des \pols en une variable.
Avec deux \pols de degré~$1$ en une variable, on voit le résultat 
à l'{\oe}il nu. Avec deux \pols de degré 2, on voit que ce n'est 
pas si simple. Néanmoins des \prcos existent dans la littérature
(cf. \cite{fEd,fHu}, et pour un  article de synthèse \cite{fCDLQ}).

\section{Théorème de Traverso-Swan. Le cas intègre.} 
\label{fsec2}
\markboth{Anneaux seminormaux}{2. Théorème de Traverso-Swan. Le cas intègre.}

\subsubsection*{La condition est nécessaire: l'exemple de Schanuel} 
\addcontentsline{toc}{subsection}{L'exemple de Schanuel}

On montre que si $\gA$ est réduit et $\Pic\,\gA=\Pic\,\gA[X]$ alors 
$\gA$ est seminormal. 
On utilise la \carn donnée dans le lemme \ref{flempropImProjLib}.

Soient $b,c\in\gA$ réduit avec $b^2=c^3$. Soit $\gB=\gA[a]=\gA+a\gA$ 
un anneau réduit contenant $\gA$ avec $a^3=b, \,a^2=c$.
On considère $f_1=1+aX$, $f_2=cX^2=g_2$ et $g_1=(1-aX)(1+cX^2)$.
On~a $f_1g_1+f_2g_2=1$, donc  la matrice $M(X)$ des $f_ig_j$ est \idme 
de rang $1$.
On vérifie alors sans peine que ses \coes sont dans $\gA$ et que 
$M(0)=\I_{2,1}$.
Son image est libre sur $\gB[X]$. Si elle est libre sur~$\gA[X]$ 
il existe des $f'_i$ et $g'_j$ dans $\gA[X]$ avec $f'_ig'_j=f_ig_j$.
Par unicité $f'_i=uf_i$ avec $u$ inversible dans~$\gA[X]$ donc dans 
$\gA$. Avec $i=1$  on obtient $a\in\gA$.

\noi NB: pour $\gB$ on peut prendre $\left(\aqo{\gA[T]}{T^2-c, T^3-b} 
\right)\red$. Si un $a$ est déjà présent dans $\gA$, on obtient 
par unicité $\gB=\gA$.
\subsubsection*{Cas d'un anneau à pgcd} 
\addcontentsline{toc}{subsection}{Cas d'un anneau à pgcd}

Rappelons qu'un anneau (intègre) à pgcd est un anneau dans lequel
deux \elts arbitraires admettent un plus grand commun diviseur,
\cad une borne inférieure pour la relation de divisibilité.
Rappelons aussi que si $\gA$ est un anneau à pgcd, il en va de 
même pour l'anneau des \polsz~$\AuX$. 
\begin{flemma} 
\label{flemPicGcd} 
Si $\gA$ est un anneau intègre à pgcd, $\Pic\,\gA=\so{1}$.
\end{flemma}
\rem En conséquence $\Pic\,\gA\to \Pic\,\AuX$ est un \isoz.
Notez que le résultat s'applique si $\gA$  est un corps discret.

\begin{proof}
On utilise la \carn donnée dans le lemme \ref{flempropImProjLib}. 
Soit $P=(m_{i,j})$ une matrice \idme de rang 1. 
Puisque $\sum_i m_{i,i}=1$ on peut supposer que $m_{1,1}$ est 
régulier.
Soit $f$ le pgcd des \elts de la première ligne. On a $m_{1,j}=fg_j$
avec le pgcd des $g_j$ égal à 1. Puisque $f$ est régulier et 
$m_{1,1}m_{i,j}=m_{1,j}m_{i,1}$ on obtient $g_1m_{i,j}=m_{i,1}g_j$.
Ainsi $g_1$ divise tous les $m_{i,1}g_j$ donc aussi leur pgcd 
$m_{i,1}$.
On écrit $m_{i,1}=g_1f_i$. Puisque $g_1f_1=m_{1,1}=fg_1$ cela donne 
$f_1=f$. Enfin l'\egt $m_{1,1}m_{i,j}=m_{1,j}m_{i,1}$ donne 
$f_1g_1m_{i,j}=f_1g_jg_1f_i$ puis $m_{i,j}=f_ig_j.$
\end{proof}

\subsubsection*{Cas d'un anneau intègre normal} 
\addcontentsline{toc}{subsection}{Cas d'un anneau intègre normal}
\begin{flemma} 
\label{flemIntegclos} 
Si $\gA$ est intègre et intégralement clos, alors 
$\Pic\,\gA=\Pic\,\AuX$.
\end{flemma}
\begin{proof}
On utilise la \carn donnée au lemme \ref{flemPicPic1}. Soit 
$P(\uX)=(m_{i,j}(\uX))_{i,j=1,\ldots ,n}$ une matrice \idme de rang 
$1$  avec
$P(0)=\I_{n,1}$. Soit $\gK$ le corps des fractions de $\gA$.
Sur $\gK[\uX]$ le module $\Im\,P(\uX)$ est libre et il existe donc 
$f=(f_1(\uX),\ldots ,f_n(\uX))$ et $g=(g_1(\uX),\ldots ,g_n(\uX))$ 
dans
$\gK[\uX]^n$ tels que $m_{i,j}=f_ig_j$ pour tous $i,j$.
En outre puisque $f_1(0)g_1(0)=1$ et puisqu'on peut modifier $f$ et 
$g$ en les multipliant par une unité, on peut supposer que 
$f_1(0)=g_1(0)=1$. Alors puisque $f_1g_j=m_{1,j}$ et vu le \tho de 
Kronecker, les \coes des $g_j$ sont entiers sur l'anneau engendré 
par les \coes des $m_{1,j}$. De même les \coes des $f_i$ sont 
entiers sur l'anneau engendré par les \coes des $m_{i,1}$.
Mais  on suppose $\gA$  intégralement clos, donc les $f_i$ et les 
$g_j$ sont dans $\gA[\uX]$. 
\end{proof}

\subsubsection*{Cas d'un anneau intègre seminormal} 
\addcontentsline{toc}{subsection}{Cas d'un anneau intègre 
seminormal}

Traverso \cite{fTra} avait démontré le théorème dans le cas d'un anneau noethérien réduit
$\gA$ (avec une restriction supplémentaire). Pour le cas intègre sans hypothèse 
noethérienne on peut 
consulter \cite{fQuerre,fBC,fGH}. 

\begin{ftheorem} 
\label{fpropIntSemin} 
Si $\gA$ est intègre et seminormal, alors $\Pic\,\gA=\Pic\,\AuX$.
\end{ftheorem}
\begin{proof}
On commence la démonstration comme celle du lemme \ref{flemIntegclos}.
On a $f_1(\uX),\ldots ,f_n(\uX)$, $g_1(\uX),\ldots ,g_n(\uX)$ dans
$\gK[\uX]^n$ tels que $m_{i,j}=f_ig_j$ pour tous $i,j$. En outre 
$f_1(0)=g_1(0)=1$. On appelle $\gB$ le sous anneau de $\gK$ engendré 
par $\gA$ et par les \coes des $f_i$ et des $g_j$. Alors,
vu le \tho de Kronecker,~$\gB$ est une extension finie de 
$\gA$ (i.e., $\gB$ est un \Amo \tfz). Notre but est de montrer que
$\gA=\gB$.
On appelle  $\fa$ le conducteur de $\gA$ dans $\gB$, \cad
l'ensemble $\so{x\in\gB\,|\,x\gB\subseteq\gA}$. C'est à la fois un 
\id de $\gA$ et de $\gB$. Notre but est maintenant de montrer 
$\fa=\gen{1}$, \cade que $\gC=\gA\sur{\fa}$ est trivial.
Nous avons besoin de lemmes préparatoires.
\begin{flemma} 
\label{flemIntSemin1} 
Si $\gA\subseteq\gB$, $\gA$ seminormal et $\gB$ réduit, alors le  
conducteur $\fa$ de $\gA$ dans $\gB$ est un \id radical de $\gB$.
\end{flemma}
\begin{Proof}{Démonstration du lemme \ref{flemIntSemin1}.}
On doit montrer  que si $u\in\gB$ et $u^2\in\fa$ alors $u\in\fa$. Soit 
donc $c\in\gB$, on doit montrer que  $uc\in\gA$. On sait que 
$u^2c^2\in\gA$. Mais aussi
 $u^3c^3=u^2(uc^3)\in\gA$ puisque $u^2\in\fa$. 
Puisque $(u^3c^3)^2=(u^2c^2)^3$ il existe $a\in\gA$ tel que 
 $a^2=(uc)^2$ et $a^3=(uc)^3$. Comme $\gB$ est réduit cela implique 
$a=uc$, et donc $uc\in\gA$. 
\end{Proof}

\rem La \emph{clôture seminormale} d'un anneau 
$\gA$ dans un suranneau réduit $\gB$ est obtenue en partant de $\gA$ et en ajoutant 
les \elts $x$ de $\gB$  tels que $x^2$ et $x^3$ sont dans l'anneau préalablement construit. Notez que par le fait \ref{ffactRed1}, $x$ est uniquement déterminé par la donnée de $x^2$ et $x^3$.  
La démonstration du lemme précédent peut alors être interprétée comme
une démonstration de la variante suivante.

\begin{flemma} 
\label{flemIntSemin1bis} 
Soient $\gA\subseteq\gB$ réduit, $\gA_{1}$ la
clôture  seminormale de  $\gA$ dans $\gB$, et  
$\fa$ le conducteur de $\gA_{1}$ dans $\gB$. 
Alors $\fa$ est un idéal radical de
$\gB$.
\end{flemma}

\begin{flemma} 
\label{flemIntSemin2} 
Soient $\gA\subseteq\gB$,  $\gB=\gA[\cq]$ réduit fini sur $\gA$ et  
$\fa$ le conducteur de $\gA$ dans $\gB$. On suppose que $\fa$ est un 
idéal radical. Alors $\fa$ est 
égal à $\so{x\in\gA\,|\,xc_1,\ldots ,xc_q\in\gA}$.
\end{flemma}
\begin{Proof}{Démonstration du lemme \ref{flemIntSemin2}.}
En effet si $xc_i\in\gA$ alors  $x^\ell c_i^\ell\in\gA$ pour tout 
$\ell$, et donc pour un $N$ assez grand
$x^N y\in\gA$ pour tout $y\in\gB$, donc $x$ est dans le radical de 
$\fa$ (si $d$ majore les degrés des équations de dépendance 
intégrale des $c_i$ sur $\gA$, on pourra prendre $N=(d-1)q$).
\end{Proof}
\emph{La fin de la démonstration du \tho \ref{fpropIntSemin}}
 est maintenant  donnée en \clamaz.\\
Supposons au contraire que $\fa\neq \gen{1}$. 
On a $\gC=\gA\sur{\fa}\subseteq \gB\sur{\fa}=\gC'$.
Soit alors $\fp$ un \idemi de 
$\gC$, $\fP$ l'idéal correspondant de $\gA$,
$S=\gC\setminus\fp$  la partie complémentaire. 
Puisque $\fp$  
est un \idemiz, et puisque   $\gC$ est réduit, 
$S^{-1}\gC=\gL$ est un corps,
contenu dans l'anneau réduit  $S^{-1}\gC'=\gL'$.\\
Si  $x$ est un objet défini sur $\gA$ notons $\ov{x}$ ce qu'il 
devient après le changement de base  $\gA\to\gL'$.
Le module $\ov{M}$ est défini par la matrice $\ov{P}$ dont les \coes 
sont dans $\gL[\uX]$. Puisque $\gL$ est un corps, le module $\Im\,\ov{P}$ est libre sur $\gL[\uX]$. Cela implique, par unicité (lemme 
\ref{flempropImProjLib}) et vu que $f_1(0)=g_1(0)=1$,
que les~$\ov{f_i}$ et $\ov{g_j}$ sont dans $\gL[\uX]$
(si $u(X)\in\gL[\uX]$ est inversible et $u(0)=1$, alors $u=1$).
Cela signifie qu'il existe $s\in \gA\setminus \fP$  tel que les $sf_i$ et $sg_j$ sont 
à \coes dans $\gA$. D'après le lemme \ref{flemIntSemin2}, ceci 
implique que $s\in\fa$, ce qui est absurde.
\end{proof}

La démonstration donnée ci-dessus pour le \tho  \ref{fpropIntSemin} 
est une simplification des démonstrations existantes dans la littérature. 
Elle n'est cependant pas totalement
\cov et elle ne  traite que le cas intègre.
\subsubsection*{Démonstration \cov (cas seminormal intègre)} 
\addcontentsline{toc}{subsection}{Démonstration \cov }

Nous allons donner maintenant une \prco du \thoz~\ref{fpropIntSemin}.

On commence par remarquer que l'argument par l'absurde dans la démonstration 
classique, peut être interprété comme un argument indirect,
qui prouve que l'anneau $\gA\sur{\fa}$  est trivial en disant, selon 
toute apparence: si l'anneau n'était pas trivial etc\ldots, il 
serait trivial. Mais une fois remis à l'endroit, l'argument prouve 
directement que l'anneau voulu est trivial. On pourra lire à ce 
sujet le petit article de Richman sur l'anneau trivial (\cite{fRic}).

Outre cette remarque plutôt anodine (le renversement d'une démonstration 
directe en une démonstration par l'absurde est très banal en \clamaz),
il nous faut un lemme qui permet d'éliminer l'usage de l'\idep 
premier minimal \emph{purement idéal} qui intervient dans la démonstration 
classique. Dans le processus de décryptage, ceci est le point le 
plus délicat.   

Ce lemme dont l'énoncé est un peu déroutant a la signification 
intuitive suivante: \\
\emph{Soit $\gC$ un anneau réduit et $P$  un module \pro
de rang 1 sur $\gC[\uX]$; si $\gC$ n'est pas trivial, il doit y avoir 
une localisation non triviale $S^{-1}\gC$ de $\gC$ pour laquelle  $P$ 
devient libre.}

En \clama la réponse est immédiate: la localisation en un \idep 
minimal. C'est l'argument qui a été utilisé dans la démonstration  du 
cas intègre, avec l'anneau $\gC=\gA/\fa$.

Le lemme sous sa forme intuitive 
\gui{n'est pas vrai} d'un point de vue \cofz.
Mais fort heureusement c'est sa contraposée qui nous intéresse:\\
\emph{Soit $\gC$ un anneau réduit et $P$  un module \pro
de rang 1 sur $\gC[\uX]$; si toute localisation  $S^{-1}\gC$ de~$\gC$ 
pour laquelle  $P$ devient libre est triviale, c'est que $\gC$ 
lui-même est trivial.}\\
Et elle \gui{est vraie} au sens des \comaz, \cad 
qu'elle nous donne un \algoz!  

En fait nous utiliserons la version précise suivante dans laquelle
seules interviennent des localisations en un seul élément.

Voici LE lemme crucial.  

\begin{flemma} 
\label{flemThierry} \emph{(lemme d'élimination de l'idéal premier 
minimal)}\\
Soit $\gC$  un anneau réduit et $P=(m_{i,j})\in\gC[\uX]^{n \times 
n}$ une matrice \idme de rang 1 
telle que $P(0)=\I_{n,1}$. Supposons que l'implication suivante soit 
satisfaite: 

\smallskip \centerline{$\forall a\in\gC$, si $\Im\,P$ 
est libre sur $\gC[1/a][\uX], $ alors $a=0$.}

\smallskip \noindent  Alors $\gC$ est trivial, \cad $1=0$  dans $\gC$.
\end{flemma}
\begin{Proof}{Démonstration que le lemme \ref{flemThierry} implique le 
\thoz~\ref{fpropIntSemin}. }
Nous pouvons reprendre à très peu près la fin
de la démonstration du \thoz~\ref{fpropIntSemin}, qui utilisait un \idep 
minimal $\fp$. La lectrice constatera que gr\^ace AU lemme, on  
remplace simplement la localisation en $\fp$ par la localisation en un 
\elt $a$. \\
On reprend la démonstration du \tho à l'endroit où elle devenait non 
\covz. On a $\gC=\gA\sur{\fa}\subseteq \gB\sur{\fa}=\gC'$, 
deux anneaux réduits.
Pour montrer que  $\gC$ est 
trivial, il suffit de montrer que  $\gC$ vérifie, avec la matrice 
$P$ mod $\fa$, les hypothèses DU lemme.\\
Considérons donc $a\in\gA$ tel que  $\Im\,P$ 
soit libre sur $\gC[1/a][\uX]$. \\
Notons $\gC[1/a]=\gL\subseteq  \gC'[1/a]=\gL'$.\\
Si  $x$ est un objet défini sur $\gA$ notons $\ov{x}$ ce qu'il 
devient après le changement de base 
$\gA\to\gL'$.\\
Le module $\ov{M}$ est libre sur $\gL[\uX]$ et cela implique, par 
unicité (lemme \ref{flempropImProjLib}), vu que $f_1(0)=g_1(0)=1$ et 
que $\gL$ est réduit,
que les $\ov{f_i}$ et $\ov{g_j}$ sont dans $\gL[\uX]$ 
(tenir compte du lemme~\ref{flemUnitRedX}).\\
Cela signifie qu'il existe $N\in \NN$  tel que les $a^Nf_i$ et 
$a^Ng_j$ sont à \coes dans $\gA$. D'après les lemmes~\ref{flemIntSemin1} et~\ref{flemIntSemin2}, ceci implique que $a\in\fa$, 
donc $a=0$ dans
$\gC$.
\end{Proof}

\begin{Proof}{Démonstration du lemme \ref{flemThierry}.}
Une démonstration classique serait la suivante. \\
Supposons $\gC$ non trivial et soit $\fp$ un idéal premier 
minimal.\\
 Puisque $\gC$ est réduit, $\gC_\fp$ est un corps. Donc  $\Im\,P$ 
devient libre sur $\gC_\fp[\uX]$. Cela implique qu'il existe un 
$a\notin\fp$ tel que   $\Im\,P$ 
devient libre sur $\gC[1/a][\uX]$. 
Donc $a=0$ ce qui est une contradiction.\\
On a un lemme d'élimination de l'\idep minimal. Mais la démonstration du 
lemme d'élimination est une démonstration par l'absurde qui utilise un 
\idep minimal! 
\emph{N'est-ce pas une mauvaise plaisanterie?}
Non, car la démonstration du lemme peut être relue en utilisant
l'\idep minimal de manière \emph{purement idéale}, de fa\c{c}on 
dynamique. Voici ce que cela donne.\\
Imaginons que l'anneau $\gC$ soit un corps, \cad que l'on ait déjà 
localisé en un premier minimal. \\
Alors les $f_i$ et $g_j$ 
sont calculés selon un \algo que l'on déduit des \prcos données 
auparavant pour le cas des corps.\\ 
Cet \algo utilise la disjonction \gui{$a$ est nul ou $a$ est 
inversible},
pour les \elts $a$ qui sont produits par l'\algo à partir des
\coes des $m_{i,j}$.  Comme $\gC$ est seulement un anneau réduit,
sans test d'égalité à $0$ ni test d'inversibilité, l'\algo 
pour les corps, si on l'exécute avec $\gC$, doit être remplacé 
par un arbre dans lequel on ouvre deux branches chaque fois
qu'une question \gui{$a$ est-il nul ou inversible?} est posée par 
l'\algoz.\\
Nous voici en face d'un arbre, gigantesque, mais fini. 
Disons que systématiquement on a mis la branche \gui{$a$ inversible}  
à gauche, et la branche \gui{$a=0$ à droite}. 
Regardons ce qui se passe dans la branche d'extrême gauche. \\
On a inversé successivement $a_1$, \ldots , $a_n$ et le module
$P$  est devenu libre sur
$\gC[1/(a_1\cdots a_n)][\uX]$. 

\noindent \emph{Conclusion: dans l'anneau $\gC,$ on a $a_1\cdots 
a_n=0$.}

\noindent Remontons d'un cran. \\
Dans l'anneau $\gC[1/(a_1\cdots a_{n-1})]$, nous savons que $a_n=0$. 
\\
La branche de gauche n'aurait pas dû être ouverte. 
Regardons le calcul dans la branche  $a_n=0$.\\
Suivons à partir de là
la branche d'extrême gauche.\\
On a inversé  $a_1$, \ldots , $a_{n-1}$, puis, disons $b_1,\ldots 
,b_k$ (si $k=0$ convenons que $b_k=a_{n-1}$).\\
Et le module
$P$  est devenu libre sur
$\gC[1/(a_1\cdots a_{n-1} b_1\cdots  b_k)][\uX]$.

\noindent \emph{Conclusion: dans l'anneau $\gC,$ on a $a_1\cdots a_{n-
1} b_1\cdots  b_k=0$.}  Remontons d'un cran: $b_k=0$, la branche de 
gauche n'aurait pas dû être ouverte. 
Regardons le calcul dans la branche  $b_k=0$.\ldots 

\noindent \emph{Et ainsi de suite.} Quand on poursuit le processus 
jusqu'au bout,
on se retrouve à la racine de l'arbre avec le module $P$ libre sur
$\gC[\uX]=\gC[1/1][\uX]$. Donc $1=0$.
\end{Proof}

En utilisant le lemme \ref{flemIntSemin1bis} à la place du lemme \ref{flemIntSemin1} on obtiendra le résultat suivant, plus précis
que le \thoz~\ref{fpropIntSemin}.

\begin{ftheorem} 
\label{fpropIntSeminBis} 
Si $\gA$ est un anneau intègre seminormal et $M$ un module
\pro de rang $1$ sur $\AuX$, il existe $c_{1},\ldots,c_{m}$ dans le corps des fractions de $\gA$ tels que:
\begin{enumerate}
\item $c_{i}^2$ et $c_{i}^3$ sont dans $\gA[(c_{j})_{j<i}]$ pour $i=1,\ldots,m$,
\item $M$ est libre sur $\gA[(c_{j})_{j\leq m}][X]$.
\end{enumerate}
\end{ftheorem}

\newpage
\section*{Annexe: anneaux \zeds réduits} 
\addcontentsline{toc}{section}{Annexe: anneaux \zeds réduits}
\label{fAnnexe}
\markboth{Anneaux seminormaux}{Annexe: anneaux \zeds réduits}
\setcounter{section}{1}
\setcounter{ftheorem}{0}

\def\thesection{\Alph{section}}

Dans cette annexe, nous donnons quelques bases de la théorie des
anneaux dimensionnels réduits, qui sont de bons substituts au corps.

Ceci permet d'obtenir le \tho de Traverso-Swan dans le cas général 
d'un anneau seminormal non 
néces\-sai\-rement intègre.

En outre la démonstration du lemme d'élimination de l'\idep premier minimal 
peut être débarrassée de l'arbre gigantesque
qui pouvait faire peur. 
Celui-ci est caché dans les \idms
et la démonstration semble plus présentable (mais c'est la même).

\medskip 
\rem L'idée de remplacer le corps des fractions de $\gA$ par un 
anneau \zed réduit contenant $\gA$ n'est pas dans \cite{fSwan}: Swan 
utilise
des arguments nettement plus sophitiqués pour ramener le cas 
général, non pas au cas intègre, mais au cas noethérien.
La démonstration dans \cite{fcoq} opère donc des simplifications très 
nettes par rapport à la démonstration classique initiale. En outre le
\tho est nouveau dans le sens qu'il donne un \algo là où 
auparavant, il y avait une affirmation purement abstraite.

\subsection*{A. Quelques faits de base}
\addcontentsline{toc}{subsection}{A. Quelques faits de base}

On dit qu'un anneau est \emph{\zedz} lorsqu'il vérifie l'axiome 
suivant:
\begin{equation} \label{feqZed}
\forall x\in \gA~\exists a\in\gA~\exists d\in
\NN\quad \quad x^{d}=ax^{d+1}
\end{equation}

Dans le cas réduit $d=1$  suffit car $x^d(1-xa)=0$ implique $x(1-
xa)=0.$   

Dans un anneau commutatif $\gC$, deux \elts $a$ et $b$ sont dits 
\emph{quasi inverses}
si l'on a:
              $$a^2b=a,\quad \quad  b^2a=b$$ 
On dit aussi que $b$ est \emph{le} quasi inverse de $a$. On vérifie 
en effet qu'il est unique: si  $a^2b=a=a^2c$,
 $b^2a=b$  et  $c^2a=c$, alors, puisque $ab=a^2b^2$, $ac=a^2c^2$ et   
$a^2(c-b)=a-a=0$, on obtient
$$c-b=a(c^2-b^2)=a(c-b)(c+b)=a^2(c-b)(c^2+b^2)=0$$

Par ailleurs si $x^2y=x$, on vérifie que $xy^2$ est quasi inverse de 
$x$. Ainsi:

\begin{ffact} 
\label{ffactZDRQI} 
Un anneau est \zed réduit \ssi tout \elt admet un quasi inverse. 
\end{ffact}

\smallskip De tels anneaux sont aussi qualifiés d'\emph{absolument 
plats} ou encore de \emph{von Neuman réguliers} (cette dernière 
expression
est surtout utilisée dans le cas non commutatif, avec les 
équations $aba=a$ et $bab=b$).

Les anneaux \zeds réduits peuvent donc être vus comme des anneaux
munis qu'une loi unaire supplémentaire $a\mapsto a\bl$ qui doit 
vérifier les axiomes
\begin{equation} \label{feqAxqiv}
a^2\,a\bl=a, \,\,\,\,a\,(a\bl)^2=a\bl.
\end{equation}

Ceux-ci impliquent notamment, en posant $e_a=aa\bl,$
\begin{equation} \label{feqQIV}
\left.\begin{array}{lll} 
 e_a^2=e_a,& e_aa=a,& e_aa\bl=a\bl,\\ 
(a\bl)\bl=a,& (ab)\bl=a\bl\, b\bl,& 0\bl = 0, \\ 1\bl=1,&  
(x\,\,\mathrm{r\aigu egulier}\,\Leftrightarrow \,x\,x\bl=1),
&(x\,\,\mathrm{\idm}\,\Leftrightarrow \,x=x\bl).
\end{array}
\right\}
\end{equation}

On en déduit facilement:

\begin{ffact} 
\label{ffactZRD} 
Un anneau est \zed réduit \ssi tout \itf est engendré par un 
\idmz.
\end{ffact}

La notion d'anneau \zed réduit est \emph{la bonne généralisation 
équationnelle} de la notion de corps. La notion de corps ne peut pas 
être définie de manière purement équationnelle, mais
un corps n'est rien d'autre qu'un anneau \zed réduit \emph{connexe} 
(\cad avec $0$ et $1$ comme seuls \idmsz).

\begin{flemma} 
\label{flemzedred0} 
Soit $\gA\subseteq\gC$ avec $\gC$ \zed réduit et $a\in\gC$. Notons 
$e_a=aa\bl$.
\begin{enumerate}
\item $e_a$ est l'unique \idm de $\gC$ qui vérifie 
$\gen{a}=\gen{e_a}$. En outre $\Ann_\gC(a)=\Ann_\gC(e_a)=\gen{1-e_a}$
\item $\gC=e_a\gC\oplus(1-e_a)\gC$ avec  
$e_a\gC\simeq\gC[1/e_a]\simeq\aqo{\gC}{1-e_a}$ 
et  $(1-e_a)\gC\simeq\aqo{\gC}{e_a}$ \\
(NB: l'\id $e_a\gC$ n'est pas un sous anneau, mais c'est un anneau 
avec $e_a$ pour \elt neutre multiplicatif).
\item \label{fitem3lemzedred0} Dans $e_a\gC$, $a$ est inversible et dans $\aqo{\gC}{e_a}$, $a$ 
est nul.
\item Si $a\in\gA$, alors $e_a\gA[a\bl]\simeq \gA[1/a]$.
\item \label{fitemlemzedred0} Plus généralement, avec $a,b,c\in\gA$
on a  $(e_ae_be_c)\gA[a\bl,b\bl,c\bl]\simeq \gA[1/(abc)]$.
\item \label{fitem2lemzedred0} Si en outre $abc=0$, alors  
$(e_ae_b)\gA[a\bl,b\bl,c\bl]\simeq \gA[1/(ab)]$.
\end{enumerate}
\end{flemma}
\begin{proof}
Les trois premiers points sont faciles et classiques. \\
Montrons le point \emph{\ref{fitemlemzedred0}}.
Dans l'anneau $\gB=(e_ae_be_c)\gA[a\bl,b\bl,c\bl]$, $abc$  est 
inversible, d'inverse
$a\bl b\bl c\bl$. Donc l'\homo composé  
$$
\psi\,:\,\gA\vers{j}\gA[a\bl,b\bl,c\bl]\vers{x\mapsto e_ae_be_cx} 
\gB$$ se factorise avec un unique $\theta$ comme suit 
$$
\gA\vers{\pi}\gA[1/(abc)]\vers{\theta}\gB
.$$ 
Puisque $\gA\subseteq\gC$, $j$ est injective 
et l'on peut identifier $x\in\gA$ et $j(x)$. 
L'\homo $\theta$ est surjectif parce que
$\theta(1/abc)= a\bl b\bl c\bl=u$ et dans $\gB$, $a\bl=bcu,\, 
b\bl=acu, \,c\bl=abu$. Par ailleurs $\Ker\,\pi=\Ann_\gA(abc)\subseteq 
\Ker\,\psi $ et si
$x\in\Ker\,\psi$,  alors $ e_ae_be_cx=e_{abc}x=0$, donc $abcx=0$. \\
Montrons le point \emph{\ref{fitem2lemzedred0}}. Puisque $abc=0$, 
$0=e_{abc}=e_ae_be_c$ et dans $(e_ae_b)\gA[a\bl,b\bl,c\bl]=\gB_1$ on a
$c\bl=e_ae_bc\bl=e_ae_b(e_cc\bl)=0$ donc 
$\gB_1=(e_ae_b)\gA[a\bl,b\bl]$ et l'on est ramené au point 
précédent.
\end{proof}

Naturellement les deux derniers points sont plus généraux et 
s'étendent avec un nombre fini arbitraire d'\elts de $\gA$.

Une signification possible du lemme est de considérer qu'il
formalise sous une forme un peu plus abstraite ce qui se passe 
lorsque l'on fait des calculs de manière dynamique dans un anneau 
réduit en \gui{faisant comme si} c'était un sous anneau d'un 
corps.
Gr\^ace au point \emph{\ref{fitem3lemzedred0}}, ce calcul dynamique est possible (sous réserve 
de l'existence de $\gC$). Gr\^ace aux derniers points,
on ramène les calculs dynamiques correspondant à la localisation 
en un \idep minimal à des calculs dans des localisés de $\gA$ 
obtenus en inversant un seul \eltz.

\setcounter{section}{2}
\setcounter{ftheorem}{0}
\subsection*{B. Plongement dans un anneau \zed réduit} 
\addcontentsline{toc}{subsection}{B. Plongement dans un anneau \zed 
réduit}

Puisque la notion d'anneau \zed réduit est purement équationnelle, 
l'algèbre universelle nous dit que tout anneau commutatif engendre
un anneau \zed réduit (cela fournit le foncteur adjoint au foncteur 
d'oubli). Nous voulons voir que dans le cas d'un anneau réduit 
$\gA$, l'\homo de $\gA$ vers le  \zed réduit qu'il engendre est 
injectif. Cela nécessite de se fatiguer un petit peu.

\begin{flemma} 
\label{flem1Zedred1} 
Si $\gA\subseteq\gC$ avec $\gC$ \zed réduit, et si nous notons 
$x\bl$ le quasi inverse de $x$, alors
$\gA[(a\bl)_{a\in\gA}]$ est \zed  (c'est donc le sous anneau \zed de 
$\gC$ engendré par $\gA$).\\
Variante: si  $\gA\subseteq\gB$ réduit, et si chaque $a\in\gA$ admet 
un quasi inverse $a\bl$ dans $\gB$, l'anneau $\gA[(a\bl)_{a\in\gA}]$ 
est \zedz.
\end{flemma}
\begin{proof}
On doit démontrer que tout \elt de $\gA[(a\bl)_{a\in\gA}]$ admet un 
quasi inverse. Puisque $(ab)\bl=a\bl b\bl$ tout \elt de 
$\gA[(a\bl)_{a\in\gA}]$ s'écrit sous forme $\sum a_i b_i\bl$ avec $ 
a_i, b_i\in\gA$.
Par ailleurs $a_i b_i\bl =a_i b_i\bl r_i$  avec $r_i=a_i a_i \bl$ 
idempotent. Par ailleurs étant donnés des \idms $r_1,\ldots ,r_k$ 
l'algèbre de Boole qu'ils engendrent contient un \sfio  $e_1,\ldots 
,e_n$ tel que chaque $r_i$ soit la somme des $e_j$ multiples de $r_i$ 
($e_jr_i=e_j$). 
Enfin si $e_1,\ldots,e_n$ est un \sfio  dans $\gC$, si $a_1,\ldots 
,a_n,b_1,\ldots ,b_n\in\gA$, si\quad 
 $c=\sum_{i=1}^na_ib_i\bl e_i$ et $c'=\sum_{i=1}^na_i\bl b_ie_i$, 
alors $c^2c'=c$ et $c'^2c=c'$, donc $c'=c\bl$.
\end{proof}

\begin{flemma} 
\label{flem1Zedred2} 
Soit $\gA$ un anneau réduit et $a\in\gA$. 
Soit $\gB=\aqo{\gA[T]}{aT^2-T,a^2T-a}$ et $\gC=\gB\red$.
Soit $a\bl$ l'image de $T$ dans $\gC$. Alors
\begin{enumerate}
\item $\gC\simeq (\aqo{\gA}{a})\red\times \gA[1/a]$ et l'\homo naturel 
$\gA\to\gC$ est injectif (on identifie $\gA$ à un sous anneau de 
$\gC$).
\item $a\bl$ est quasi inverse de $a$ dans $\gC.$
\item Pour tout \homo $\gA\vers{\varphi}\gA'$ tel que $\varphi(a)$ 
admet un quasi inverse dans $\gB$, il existe un unique \homo 
$\gC\vers{\theta} \gA'$ tel que l'\homo composé 
$\gA\to\gC\vers{\theta} \gA'$ soit égal à $\varphi$.
\end{enumerate}
\end{flemma}

La démonstration n'offre pas de difficulté et est laissée au 
lecteur. Le corolaire suivant est une conséquence de la 
propriété d'unicité forte donnée dans le lemme.

\begin{fcorollary} 
\label{fcorlem1Zedred2} 
Notons $\gA_{\so{a}}$ l'anneau construit au lemme précédent. 
Soient $a_1,\ldots, a_n$ dans $\gA$ alors l'anneau
obtenu en répétant la construction pour chacun des $a_i$,
ne dépend pas, à \iso unique près, de l'ordre dans lequel 
on prend les $a_i$ pour faire la construction.
\end{fcorollary}

Par exemple il existe un unique $\gA$-\homo de 
$((\gA_{\so{a}})_{\so{b}})_{\so{c}}$ dans 
$((\gA_{\so{c}})_{\so{b}})_{\so{a}}$ et c'est un \isoz.
Le lemme \ref{flem1Zedred2} et le corolaire \ref{fcorlem1Zedred2} ont 
pour conséquence immédiate le \tho suivant.

\begin{ftheorem} 
\label{fthAnnexe1} 
Soit $\gA$ un anneau réduit. Considérons l'anneau $\wh{\gA}$ 
obtenu comme limite inductive en itérant la construction du lemme 
\ref{flem1Zedred2}. C'est un anneau \zed réduit et l'\homo naturel 
$\gA\to \wh{\gA}$ est injectif. En outre
cet anneau est \emph{l'anneau \zed réduit engendré par~$\gA$}, au 
sens suivant: pour tout anneau \zed réduit $\gA'$, tout \homo 
$\gA\vers{\varphi }\gA'$ se factorise de manière unique via l'\homo
naturel $\gA\to\wh{\gA}$.
\end{ftheorem}

En bref:
\begin{ftheorem} 
\label{fthRedZed} 
Tout anneau réduit $\gA$ est contenu dans un anneau \zed réduit
\hbox{$\gC=\gA[(a\bl)_{a\in\gA}]$}.
\end{ftheorem}

\setcounter{section}{3}
\setcounter{ftheorem}{0}
\subsection*{C. Anneaux \zeds réduits et corps} 
\addcontentsline{toc}{subsection}{C. Anneaux \zeds réduits et corps}

Nous avons déjà dit que la notion d'anneau \zed réduit est 
\emph{la bonne généralisation équationnelle} de la notion de 
corps. Cela signifie en particulier que toute conséquence 
équationnelle de la théorie des corps est en fait une  
conséquence équationnelle de la théorie des anneaux \zeds 
réduits.

De manière informelle on peut énoncer un principe local-global 
\elr général qui s'avère en pratique assez efficace.

\medskip \noindent 
{\bf Machinerie locale-globale \elrz:     
des corps discrets aux anneaux \zeds réduits.
}\label{fmlge}
{\it La plupart des \algos qui fonctionnent avec les corps discrets  
peuvent \^{e}tre modifiés de manière à fonctionner avec les 
anneaux \zeds réduits, en cassant l'anneau
en deux morceaux chaque fois que l'\algo écrit pour les corps 
discrets utilise le test
\gui{cet \elt est-il nul ou inversible?}. Dans le premier morceau 
l'\elt en question est nul, dans le second il est inversible.}

\medskip 
Nous avons mis \gui{la plupart} plutôt que \gui{tous} dans la mesure 
où 
l'énoncé du résultat de l'\algo pour les corps discrets doit 
être écrit sous une forme où n'apparait pas qu'un corps 
discret est connexe.

\medskip L'application du principe précédent permet d'obtenir le
\tho \ref{fthZedLib} à partir du lemme \ref{flemPicGcd}, dès que l'on 
s'est convaincu que ce dernier donne un \algo pour les corps discrets.

\begin{ftheorem} 
\label{fthZedLib} 
Si $\gC$ est  un anneau \zed réduit, tout \mrc $1$
sur $\gC[\uX]$ est libre.
\end{ftheorem}

Pour la lectrice sceptique, nous donnons quelques détails
dans l'annexe E. 

\setcounter{section}{4}
\setcounter{ftheorem}{0}
\subsection*{D. Théorème de Traverso-Swan. Cas général.} 
\addcontentsline{toc}{subsection}{D. Traverso-Swan: cas général}
\begin{Proof}{Nouvelle \prco du lemme~\ref{flemThierry}.}
D'après les \thos \ref{fthRedZed} et \ref{fthZedLib} il existe un
anneau \zed réduit $\gC=\gA[(a\bl)_{a\in\gA}]\supseteq\gA$ avec  
$\Im\,P$ 
libre sur $\gC[\uX]$. Cette dernière propriété reste vraie pour 
un anneau $\gB\subseteq\gC$  engendré par un nombre fini de quasi 
inverses $a_1\bl,\ldots ,a_r\bl$ d'\elts de $\gA$. Nous écrivons 
$e_i=a_ia_i\bl$ de sorte que $e_i$ est un \idm
tel que $e_ia_i=a_i$ et $e_ia_i\bl=a_i\bl$. \'Ecrivons aussi $e'_i=1-
e_i$. Pour simplifier prenons $r=3$  et il sera clair que l'argument 
est général. On décompose l'anneau $\gB$  en un produit de $2^r$ 
anneaux, ou de manière équivalente en une somme directe de $2^r$ 
\ids
\begin{equation}  \label{feqlemThierry}
\gB=e_1e_2e_3\gB\oplus e_1e_2e'_3\gB\oplus e_1e'_2e_3\gB\oplus 
e_1'e_2e_3\gB\oplus e_1e'_2e'_3\gB\oplus e'_1e_2e'_3\gB\oplus 
e'_1e'_2e_3\gB\oplus e'_1e'_2e'_3\gB  
\end{equation}
D'après le point \emph{\ref{fitemlemzedred0}} du lemme \ref{flemzedred0}
$$e_1e_2e_3\gB\simeq e_1e_2e_3\gA[a_1\bl, a_2\bl, a_3\bl)] 
\simeq\gA[1/(a_1a_2a_3)]$$
Puisque le module  $\Im\,P$  est libre sur $\gB[\uX]$, il l'est sur 
chacune des $2^r$ composantes, et donc en particulier
sur $e_1e_2e_3\gB[\uX]\simeq\gA[1/(a_1a_2a_3)][\uX]$. D'après la 
propriété requise dans le lemme, on obtient \hbox{$a_1a_2a_3=0$}, donc 
$e_1e_2e_3=0$,  $e_1e_2e_3'=e_1e_2$, etc\ldots , 
et la décomposition (\ref{feqlemThierry}) devient
\begin{equation} \label{feqlemThierry2}
\gB= e_1e_2\gB\oplus e_1e_3\gB\oplus e_2e_3\gB\oplus 
e_1e'_2e'_3\gB\oplus e'_1e_2e'_3\gB\oplus e'_1e'_2e_3\gB\oplus 
e'_1e'_2e'_3\gB  
\end{equation}
D'après le point \emph{\ref{fitem2lemzedred0}} du lemme \ref{flemzedred0}  on~a  
$e_1e_2\gB\simeq\gA[1/(a_1a_2)]$ et on en déduit $a_1a_2=0$, donc 
\hbox{$e_1e_2=0$}, \hbox{$e_1e'_2=e_1$}, $e'_1e_2=e_2.$ De même $a_1a_3=0=e_1e_3$, 
$a_2a_3=0=e_2e_3$ et finalement $e_1e'_2e'_3=e_1$, $e'_1e_2e'_3=e_2$, 
$e'_1e'_2e_3=e_3$. 
On obtient une nouvelle décomposition
\begin{equation} \label{feqlemThierry3}
\gB=  e_1\gB\oplus e_2\gB\oplus e_3\gB\oplus e'_1e'_2e'_3\gB  
\end{equation}
Au bout du compte tous les $a_i$ sont nuls et $\gB=\gA=\gA[1/1]$ ce 
qui permet de conclure \hbox{que $1=0$} dans $\gA$.
\end{Proof}

\begin{ftheorem} 
\label{fthTSC}\emph{(Traverso-Swan-Coquand)}\\ 
Si $\gA$ est un anneau seminormal, alors $\Pic\,\gA=\Pic\,\AuX$.\\
Plus précisément pour toute matrice $P\in\AuX^{n\times n}$ \idme 
de rang 1 sur $\AuX$  vérifiant \hbox{$P(0)=\I_{n,1}$}, 
on peut construire un vecteur colonne  $f\in\AuX^{n\times 1}$  et un 
vecteur ligne  $g\in\AuX^{1\times n}$ tels que $P=fg$.
\end{ftheorem}
\begin{proof}
On reprend mutatis mutandis la démonstration donnée dans le cas intègre.
Pour le lecteur sceptique voici ce que cela donne.\\
On utilise la \carn donnée au lemme \ref{flemPicPic1}. Soit 
$P(\uX)=(m_{i,j}(\uX))_{i,j=1,\ldots ,n}$ une matrice \idme de rang 
$1$  avec
$P(0)=\I_{n,1}$. Soit $\gK$ un anneau \zed réduit contenant $\gA$.
Sur $\gK[\uX]$ le module $\Im\,P(\uX)$ est libre et il existe donc 
$f=(f_1(\uX),\ldots ,f_n(\uX))$ et $g=(g_1(\uX),\ldots ,g_n(\uX))$ 
dans
$\gK[\uX]^n$ tels que $m_{i,j}=f_ig_j$ pour tous $i,j$.
En outre puisque $f_1(0)g_1(0)=1$ et puisqu'on peut modifier $f$ et 
$g$ en les multipliant par une unité, on peut supposer que 
$f_1(0)=g_1(0)=1$. Alors puisque $f_1g_j=m_{1,j}$ et vu le \tho de 
Kronecker, les \coes des $g_j$ sont entiers sur l'anneau engendré 
par les \coes des $m_{1,j}$. De même les \coes des $f_i$ sont 
entiers sur l'anneau engendré par les \coes des $m_{i,1}$.\\
On appelle $\gB$ le sous anneau de $\gK$ engendré par $\gA$ et par 
les \coes des $f_i$ et des $g_j$. Alors~$\gB$ est une extension finie 
de 
$\gA$ (i.e., $\gB$ est un \Amo \tfz). Notre but est de montrer que
$\gA=\gB$.
On appelle  $\fa$ le conducteur de $\gA$ dans $\gB$. Notre but est 
maintenant de montrer 
$\fa=\gen{1}$, \cad que $\gA\sur{\fa}$ est trivial. \\
D'après le lemme \ref{flemIntSemin1} l'idéal $\fa$ est un idéal radical de 
$\gB$. Le lemme \ref{flemIntSemin2} s'applique avec $\gA\subseteq\gB$. 
On~a $\gA\sur{\fa}=\gC\subseteq \gB\sur{\fa}=\gC'$ réduits, 
et $f_ig_j=m_{i,j}$ au 
niveau $\gB\sur{\fa}$. 
Pour montrer que  $\gC$ est 
trivial, il suffit de montrer que  $\gC$ vérifie, avec la matrice 
$P$ mod $\fa$, les hypothèses du lemme \ref{flemThierry}.\\
Considérons donc $a\in\gA$ tel que  $\Im\,P$ 
soit libre sur $\gC[1/a][\uX]$.\\
Notons $\gC[1/a]=\gL\subseteq  \gC'[1/a]=\gL'$.\\
Si  $x$ est un objet défini sur $\gA$ notons $\ov{x}$ ce qu'il 
devient après le changement de base 
$\gA\to\gL'$.\\
Le module $\ov{M}$ est libre sur $\gL[\uX]$ et cela implique, par 
unicité (lemme \ref{flempropImProjLib}), vu que \hbox{$f_1(0)=g_1(0)=1$} et 
que $\gL$ est réduit,
que les $\ov{f_i}$ et $\ov{g_j}$ sont dans $\gL[\uX]$ 
(tenir compte du lemme~\ref{flemUnitRedX}).\\
Cela signifie qu'il existe $N\in \NN$  tel que les $a^Nf_i$ et 
$a^Ng_j$ sont à \coes dans $\gA$. D'après le lemme~\ref{flemIntSemin2}, ceci implique que $a\in\fa$, donc $a=0$ dans
$\gC$.
\end{proof}

En utilisant le lemme \ref{flemIntSemin1bis} à la place du lemme \ref{flemIntSemin1} on obtiendra le résultat suivant, plus précis
que le \thoz~\ref{fpropIntSemin}.

\begin{ftheorem} 
\label{fthTSCBis} 
Si $\gA$ est un anneau contenu dans
un anneau \zed réduit $\gB$ et $M$ un module
\pro de rang $1$ sur $\AuX$, il existe $c_{1},\ldots,c_{m}$ dans  $\gB$ tels que:
\begin{enumerate}
\item $c_{i}^2$ et $c_{i}^3$ sont dans $\gA[(c_{j})_{j<i}]$ pour $i=1,\ldots,m$,
\item $M$ est libre sur $\gA[(c_{j})_{j\leq m}][X]$.
\end{enumerate}
\end{ftheorem}

\setcounter{section}{5}
\setcounter{ftheorem}{0}
\subsection*{E. Anneaux à pgcd} 
\addcontentsline{toc}{subsection}{E. Anneaux à pgcd}

Nous terminons avec une démonstration détaillée du \tho \ref{fthZedLib},
pour la lectrice sceptique quant à la validité de la machinerie 
locale globale élémentaire page~\pageref{fmlge}.

\begin{fdefinition} 
\label{fdefqi} Un anneau $\gA$ est dit \emph{\qiz} lorsque tout \elt  
admet pour annulateur un (idéal principal engendré par un) \idmz.
Pour $a\in\gA$, on note alors $e_a$  l'\idm tel que $\Ann(a)=\gen{1-
e_a}$, de sorte que $a$  est régulier dans $\gA[1/e_a]$ et nul dans
 $\gA[1/(1-e_a)]$.
\end{fdefinition}

Un anneau intègre n'est autre qu'un anneau \qi connexe.

\begin{flemma} \label{flemQI}
On consid\`{e}re des \elts $x_1,\dots,x_n$ d'un anneau commutatif. 
Si l'on a
$\Ann(x_i) = \gen{r_i}$ o\`{u} $r_i$ est un idempotent pour $1 \leq i
\leq n$, soit $s_i$ tel que $s_i + r_i = 1$, et
posons $t_1=s_1$, $t_2=r_1s_2$, $t_3=r_1r_2s_3 ,\dots$,
$t_{n+1}=r_1r_2\cdots r_n$.
Alors $t_1,\dots,t_{n+1}$ est un \sfio
et l'\elt $x=x_1+t_2x_2+\cdots +t_nx_n$ v\'{e}rifie
\[\Ann(x_1,\dots,x_n) = \Ann(x) = \gen{t_{n+1}}.\]
\end{flemma}

\begin{fcorollary} 
\label{fcorlemQI2} 
Sur un anneau \qi $\gA$ soit $P$ une matrice carrée telle que 
$\Tr(P)$ est régulier. Alors il existe une matrice $J$ de même 
format telle que $J^2=J$ et $JPJ=JPJ^{-1}$ admet un \coe régulier en 
position $(1,1)$
\end{fcorollary}
\begin{proof}
On applique le lemme précédent avec les \elts $x_i=m_{i,i}$ de la 
diagonale de la matrice. On~a $t_{n+1}=0$ car $t_{n+1}\Tr(P)=0$.
Donc $(t_1,\ldots ,t_n)$ est un \sfioz.
Notons~$J_k$ la matrice de permutation qui échange les vecteurs
numéros 1 et $k$ de la base canonique. On pose 
$J=t_1\In+t_2J_2+\cdots t_nJ_n$. On a $J^2=J$ et  le \coe en position 
$(1,1)$ dans~$JPJ$ est égal~à 
\[x=t_1x_1+t_2x_2+\cdots 
+t_nx_n=x_1+t_2x_2+\cdots +t_nx_n,\] 
donc il est régulier.
\end{proof}

Un anneau \zed réduit est \qiz. Inversement
si $\gA$ est \qiz, l'anneau total des fractions de $\gA$, que nous 
notons $\Frac(\gA),$ est un anneau \zed réduit: pour tout~$a$, l'\elt
$\wi{a}=(1-e_a)+a$ est régulier et $a/\wi{a}=a\bl$ est quasi inverse 
de $a$ dans $\Frac(\gA)$. En outre, pour tout $a\in\gA$, 
$\gA[1/a]$ est un anneau \qi et $\Frac(\gA[1/a])$ s'identifie
à $e_a\Frac(\gA)\simeq \Frac(\gA)[1/a]$. 

Enfin, si $\gA$ est \qiz,
il en va de même pour $\gA[X]$, l'annulateur d'un \pol $f$ étant
engendré par l'\idm produit des annulateurs des \coesz.

\smallskip Dans un anneau \qi si $a$ divise $b$ et $b$ divise $a$,
on a $e_a=e_b$ et $ua=b$  avec un \elt $u$ inversible.
Ceci permet de développer pour les anneaux \qis une théorie du 
pgcd tout à fait analogue à celle des anneaux intègres.

\begin{fdefinition} 
\label{fdefMongcd}  Un monoïde commutatif régulier est appelé 
un  \emph{monoïde à pgcd} lorsque deux \elts arbitraires 
admettent toujours un plus grand commun diviseur. 
Si $g$ est un pgcd pour $a$ et $b$ on écrit  $g=\pgcd(a,b)$
(en fait un pgcd n'est défini qu'à un inversible 
près).
\end{fdefinition}

\begin{flemma} 
\label{flemQigcd} Soit  $\gA$ un anneau \qiz. \Propeq 
\begin{enumerate}
\item Le monoïde des \elts réguliers est un monoïde à 
pgcd.
\item Pour chaque \idm $e$, les \elts réguliers de $\gA[1/e]$ 
forment un monoïde à pgcd.
\item Deux \elts arbitraires admettent toujours un plus grand commun 
diviseur.
\end{enumerate}
Dans ce cas on dit que $\gA$ est un anneau  \emph{\qi à pgcd}, et le pgcd de 
deux \elts $a$ et $b$, bien défini à un inversible près est 
noté
$\pgcd(a,b)$.
\end{flemma}
Par exemple, pour \emph{1.} implique \emph{2.}, on considère, pour $a\in e\gA$ 
avec $a$ régulier dans $\gA[1/e]$, l'\elt $\wi{a}=(1-e_a)+a$ 
régulier dans $\gA$. Si $g$ est le pgcd de  $\wi{a}$ et  $\wi{c}$ 
dans $\gA$, le même \elt $g$, vu dans $\gA[1/e]$, est le pgcd de $a$ 
et $c$.

Sur un anneau \qi à pgcd, soit un \pol $f(X)=\sum_{k=0}^nf_kX^k$, on 
note $\rG(f)$ le pgcd (défini à une unité près) des \coes de 
$f$. Si $\rG(f)=1$ on dira que $f$ est primitif{\footnote{~Ceci entre 
en conflit avec une autre tradition, qui dit que $f$ est primitif 
lorsque l'idéal des \coes de $f$ est égal à $\gen{1}$.}}.

\medskip 
Un anneau \qi à pgcd connexe est un anneau à pgcd usuel.

Il est clair qu'un groupe est un monoïde à pgcd
ce qui implique qu'un anneau \zed réduit est un anneau \qi à pgcd.

Il nous faut vérifier que les arguments dans la démonstration du lemme
\ref{flemPicGcd} s'appliquent aussi bien aux anneaux \qis à pgcd 
qu'aux anneaux à pgcd usuels. En particulier, si $\gA$ est un anneau 
\qi à pgcd, il en va de même
pour $\gA[X]$. Ainsi on obtiendra que pour tout anneau \zed réduit 
$\gA$, l'anneau $\AuX$ est un anneau \qi à pgcd et donc tout
\mrc 1 sur $\AuX$ est libre.

Regardons tout d'abord ce qui concerne le premier argument dans la 
démonstration:\\
\emph{Soit $P=(m_{i,j})$ une matrice \idme de rang 1. 
Puisque $\sum_i m_{i,i}=1$ on peut supposer que~$m_{1,1}$ est 
régulier.}\\
Il est clair que notre corolaire \ref{fcorlemQI2} fait l'affaire.

Pour le reste nous nous reportons à \gui{la bible} \cite{fMRR}, livre 
dans lequel les démonstrations sont en général réduites à 
leur forme algorithmique la plus simple.

\begin{flemma} 
\label{flemGCD1} \emph{(cf. Theorem 1.1 page 108 dans \cite{fMRR})}\\
Soient $a,b,c$ dans un anneau \qi à pgcd. Alors
\begin{enumerate}
\item $\pgcd(\pgcd(a,b),c)=\pgcd(a,\pgcd(b,c))$.
\item $c\cdot \pgcd(a,b)=\pgcd(ca,cb)$.
\item Si $x=\pgcd(a,b)$, alors $\pgcd(a,bc)=\pgcd(a,xc)$.
\item Si $a|bc$ et $\pgcd(a,b)=e_b$ alors $a|e_bc$. 
\end{enumerate}
\end{flemma}

Si l'un des trois \elts $a,b,c$ est nul, les affirmations sont évidentes. 
Dans le cas général, on considère un \sfio engendré par $e_a$, $e_b$ et $e_c$. 
Si~$r_i$ est un \elt de ce système, alors dans $\gA[1/r_i]$ chacun des 3 \elts $a,b,c$ est nul ou régulier. 
La démonstration donnée dans \cite{fMRR} pour les monoïdes commutatifs réguliers s'applique dans la composante où les trois \elts sont réguliers. 

\smallskip Une conséquence du point \emph{2} dans le lemme ci-dessus est
que pour un anneau \qi à pgcd, un \pol primitif est un \elt 
régulier de $\gA[X]$. 

\begin{flemma} 
\label{flemPrimFact} \emph{(en suivant le lemme 4.2 page 123 dans \cite{fMRR})}~\\
Soit $\gA$  un anneau \qi à pgcd et $\gK=\Frac(\gA)$. Si  
$f\in\gK[X]$ nous pouvons trouver un \pol primitif $g\in\gA[X]$ et
$c\in\gK$ tel que $f=cg$. Pour une autre décomposition $f=c'g'$ 
du même type, il existe $u\in\gA^\times$ tel que $c=uc'$.
\end{flemma}

%
\begin{proof}
Si $f=0$ on prend $g=1$ et $c=0$. 
Si $\rG(f)$ est régulier la démonstration dans \cite{fMRR} fonctionne,
en rempla\c{c}ant \gui{$a\neq 0$} par \gui{$a$ est régulier}.
Il suffit donc de casser l'anneau en deux morceaux au moyen
de l'\idm $e_{\rG(f)}$. 
\end{proof}
%

\begin{flemma} 
\label{flemGauss} \emph{(lemme de Gauss, lemme 4.3 page 123 dans 
\cite{fMRR})}\\
Soit $\gA$  un anneau \qi à pgcd et $f,g\in\gA[X]$ alors 
$\rG(f)\rG(g)=\rG(fg)$.
\end{flemma}

On considère le \sfio $(r_i)$ engendré par les $e_c$ pour tous les 
\coes $c$  de $f$ et $g$. Dans chacun des anneaux $\gA[1/r_i]$ les 
\pols $f$ et $g$ ont un degré bien 
déterminé{\footnote{~Précisément un \pol a un degré bien 
déterminé lorsqu'on connait un entier $q\geq 0$ tel que le 
\coe de degré $q$ est à la fois dominant et régulier, sans avoir 
à supposer que l'anneau est trivial ou non.}}. L'élégante démonstration 
par récurrence sur $n+m=\deg(f)+\deg(g)$ donnée dans \cite{fMRR} 
s'applique:

On raisonne par induction sur $m+n$.
Par distributivité (point \emph{2} du lemme \ref{flemGCD1}) et vu le lemme~\ref{flemPrimFact}, on se ramène au cas où $\rG(f)=\rG(g)=1$. On 
pose $c=\rG(fg)$ et $d=\pgcd(f_n,c)$. Alors $d$ divise
$(f-f_nX^n)\,g$. Si $f=f_nX^n$ le résultat est clair. 

Sinon, par \hdr
$d$ divise $\rG(f-f_nX^n)\,\rG(g)=\rG(f-f_nX^n)$, donc $d$ divise~$f$, 
et~$d=1$. Ainsi $\pgcd(f_n,c)=1$. De m\^{e}me  $\pgcd(g_m,c)=1$ et 
puisque~$c$ divise $f_ng_m$, $c=1$. 

\begin{fcorollary} 
\label{fcorlemGauss} \emph{(corolaire 4.4 page 123 dans \cite{fMRR})}\\
Soit $\gA$  un anneau \qi à pgcd,  $f,g\in\gA[X]$ et 
$\gK=\Frac(\gA)$. Alors $f$ divise $g$ dans $\gA[X]$ \ssi
 $f$ divise $g$ dans $\gK[X]$ et $\rG(f)$ divise $\rG(g)$.
\end{fcorollary}

\begin{ftheorem} 
\label{fthGauss} \emph{(théorème 4.6 page 124 dans \cite{fMRR})}\\
Si $\gA$  est un anneau \qi à pgcd, il en va de même pour  
$\gA[X]$.
\end{ftheorem}

Les démonstrations dans \cite{fMRR} s'appliquent.

\medskip En fait, tout ceci est parfaitement automatique.
Les démonstrations dans \cite{fMRR}, qui sont aussi des algorithmes, sont 
basées sur la disjonction \gui{$x=0$ 
ou $x$ régulier} valable pour tout $x$ dans un anneau intègre. 
Quand on passe aux
anneaux \qisz, il suffit de réaliser la disjonction en cassant
l'anneau en 2 morceaux, au moyen de l'\idm $e_x$, chaque fois 
que la démonstration (i.e., l'algorithme) trouve un $x$  qu'il faut traiter.

\newpage

\rdb

\markboth{References}{Références}



\endgroup
\end{document}